\documentclass[10pt,reqno]{amsart}

\usepackage{color}

\usepackage[a4paper,left=24.2mm,right=24.2mm,top=25mm,
bottom=25mm,marginpar=25mm]{geometry}

\usepackage{amsmath, amssymb ,amsthm, amsfonts, amsgen,
mathrsfs}

\usepackage{color}

\usepackage{amsmath}
\usepackage{amssymb}
\usepackage{amsthm}
\usepackage[latin1]{inputenc}
\usepackage{eurosym}
\usepackage{graphicx}
\usepackage{epsfig}

\usepackage{comment,graphicx,esint}

\usepackage[displaymath,mathlines]{lineno}

\allowdisplaybreaks

\usepackage{hyperref}

\usepackage{ifthen}
\usepackage{enumitem}

\providecommand{\varitem}{} 


\providecommand{\varitem}{} 


\providecommand{\varitem}{} 


\makeindex

\usepackage{fancyhdr}
\usepackage[us,24hr]{datetime}

\usepackage{scrtime}
\usepackage{comment}

\usepackage{tikz}
\usetikzlibrary{arrows,shapes,trees}

\usepackage[english]{babel}

\usepackage{color}
\usepackage{amsfonts}

\usepackage{lmodern}

 \definecolor{Kblue}{rgb}{0,0.651,0.667}

 \definecolor{Korange}{rgb}{0.945,0.561,0}

 \definecolor{Kgreen}{rgb}{0.804,0.808,0}

 \definecolor{Kyellow}{rgb}{0.941,0.71,0}
 
 \definecolor{Mgray}{rgb}{0,0,0}

\newcommand{\ce}{{\mathcal{C}W}}
\newcommand{\qe}{{\mathcal{Q}W}}
\newcommand{\qce}{{\mathcal{QC}W}}

\newcommand{\tr}{\operatorname{tr}}

\renewcommand{\div}{\operatorname{div}}

\newcommand{\Nn}{{\mathbb{N}}}

\newcommand{\calG}{{\mathcal{G}}}

\newcommand{\calB}{{\mathcal{B}}}

\newcommand{\calX}{{\mathcal{X}}}

\def\dist{{\rm dist}}

\newcommand{\epsi}{\varepsilon}
\def\d{{\rm d}}
\def\dx{{\rm d}x}

\def\dt{{\rm d}t}
\def\leq{\leqslant}
\def\geq{\geqslant}

\newcommand{\grad}{\nabla}

\newcommand{\ffi}{\varphi} 
\newcommand{\aev}{{a.e.}}

\newcommand{\ellp}{{\ell_+}}

\let\weakly\rightharpoonup
\def\weaklystar{\buildrel{\hskip-.6mm\star}\over\weakly}

\newcommand{\II}{\mathbb{I}}
\newcommand{\RR}{\mathbb{R}}
\newcommand{\NN}{\mathbb{N}}

\newcommand{\Nb}{{\mathbb{N}}}
\newcommand{\R}{{\mathbb{R}}}

\newcommand{\A}{{\mathcal{A}}}

\newcommand{\HH}{{\mathcal{H}}}
\newcommand{\M}{{\mathcal{M}}}
\newcommand{\beq}{\begin{equation}}
\newcommand{\eeq}{\end{equation}}

\numberwithin{equation}{section}

\newtheoremstyle{thmlemcorr}{10pt}{10pt}{\itshape}{}{\bfseries}{.}{10pt}{{\thmname{#1}\thmnumber{
#2}\thmnote{ (#3)}}}
\newtheoremstyle{thmlemcorr*}{10pt}{10pt}{\itshape}{}{\bfseries}{.}\newline{{\thmname{#1}\thmnumber{
#2}\thmnote{ (#3)}}}
\newtheoremstyle{defi}{10pt}{10pt}{\itshape}{}{\bfseries}{.}{10pt}{{\thmname{#1}\thmnumber{
#2}\thmnote{ (#3)}}}
\newtheoremstyle{remexample}{10pt}{10pt}{}{}{\bfseries}{.}{10pt}{{\thmname{#1}\thmnumber{
#2}\thmnote{ (#3)}}}
\newtheoremstyle{ass}{10pt}{10pt}{}{}{\bfseries}{.}{10pt}{{\thmname{#1}\thmnumber{
A#2}\thmnote{ (#3)}}}

\theoremstyle{thmlemcorr}
\newtheorem{theorem}{Theorem}
\numberwithin{theorem}{section}
\newtheorem{lemma}[theorem]{Lemma}

\newtheorem{proposition}[theorem]{Proposition}

\theoremstyle{thmlemcorr*}
\newtheorem{theorem*}{Theorem}
\newtheorem{lemma*}[theorem]{Lemma}
\newtheorem{corollary*}[theorem]{Corollary}
\newtheorem{proposition*}[theorem]{Proposition}
\newtheorem{problem*}[theorem]{Problem}
\newtheorem{conjecture*}[theorem]{Conjecture}

\theoremstyle{defi}

\theoremstyle{remexample}
\newtheorem{remark}[theorem]{Remark}

\theoremstyle{ass}

\usepackage{color}

  \usepackage{frcursive}
   \newcommand{\bcal}{\text{\cursive{\begin{scriptsize}\,b\,\end{scriptsize}}}}
   
  \newcommand{\Bcal}{\text{\cursive{\begin{scriptsize}\,G\,\end{scriptsize}}}}

\def\weaklystar{\buildrel{\hskip-.6mm\star}\over\weakly}

\allowdisplaybreaks

\usepackage[sort, nocompress, space]{cite}

\address[R.~Ferreira]{King Abdullah University
of Science
and Technology (KAUST), CEMSE Division,  Thuwal 23955-6900,
Saudi Arabia. \bf{E-mail:} rita.ferreira@kaust.edu.sa}

\address[E.~Zappale]{Dipartimento di Ingegneria Industriale, Universit\'a degli Studi di Salerno, via Giovanni Paolo II, 132 Fisciano (SA), Italy. \bf{E-mail:} ezappale@unisa.it}

\usepackage[normalem]{ulem}

\begin{document}

\title[Bending moments in thin multi-structures in
 nonlinear elasticity]{Bending-torsion moments in thin multi-structures\\ in the context of nonlinear elasticity}
\author[\hfill Rita Ferreira and Elvira Zappale 
\hfill]{Rita Ferreira and Elvira Zappale}

\begin{abstract}
Here,  we  address a dimension-reduction problem
in the context of nonlinear elasticity where the applied
external
surface
forces induce  bending-torsion  moments. The underlying body
is a multi-structure in \(\mathbb{R}^3\) consisting of a
thin tube-shaped domain placed upon a thin  plate-shaped
 domain. The problem involves two small parameters, the
 radius of the cross-section of the tube-shaped domain
 and the thickness of the plate-shaped
 domain. We characterize the different limit models, including the limit
junction condition, in the membrane-string regime according to the ratio between  these two parameters
 as they converge to zero. 
 


\textit{Date:}\ \today
%

\end{abstract}

\maketitle
%

\section{Introduction}

Thin structures are three-dimensional
structures having one or two of its
dimensions  much smaller than
the others. Because of this geometric
feature, thin structures are often seen
as  two- or one-dimensional objects.
Common examples  are the board of a
bridge,
the sail of a boat, the wing of an airplane, shelves,
 domes, antennae, pillars,
bars, cables, to mention  but a few.

In the context of the Theory of Elasticity
(see, e.g., \cite{Cia88}), a key question
 is the prediction of
the 
behavior of a thin elastic  structure
when subjected
to a given system of applied forces.
Although valid, three-dimensional models
are discarded  in favor of lower-dimensional
ones because lower-dimensional
models have a simpler 
structure. This simpler structure  allows
for  richer theoretical results and
easier numerical  treatments. On the
other hand, it only makes sense to use
a lower-dimensional model if it is a
\textit{good} model; that is, a  model
whose response is sufficiently close
to the response of the three-dimensional
model. In other words, a central question is how to rigorously
justify a lower-dimensional model starting
from the three-dimensional one.
This question is at the core of dimension-reduction
problems.
  
The rigorous justification of lower-dimensional
models was first obtained through the
 method of asymptotic expansions.  This
method was highly successful within
linear elasticity by enabling numerous
convergence results. However, in nonlinear
elasticity, the
 method of asymptotic expansions provided
few convergence results. We refer to
the books \cite{Cia97, TrVi96} for a historical
overview and a thorough description
of the use of  asymptotic expansions
to  derive one- and two-dimensional models
for thin elastic  structures.

The seminal work \cite{ABP91} gave rise
to a new   approach  to study dimension-reduction
problems
 based  on
\(\Gamma\)-convergence technics. The
notion of \(\Gamma\)-convergence was 
  introduced by De Giorgi in the 70's,
and we refer to the book \cite{DM93} for a comprehensive introduction to
this notion.
The use of \(\Gamma\)-convergence has
proved successful both for linear and
nonlinear elasticity dimension-reduction
problems. In particular, it provided
the unique known results of convergence
for the nonlinear case. Among a vast list, we refer, for instance, to
\cite{ABP91,BFM03,FMP12,FMP13,FJM06,LDR95,LDR00,MoMu03,Sca09} and
to the references therein for the rigorous
justification of nonlinear lower-dimensional
theories (such as membranes, plates,
shells,
rods, beams, strings) through \(\Gamma\)-convergence.  

In this paper, we consider a more complex type
of thin structure, commonly called 
a thin multi-structure or multi-domain.
A thin multi-structure is a  structure 
made of two or more different thin structures.
Two simple but important examples are
bridges (where, for instance, cables
are connected
to the  board of the bridge) and airplanes
(where, for instance, the wings are
attached
to the  body of the airplane). In
such structures,
the behavior/interaction at the junction
between
their different   thin components
plays a crucial role and, from the mathematical
viewpoint,  adds  nontrivial difficulties.

There exists a somewhat extensive literature
on dimension-reduction problems involving
thin multi-structures.
A substantial part of this literature
pertains to the context of linear
elasticity; see, for instance, \cite{Cia90, GMMMS07, LeD91}
and the references therein. Concerning
the case of non-linear
elasticity and the case dealing with multi-structures
in  contexts other than elasticity, we refer  to \cite{BlGr13, GaZa07,  GGLM02, GaZa06, Gru93, TrVi96}
and  \cite{BCN17, GGLM02b, GaSi11},
respectively, and to the references
therein.

Here, using \(\Gamma\)-convergence,
we derive   lower-dimensional models
with bending-torsion moments for multi-structures.
Our starting point is
 the standard nonlinear three-dimensional
equilibrium  problem for a three-dimensional
thin multi-structure that consists of
a thin tube-shaped structure placed
upon a thin plate-shaped structure.
   One of the main features of our setting
is a non-standard scaling
of the applied forces, which induce
bending-torsion moments in the limit model.
These forces were introduced in \cite{BFM03}
(also see \cite{BoFo04}) concerning the membrane case, then 
adapted to the string case in \cite{RiMSc} and
also to the membrane case in the  \(BV\) setting in \cite{BZZ08} and in the Orlicz--Sobolev  setting in \cite{LaNg14, LaNg16}.
We also refer  the work \cite{CMMO17} for a related study in the context of structured deformations.   Interestingly,
besides similar bending effects as those
derived in \cite{BZZ08, BoFo04, BFM03, CMMO17, RiMSc, LaNg14, LaNg16}, we observe
here a fine interaction between the
non-standard forces
 and the junction of the multi-structure.
Moreover, we assume that our structure satisfies a deformation condition on a suitable part of its boundary that goes beyond the clamped case, which is the case commonly assumed in the literature. Further, we characterize the
limit problem according to the asymptotic
behavior of the ratio between the area
of the cross-section of the thin tube-shaped
part of the structure and the  thickness
of the thin plate-shaped part, providing new results in the nonlinear setting. 
To precisely state our main results,
we describe
next  the set-up of our problem. 

In what follows, we  use Greek indices
to distinguish the first two components
of a tensor; for instance, \((x_\alpha)\)
and \((x_\alpha,x_3)\) stand for \((x_1,x_2)\)
and  \((x_1,x_2,x_3)\), respectively.
We represent by \(\R^{m\times n}\) the
vector space of \(m\times n\) real-valued
matrices, endowed with the norm \(\vert
M\vert:= \sqrt{\tr(M^T M)}\) associated
with the inner product \(M: M':= \tr(M^T
M')\) for \(M,\, M'\in \R^{m\times n}\).
If \(M\in \R^{3\times3}\), then \(M_\alpha\)
represents the \(3\times 2\) matrix
obtained from \(M\) by removing its
last column, which in turn is denoted
by \(M_3\); conversely, if \(M_\alpha
\in \R^{3\times 2}\) and \(M_3 \in \R^3\),
then \(M:=(M_\alpha|M_3)\) represents
the \(3\times 3\) matrix whose first
two columns are those of \(M_\alpha\)
and the third one is \(M_3\).
Moreover, we assume
that \(\epsi\) is a parameter taking
values on a sequence of positive numbers
convergent to zero and containing the
number one; we write ``for each \(\epsi>0\)"
in place of ``for each term in the sequence
where \(\epsi\) takes values".

Let \(\omega^a \) and \(\omega^b\) be two   
bounded domains (open and connected sets with Lipschitz boundaries)
in \(\RR^2\) containing the origin and such that \(\omega^a \subset \subset\omega^b\),
  let \(L>0\), and let \((r_\epsi)_{\epsi>0}\) and 
\((h_\epsi)_{\epsi>0}\) be two sequences of positive
numbers convergent to 0. For each \(\epsi>0\),
let \(\Omega_\varepsilon:={\rm int}(\overline\Omega^a_\varepsilon
\cup \overline\Omega^b_\varepsilon)\) be the union of two vertical cylinders, where 
$\Omega^a_\varepsilon:=r_\varepsilon\omega^a  \times(0,L)$
has small cross-section and fixed height and $\Omega^b_\varepsilon:=\omega^b\times
(-h_\varepsilon,0)$ has fixed cross-section and small
height; note that \(r_\epsi \omega^a\times
\{0\}\) represents the interface between
the two cylinders. We further define  $\Gamma^a_\varepsilon:= 
r_\varepsilon\omega^a\times
\{L\}$, 
$\Gamma^b_\varepsilon:=\partial\omega^b\times (-h_\varepsilon,
0)$,
$\Gamma_\varepsilon:=\Gamma^a_\varepsilon\cup\Gamma^b_\varepsilon$,
$S^a_\varepsilon:=(r_\varepsilon\partial\omega^a)\times
(0,L)$, $S^{b,-}_\varepsilon:=\omega^b\times
\{-h_\varepsilon\}$, 
$S^{b,+}_\varepsilon:=(\omega^b\setminus (r_\varepsilon\overline\omega^a))\times
\{0\}$,
$S_\varepsilon:=S^a_\varepsilon \cup S^{b,+}_\varepsilon\cup
S^{b,-}_\varepsilon$ (see Fig.1). We observe that the superscripts ``a" and
``b"
stand
for ``above" and  ``below", respectively; moreover, we  omit the index \(\epsi\)
whenever \(\epsi=1\) and, without loss of generality,
we suppose that \(r_1 = h_1=1\).




\begin{center}
\begin{tikzpicture}[scale=.75]

%
        \draw[Mgray,
           thick] (0.28,0) -- (0.28,-5.0);
         \draw[Mgray,
           thick] (1.72,0)-- (1.72,-5.0);

       
        \draw[Mgray,
           thick]
        (1,0)  ellipse (.72 and 0.25);
        
                \draw [   thin,->]
(1.5,-1)                 arc (-110:3:10pt);
                
                \draw(2.1,-.4) node
{$\textcolor{red}{S^a_\varepsilon}$};
                
                \draw [   thin,<-]
(.3,0.4)
                arc (110:3:10pt);
                
                \draw(-.0,0.4) node
{$\textcolor{red}{\Gamma^a_\varepsilon}$};
        
        \draw[gray,
            thick,fill=Kyellow!40]
        (1,-5)  ellipse (.72 and 0.25);
        
         \draw(1.02,-5) node 
         {\fontsize{4}{4}\selectfont   $\textcolor{red}
         {r_\varepsilon\omega^a\times\{0\}}$
        };

        \draw[densely dashed, <->] (.1,0)
--
        (.1,-5);
        
        \draw(-.2,-2.5) node      
        {$\textcolor{red}{L}$};

%
        \draw[Mgray,
           thick] (-5,-6) -- (-3,-4)
         -- (0.3,-4);
         
         \draw[Mgray,
        densely dashed]  (0.3,-4)  --
 (1.7,-4);

        \draw[Mgray,
           thick] (1.7,-4) -- (5,-4)
         -- (4,-6) -- (-5,-6);

        \draw[Mgray,
        densely dashed]  (-3,-4)
         -- (-3,-4.4) -- (5,-4.4);
         
         \draw[Mgray,
            thick]  (5,-4)
         -- (5,-4.4)  -- (4,-6.4);
         
         \draw[Mgray,
            thick]  (4,-6)
         -- (4,-6.4) -- (-5,-6.4);
         
         \draw[Mgray,
            thick]  (-5,-6)
         -- (-5,-6.4);
         
        \draw[Mgray,
        densely dashed]  (-3,-4.4) 
--  (-5,-6.4);

                \draw [   thin,<-]
(-3.7,-4.2)
                arc (110:3:10pt);
                
                \draw(-4.2,-4.25) node
                              {
$\textcolor{blue}{S^{b,+}_\varepsilon}$};

                \draw [   thin,->]
(4.8,-4.55)
                arc (-110:3:10pt);
                
                \draw(5.4,-3.95) node
                              {$\textcolor{blue}{\Gamma^b_\varepsilon}$};
                
                \draw [thin,->]
(3.6,-6.3)
                arc (-140:4:17pt);
                
                \draw(5,-5.6) node
                {$\textcolor{blue}{S^{b,-}_\varepsilon}$};
                
        \draw [densely dashed,<->] (-5.2,-6.42)
                --  (-5.2,-6.01);
        \draw(-5.57,-6.2) node
                {${\textcolor{blue}{h_\varepsilon}}$};
       
        \draw[   thick, ->] (1,0)
--         (1,1);
        
        \draw[densely dashed, -] (1,0)
--
        (1,-5);
 
        \draw(1.45,.8) node { $x_3$};
        
\draw(0,-7) node {Figure 1. $\overline\Omega^\varepsilon$
- reference configuration};        
              \end{tikzpicture}

\end{center}

%
%
%
We assume that \(\overline \Omega_\epsi\) is the reference
configuration of a three-dimensional body made of a hyperelastic
and homogeneous material, whose stored energy is  a Borel
function \(W:\RR^{3\times 3} \to \RR\) satisfying the
following \(p\)-growth conditions for some  \(p\in(1,\infty)\):
there exists a positive constant, \(C\), such that
for all \(\xi\in \R^{3\times 3}\), we have \vspace{-1.mm}
\begin{equation*}
\frac{1}{C}|\xi|^p  - C \leq W(\xi) \leq C(1+ |\xi|^p).
\label{pgrowth}\tag{p-growth}
\end{equation*}
We assume that the body is subjected to applied body 
forces acting in its interior, \(\Omega_\epsi\), and
to applied surface  forces acting on  the portion
\(S_\epsi\) of  its  boundary,
both of the type dead loads  and of densities \(\tilde f_\epsi\in L^q(\Omega_\varepsilon;{\mathbb
R}^3)\) and \(\tilde g_\epsi\in L^q(S_\varepsilon;{\mathbb
R}^3)\), respectively, where \(q\) satisfies \(\frac{1}{p} + \frac{1}{q}=1\). We assume further that the body satisfies
a deformation condition \(\tilde \ffi_{0,\epsi}\in W^{1,p}(\Omega_\epsi;\RR^3)\)
imposed on \(\Gamma_\epsi\). In the literature,   \(\tilde
\ffi_{0,\epsi} \)  commonly coincides
with the identity function on \(\overline\Omega^\epsi\), which corresponds to the
clamped setting. Here, we   address a more general case
that we  detail later on.

In this setting, the equilibrium problem can be formulated
as the  minimization problem \vspace{-.5mm}
\begin{equation}\label{Ptepsi} 
\inf
\big\{\widetilde E_\varepsilon(\tilde\psi)\!:\, \tilde\psi \in
\tilde \Phi_\epsi \big\},
\tag{$\widetilde
{\mathcal{P}}_\epsi$}
\end{equation}
where, denoting by \(\HH^2\) the two-dimensional
Hausdorff  measure, 
\begin{equation}\label{energystandardforces}
\begin{aligned}
\widetilde E_\varepsilon(\tilde\psi):=\int_{\Omega_\varepsilon}
W(\nabla\tilde\psi)\,\d\tilde x - \int_{\Omega_\varepsilon}
\tilde
f_\epsi\cdot\tilde\psi\,\d\tilde x - \int_{S_\varepsilon}
\tilde g_\epsi\cdot\tilde\psi\,\d{\HH}^2(\tilde x)
\end{aligned}
\end{equation}
and \vspace{-1mm}
\begin{equation}\label{deftildePhiepsi}
\begin{aligned}
\tilde \Phi_\epsi:= \big\{\tilde\psi \in 
W^{1,p}(\Omega_\varepsilon;{\mathbb
R}^3)\!:\, \tilde \psi = \tilde
\ffi_{0,\epsi} \hbox{ on } \Gamma_\epsi\big\}.
\end{aligned}
\end{equation}
Note that we can write  $\widetilde E_\varepsilon (\tilde \psi)=
\widetilde E^a_\varepsilon
(\tilde \psi) + \widetilde E^b_\varepsilon (\tilde \psi)$,
where
\begin{equation*}
\begin{aligned}
&\widetilde E^a_\varepsilon(\tilde \psi):=\int_{\Omega^a_\varepsilon}
W(\nabla\tilde\psi)\,\d\tilde x - \int_{\Omega^a_\varepsilon}
\tilde
f_\epsi\cdot\tilde\psi\,\d\tilde x - \int_{S^a_\varepsilon}
\tilde g_\epsi\cdot\tilde\psi\,\d{\HH}^2(\tilde x),\\
& \widetilde E^b_\varepsilon(\tilde \psi):=\int_{\Omega^b_\varepsilon}
W(\nabla\tilde\psi)\,\d\tilde x - \int_{\Omega^b_\varepsilon}
\tilde
f_\epsi\cdot\tilde\psi\,\d\tilde x - \int_{S^{b,-}_\varepsilon\cup
S^{b,+}_\varepsilon}
\tilde g_\epsi\cdot\tilde\psi\,\d{\HH}^2({\tilde x}).
\end{aligned}
\end{equation*}

As it is usual in the framework of dimension-reduction
problems, the first step to study the asymptotic behavior
of a diagonal infimizing sequence of the sequence of problems \eqref{Ptepsi} is to transform these problems into equivalent ones defined on a fixed domain. To this end, we consider the change of variables that to each point \(\tilde x =(\tilde x_\alpha,\tilde x_3)\in \overline \Omega^a_\varepsilon
\) associates the point \(x=(x_\alpha,x_3):=
({r_\varepsilon}^{-1}{\tilde x_\alpha},\tilde x_3)\in\overline
\Omega^a\)
and  that to each point  \(\tilde x
=(\tilde x_\alpha,\tilde x_3)\in \overline\Omega^b_\varepsilon
\)  associates the point \(x=(x_\alpha,x_3):=
(\tilde x_\alpha, {h_\varepsilon}^{-1}{\tilde x_3})\in\overline\Omega^b\),
and we define
\begin{equation*}
\begin{aligned}
&\psi^a(x):=\tilde \psi (r_\varepsilon x_\alpha,
x_3) \text{ and } \ffi^a_{0,\varepsilon}(x):=\tilde \ffi_{0,\epsi}(r_\varepsilon x_\alpha,
x_3) \text{ for $x=(x_\alpha,x_3)\in\Omega^a$},\\
& \psi^b(x):=\tilde \psi (x_\alpha, h_\varepsilon
x_3) \text{ and } \ffi^b_{0,\varepsilon}(x):=\tilde \ffi_{0,\epsi}
(x_\alpha, h_\varepsilon
x_3) \text{ for $x=(x_\alpha,x_3)\in\Omega^b$}, \\
& \psi^{b,+}(x_\alpha):=\psi^b(x_\alpha,0)  \text{ and
} \psi^{b,-}(x_\alpha):=\psi^b(x_\alpha, -1) \text{ for \(x_\alpha\in \omega^b.\)}
\end{aligned}
\end{equation*}
Observe that \(\tilde\psi \in
\tilde \Phi_\epsi\) if and only if \((\psi^a - \ffi^a_{0,\varepsilon}, \psi^b - \ffi^b_{0,\varepsilon})\in
W^{1,p}_{\Gamma^a}(\Omega^a;{\mathbb R}^3)\times
W^{1,p}_{\Gamma^b}(\Omega^b;{\mathbb R}^3)\), where \(W^{1,p}_{\Gamma}(\Omega;{\mathbb R}^3) \allowbreak := \{\psi \in W^{1,p}(\Omega;{\mathbb R}^3)\!:\,
\psi = 0 \text{ on } \Gamma\}\), and the  junction
condition \vspace{-.5mm}
\begin{equation}\label{junction}
\psi^a(x_\alpha,0_3) =
\psi^b(r_\varepsilon x_\alpha,0_3)=\psi^{b,+}(r_\varepsilon x_\alpha) \enspace \hbox{for \aev\
}
x_\alpha\in\omega^a
\end{equation}
holds. Note  that \(\ffi^a_{0,\varepsilon}\) and 
\(\ffi^b_{0,\varepsilon}\) satisfy \eqref{junction};
i.e., \(\ffi^a_{0,\varepsilon}(x_\alpha,0) =
\ffi^b_{0,\varepsilon}(r_\varepsilon x_\alpha,0)\) for \aev\
\(x_\alpha\in\omega^a\).

Regarding the densities of the applied
forces, we similarly define
\begin{equation}\label{forcesepsi}
\begin{aligned}
&f^a_\epsi(x):=\tilde f_\epsi (r_\varepsilon
x_\alpha,
x_3)  \text{ for $x=(x_\alpha,x_3)\in\Omega^a$},\enspace g^a_\epsi(x):=\tilde g_\epsi (r_\varepsilon
x_\alpha,
x_3)  \text{ for $x=(x_\alpha,x_3)\in S^a$},\\
& f^b_\epsi(x):=\tilde f_\epsi (
x_\alpha, h_\varepsilon
x_3) \text{ for $x=(x_\alpha,x_3)\in\Omega^b$}, \enspace g^{b,+}_\epsi(x_\alpha):=\tilde
g_\epsi (x_\alpha ,0) \text{ for \(x_\alpha\in \omega^b\backslash r_\epsi\overline
\omega^a\)}, \\
&  g^{b,-}_\epsi(x_\alpha):=\tilde
g_\epsi (x_\alpha,
-h_\epsi) \text{ for \(x_\alpha\in \omega^b.\)}
\end{aligned}
\end{equation}
In what follows, we assume that
the limit \vspace{-1mm}
\begin{equation}
\label{ell}
\begin{aligned}
\ell:=\lim_{\varepsilon \to 0}\frac{h_\varepsilon}{r_\varepsilon^2}
\end{aligned}
\end{equation}
 exists. We note that \({h_\varepsilon}/{r_\varepsilon^2}\)  represents the ratio between the  thickness of the plate-shaped
 domain and the area of the cross-section of the
  tube-shaped domain and that three cases, \(\ell=0\),
\(\ell\in \mathbb R^+\), and \(\ell=\infty\), must be distinguished.
We will often use the index  \(\ell_0\) if \(\ell=0\),  \(\ell_+\) if \(\ell\in\RR^+\),
and  \(\ell_\infty \) if \(\ell=\infty\)
to highlight the dependence on the value of the limit in \eqref{ell}.

As it is well-known (see, for instance, \cite{FJM06}), different limit regimes appear according to a balance between the scaling of the applied forces and the energy functional. Here,  we aim at the derivation of membrane-string models incorporating bending-torsion moments understanding, simultaneously,  the impact of the ratio 
\({h_\varepsilon}/{r_\varepsilon^2}\).  Accordingly, we 
 further specify the asymptotic
behavior of the functions in \eqref{forcesepsi} as follows.
We  assume that  there exist functions
$f^a\in L^q(\Omega^a;{\mathbb R}^3)$,
$f^b\in L^q(\Omega^b;{\mathbb
R}^3)$, $g^a,G^a\in
L^q(S^a;{\mathbb R}^3)$,
$g^{b,\pm},G^b\in L^q(\omega^b;{\mathbb
R}^3)$, \(\hat g^{b,-}\in L^q(\omega^a;\RR^3)\), and
\(\hat
G^b\in L^q(\mathfrak{C};{\mathbb
R}^3)\) with \(\mathfrak{C}\)
a convex subset of \(\RR^2\)
containing
\(\omega^a
\), all independent
of  $\varepsilon$, such
that 
\begin{description}
\item[Case \(l\in\RR^+\)] 
\vspace{-.5mm}
\begin{equation}\label{forcesl+}
\begin{aligned}
&f_\epsi^a =
f^a, \enspace g_\epsi^a = r_\varepsilon
g^a +G^a,\\
& f_\epsi^b =
f^b, \enspace g^{b,+}_\epsi = h_\varepsilon
g^{b,+} + G^b, \enspace g^{b,-}_\epsi =
-\big(h_\varepsilon
g^{b,-}
+ G^b\big)\chi_{\omega^b\backslash
r_\varepsilon\overline\omega^a} - \big(h_\varepsilon
\hat g^{b,-}
+ \hat G^b\big)\chi_{
r_\varepsilon\overline\omega^a}.
\end{aligned}
\end{equation}

\item[Case \(l=\infty\)] 
\vspace{-.5mm}
\begin{equation}\label{forcesli}
\begin{aligned}
&f_\epsi^a =
f^a, \enspace g_\epsi^a = r_\varepsilon
g^a +G^a,\\
& f_\epsi^b =
\tfrac{r_\epsi^2}{h_\epsi}f^b, \enspace g^{b,+}_\epsi = r_\varepsilon^2
g^{b,+} + \tfrac{r_\epsi^2}{h_\epsi}G^b, \enspace g^{b,-}_\epsi =
-\big(r_\varepsilon^2
g^{b,-}
+ \tfrac{r_\epsi^2}{h_\epsi}G^b\big)\chi_{\omega^b\backslash
r_\varepsilon\overline\omega^a} - \big(r_\varepsilon^2
\hat g^{b,-}
+\tfrac{r_\epsi^2}{h_\epsi} \hat G^b\big)\chi_{
r_\varepsilon\overline\omega^a}.
\end{aligned}
\end{equation}

\item[Case \(l=0\)] 
\vspace{-.5mm}
\begin{equation}\label{forcesl0}
\begin{aligned}
&f_\epsi^a =
\tfrac{h_\varepsilon}{r_\varepsilon^2}f^a, \enspace g_\epsi^a = \tfrac{h_\varepsilon}{r_\varepsilon}g^a
+\tfrac{h_\varepsilon}{r_\varepsilon^2}G^a,\\
& f_\epsi^b =
f^b, \enspace g^{b,+}_\epsi = h_\varepsilon
g^{b,+} + G^b, \enspace g^{b,-}_\epsi =
-\big(h_\varepsilon
g^{b,-}
+ G^b\big)\chi_{\omega^b\backslash
r_\varepsilon\overline\omega^a} - \big(h_\varepsilon
\hat g^{b,-}
+ \hat G^b\big)\chi_{
r_\varepsilon\overline\omega^a}.
\end{aligned}
\end{equation}

\end{description}
Here, the symbol $\chi_A$
stands for the characteristic function
of the set $A$. We assume further that $G^a(x_\alpha,x_3)\allowbreak=\Bcal^a(x_3)\nu(x_\alpha,x_3)$,
where $\Bcal^a$  is a
matrix, only depending on $x_3$, associated
with a linear
application
from ${\mathbb R}^3$ into ${\mathbb
R}^3$ and $\nu$ is
the unit outer normal to $S^a$.

Finally, we re-scale the total energy  \(\widetilde
E_\epsi\) by setting \(E_\varepsilon(\psi^a,\psi^b):=\tfrac1{
r_\varepsilon^2}
\widetilde E_\varepsilon(\tilde\psi)
= \tfrac1{
r_\varepsilon^2}
\widetilde E^a_\varepsilon(\tilde\psi)
+\tfrac1{
r_\varepsilon^2}
\widetilde E^b_\varepsilon(\tilde\psi) =:
E^a_\varepsilon(\psi^a) + E^b_\varepsilon(\psi^b) \). We have that \vspace{-2mm}
\begin{equation}\label{Eabepsi}
\begin{aligned}
E^a_\varepsilon(\psi^a) = F^a_\varepsilon(\psi^a)
- L^a_\varepsilon(\psi^a) \enspace \text{
and }\enspace E^b_\varepsilon(\psi^b)=\frac{h_\varepsilon}{r_\varepsilon^2}
F^b_\varepsilon(\psi^b) - \frac{h_\varepsilon}{r_\varepsilon^2}
L^b_\varepsilon(\psi^b), 
\end{aligned}
\end{equation}%
where \vspace{-1.5mm}
\begin{align}
&F^a_\varepsilon(\psi^a):=\int_{\Omega^a}
W({r_\varepsilon^{-1}}\nabla_\alpha
\psi^a|\nabla_3\psi^a)\,\dx , \quad
F^b_\varepsilon(\psi^b):=\int_{\Omega^b}
W(
\nabla_\alpha \psi^b|{h_\varepsilon^{-1}}\nabla_3\psi^b)\,\dx,
\label{Fabepsi}\\
& L^a_\varepsilon(\psi^a):=\int_{\Omega^a}
f^a_\epsi\cdot\psi^a\, \dx + \frac{1}{r_\epsi}
\int_{S^a} g^a_\epsi\cdot\psi^a\,\d{\HH}^2(x)
, \label{Laepsi}\\
& L^b_\varepsilon(\psi^b):=\int_{\Omega^b}
f^b_\epsi\cdot\psi^b\,\dx +\frac{1}{h_\epsi}
\int_{\omega^b\backslash r_\varepsilon
\overline\omega^a}
(g^{b,+}_\epsi\cdot \psi^{b,+} +g^{b,-}_\epsi\cdot\psi^{b,-})\,\dx_\alpha
+\frac{1}{h_\epsi}
\int_{r_\varepsilon \omega^a}
g^{b,-}_\epsi\cdot\psi^{b,-} \dx_\alpha
.\label{Lbepsi}
\end{align}
%

As justified in \cite{FRS93} ({also see
\cite{TrVi96}}), to obtain
a nonlinear membrane (string) behavior
in the limit as the thickness (cross-section) parameter of the thin plate-shaped
(tube-shaped)
domain goes to zero, the  scaling magnitude of  the applied body  forces
should be of order one, while the  scaling
magnitude of
 the applied  surface forces should
be of the same order of the thickness (cross-section)
parameter. The assumptions on the asymptotic behavior of the forces in  
\eqref{forcesl+}--\eqref{forcesl0} regarding
the terms \(f^a\), \(g^a\), \(f^b\),
\(g^{b,+}\), and \(g^{b,-}\)
are the simplest compatible with these
order of scaling magnitudes having in mind the
scaling of the total energy functional
\(E_\epsi\) and  the value of \(\ell\) in \eqref{ell}.

As
it will become clear later on, the presence
of the terms \(G^a\) and  \(G^b\), of the same order
of \(f^a\) and \(f^b\), respectively,
will induce the appearance of \textit{bending-torsion
moments}
terms
in the limit model.
As we mentioned before, this approach was considered before in \cite{BZZ08,
BoFo04, BFM03, CMMO17, RiMSc, LaNg14, LaNg16} for thin structures but not
multi-structures.

We mention further that in \cite[Sect.~3.3]{FRS93}
the authors
assert that a thin plate-shaped domain
cannot support a non-vanishing resultant
surface load as the thickness parameter
goes to zero. Due to the multi-domain
feature of the body considered here,
where there are no applied surface forces
on \(r_\epsi\omega^a \times \{0\}\)
(which, we recall, represents the interface
between  \(\Omega^a_\epsi\)
and \(\Omega^b_\epsi\)), that principle
is not satisfied if \(\hat g^{b,-}\)
or \(\hat G^b\) are different from zero.
It turns out that the term \(\hat g^{b,-}\),
of the same order of \( g^{b,-}\),  plays
no role in the limit model because it has the
standard order of scaling  magnitude and is acting
on a set of vanishing area. In contrast,    the term \(\hat G^b\), 
of the
same order of \(f^b\), will contribute
to a junction-type
term in the limit model (for \(\ell\in\RR^+)\) that is independent
of \(p\). 
   This represents a novelty
compared to \cite{GGLM02}, where the limit
model has junction-type terms only if \(p>2\).

Finally, we observe that the above change of variables and re-scaling allow us
to re-write \eqref{Ptepsi} as \vspace{-.5mm}
\begin{equation}\label{Pepsi} 
\inf
 \big\{ E^a_\varepsilon(\psi^a) + E^b_\varepsilon(\psi^b)\!: \, (\psi^a,\psi^b) \in
 \Phi_\epsi \big\},
\tag{$
{\mathcal{P}}_\epsi$}
\end{equation}
where \vspace{-2.5mm}
\begin{equation}\label{Phiepsi}
\begin{aligned}
\Phi_\epsi := \big\{ (\psi^a,\psi^b)\in W^{1,p}(\Omega^a;{\mathbb
R}^3) \times W^{1,p}(\Omega^b;{\mathbb
R}^3)\!:& \,\, \psi^a = 
\ffi^a_{0,\epsi} \hbox{ on } \Gamma^a,\,
 \psi^b =   \ffi^b_{0,\epsi}
\hbox{ on } \Gamma^b,\\
& \hbox{ and } (\psi^a,\psi^b)
\hbox{ 
satisfies } \eqref{junction}  \big\}.
\end{aligned}
\end{equation}

To describe the asymptotic behavior of \eqref{Pepsi},
we are left to detail the assumptions on 
\((\ffi^a_{0,\epsi})_{\epsi>0}\) and \((\ffi^b_{0,\epsi}
)_{\epsi>0}\).  We assume that there exist  \(\ffi^a_0 \in W^{1,p}(\Omega^a;\R^3)\) and  \(\ffi^b_0 \in W^{1,p}(\Omega^b;\R^3)\)
such that \vspace{-.mm}
\begin{align}
&\ffi^a_{0,\epsi} \weakly \ffi^a_0 \hbox{ weakly in } 
W^{1,p}(\Omega^a;\R^3), \enspace \big(|r_{\epsi}^{-1}\grad_\alpha \ffi^a_{0,\epsi}|^p
+ |\grad_3\ffi^a_{0,\epsi}|^p
\big)_{\epsi>0}\subset L^1(\Omega^a)
\hbox{ is equi-integrable},\label{bca}\tag{$b.c.^a$}
\\
&\ffi^b_{0,\epsi} \weakly \ffi^b_0 \hbox{ weakly in } 
W^{1,p}(\Omega^b;\R^3),\enspace\big(|\grad_\alpha \ffi^b_{0,\epsi}|^p + |h_{\epsi}^{-1}\grad_3\ffi^b_{0,\epsi}|^p
\big)_{\epsi>0}\subset L^1(\Omega^b) \hbox{ is equi-integrable}.
\label{bcb}\tag{$b.c.^b$}
\end{align}
Note that the functions \(\ffi^a_{0,\epsi}(x) =
(r_\varepsilon x_\alpha, x_3) \) and \(\ffi^b_{0,\epsi}
= (x_\alpha, h_\varepsilon x_3) \) corresponding to the
clamped case, which is commonly considered in the literature,
satisfy \eqref{bca}--\eqref{bcb}.

Next, we state our main results concerning the three cases \(\ell\in\RR^+\),
\(\ell=\infty\), and \(\ell=0\), where \(\ell\) is given by \eqref{ell}.
We start by introducing the spaces \vspace{-.5mm}
\begin{equation}\label{Phipl+}
\begin{aligned}
\Phi_{\ell_+}^p := \big\{ &(\psi^a,\psi^b)\in W^{1,p}(\Omega^a;{\mathbb
R}^3) \times W^{1,p}(\Omega^b;{\mathbb
R}^3)\!: \, \psi^a = 
\ffi^a_{0} \hbox{ on } \Gamma^a,\,
 \psi^b =   \ffi^b_{0}
\hbox{ on } \Gamma^b,\\
&\quad \psi^a \hbox{ is  independent of } x_\alpha, 
 \, \psi^b \hbox{ is  independent of } x_3,
 \, \hbox{and for } p>2,\, \psi^a(0_3)=\psi^b(0_\alpha) \big\},
\end{aligned}
\end{equation}
\begin{align}\label{Phipli}
\Phi_{\ell_\infty }^p := \big\{ \psi^a\in
W^{1,p}(\Omega^a;{\mathbb
R}^3) \!: \, \psi^a = 
\ffi^a_{0} \hbox{ on } \Gamma^a,\,
  \psi^a \hbox{ is  independent of }
x_\alpha,  \, 
\hbox{and for } p>2,\, \psi^a(0_3)=0
\big\},
\end{align}
and \vspace{-1mm}
\begin{align}\label{Phiplz}
\Phi_{\ell_0}^p := \big\{ \psi^b\in
 W^{1,p}(\Omega^b;{\mathbb
R}^3)\!: \, \psi^b =   \ffi^b_{0}
\hbox{ on } \Gamma^b, \psi^b \hbox{ is  independent of
} x_3, \,
  \hbox{and for } p>2,\, \psi^b(0_\alpha)=0
\big\}.
\end{align}

We refer the reader to Section~\ref{Sect:prelim} for a brief overview regarding the convex, quasiconvex, and cross-quasiconvex-convex envelopes of a function, which appear in our main theorems below.

\begin{theorem}[\(\ell\in\RR^+\)]\label{thm:ellr+}
Let \(W:\R^{3\times 3} \to \R\) be a Borel function
satisfying
\eqref{pgrowth} and let \((\psi^a_\epsi,\psi^b_\epsi)_{\epsi>0}\) be a diagonal
infimizing sequence of the sequence of problems \eqref{Pepsi},
where  \(\ell\)  given by \eqref{ell} 
is such that  \(\ell\in\RR^+\),     \((\ffi^a_{0,\epsi},\ffi^b_{0,\epsi})_{\epsi>0}\)
satisfies
\eqref{bca}--\eqref{bcb} and \eqref{junction},
and \eqref{forcesl+} holds.
Assume that \(0_\alpha\) is a Lebesgue point of  \( |
\hat G^b|^q\).  Then, the sequences 
\((\bar \bcal^a_\epsi,\psi^a_\epsi)_{\epsi>0}\) and \((\psi^b_\epsi,\bar \bcal^b_\epsi)_{\epsi>0}\), where \(\bar\bcal^a_\epsi:={r_{\varepsilon}^{-1}} \int_{\omega^a}\nabla_\alpha
\psi^a_\epsi\, \dx_\alpha\) and \(\bar\bcal^b_\epsi:= {h_\epsi^{-1}}\int_{-1}^{0}\nabla_3 \psi^b\,\dx_3 \), are sequentially, weakly compact in
\(  L^p((0,L);\R^{3 \times 2}) \times W^{1,p}(\Omega^a;\R^3) \) and \(W^{1,p}(\Omega^b;\R^3)\times L^p(\omega^b;\R^{3})\),
respectively. If \((\bar\bcal^a, \psi^a)\) and \(( \psi^b,\bar \bcal^b)\) are corresponding accumulation points, then \((\psi^a,\psi^b)\in
\Phi_{\ell_+}^p\) and they  solve the
minimization problem \vspace{-1.7mm}
\begin{equation}\label{Pell+} 
\min
\big\{ E_{\ell_+}((\bar\bcal^a,\psi^a), (\psi^b,\bar \bcal^b))\!:\, (\psi^a,\psi^b) \in\Phi_{\ell_+}^p,\, \, (\bar\bcal^a,\bar
\bcal^b) \in L^p((0,L);\R^{3 \times
2}) \times L^p(\omega^b;\R^{3})\big\} ,
\tag{$
{\mathcal{P}}_{\ell_+}$}
\end{equation}
where, for \(\bar a:=|\omega^a|\),
\(\ce\) the convex envelope of \(W\), and \(\qce\) the cross-quasiconvex-convex envelope of \(W\), \vspace{-3.9mm} %
\begin{equation}
\label{Eell+}
\begin{aligned}
E_{\ell_+}((\bar\bcal^a,\psi^a), (\psi^b,\bar \bcal^b)):= &\, \bar
a\int_0^L \ce({\bar a^{-1}}\bar\bcal^a|\nabla_3\psi^a)\,\dx_3+ \ell\int_{\omega^b}
 \qce(\nabla_\alpha \psi^b|\bar\bcal^b)\,\dx_\alpha\\
 &\quad - \int_0^L
\big(
\bar f^a\cdot\psi^a+ \bar g^a\cdot\psi^a + 
\Bcal^a:(\bar\bcal^a|0)\big)\,\dx_3 \\
&\quad -\ell \int_{\omega^b} \big(\bar f^b\cdot\psi^b +
(g^{b,+}-g^{b,-})\cdot\psi^{b} + G^b\cdot\bar\bcal^b\big)
\,\dx_\alpha +\bar a\, \hat G^b(0_\alpha)\cdot
\psi^a(0_3)
\end{aligned}
\end{equation}
with \vspace{-3.2mm}
\begin{equation}\label{barforces}
\begin{aligned}
\bar f^a(x_3):= \int_{\omega^a} f^a(x)\,\dx_\alpha, \quad
\bar g^a(x_3):= \int_{\partial\omega^a} g^a(x)\,\d\HH^1(x_\alpha), \quad \bar f^b(x_\alpha):= \int_{-1}^0 f^b(x)\,\dx_3.
\end{aligned}
\end{equation}
\end{theorem}

\begin{remark}[on Theorem~\ref{thm:ellr+}]
\label{rmk:commentl+}
The problem treated in Theorem~\ref{thm:ellr+} is in the spirit of that in \cite{GGLM02}. Precisely, in  \cite[Theorem~1.1 with \(N=3\)]{GGLM02}, the authors study the asymptotic behavior as \(\epsi\to0^+\) of \vspace{-.7mm}
\begin{equation*}
\begin{aligned}
\min\bigg\{&\int_{\Omega^a}
\Big(A(x,\psi^a,{r_\varepsilon^{-1}}\nabla_\alpha
\psi^a,\nabla_3\psi^a) + f_\epsi\psi^a\Big)\,\dx + \frac{h_\varepsilon}{r_\varepsilon^2}\int_{\Omega^b}
\Big(A(x,\psi^b,
\nabla_\alpha \psi^b,{h_\varepsilon^{-1}}\nabla_3\psi^b) + f_\epsi\psi^b\Big)\, \dx\!:\\
&\qquad (\psi^a,\psi^b)\in W^{1,p}(\Omega^a) \times W^{1,p}(\Omega^b), \, \psi^a(x_\alpha,0) =
\psi^b(r_\varepsilon x_\alpha,0) \enspace
\hbox{for \aev\
}
x_\alpha\in\omega^a\bigg\},
\end{aligned}
\end{equation*}
where \(A:\Omega\times\RR\times \RR^2 \times\RR \to \RR\) is a Caratheodory function satisfying the usual \(p\)-growth conditions and is such that 
\(A(x,\cdot,\cdot,\cdot)\) is convex for \aev\ \(x\in\Omega\); in 
\cite{GGLM02}, \(\omega^a \equiv \omega^b=:\omega\), \(L=1\), and 
\(\Omega= \omega\times (-1,1)\).
Moreover, \(f^\epsi \weakly f\) in \(L^q(\Omega)\), for some \(f\in L^q(\Omega)\), and 
\(\lim_{\epsi\to0^+} \frac{h_\varepsilon}{r_\varepsilon^2}= \ell \in \RR^+ \).

Here, we do  not assume any   convexity or continuity hypotheses; our stored energy function, \(W\), is only assumed to be a Borel function. Thus, we cannot avoid the relaxation step in our analysis. We also observe that the study in \cite{GGLM02} takes into account the behavior of \((\psi^a_\epsi)_\epsi\) and \((\psi^b_\epsi)_\epsi\) in \(W^{1,p}\) and of \((\bcal^a_\epsi)_\epsi \equiv
({r_\varepsilon^{-1}}\nabla_\alpha
\psi^a)_\epsi\) and \((\bcal^b_\epsi)_\epsi \equiv({h_\varepsilon^{-1}}\nabla_3\psi^b)_\epsi\) in \(L^p\). Here, given the type of forces that we consider, besides the behavior of \((\psi^a_\epsi)_\epsi\) and \((\psi^b_\epsi)_\epsi\) in
\(W^{1,p}\),  the relevant behavior
is that of  the averages \((\bar\bcal^a_\epsi)_\epsi \equiv({r_\varepsilon^{-1}} \int_{\omega^a} \nabla_\alpha
\psi^a\,\dx_\alpha)_\epsi\) and 
\((\bar\bcal^b_\epsi)_\epsi \equiv({h_\varepsilon^{-1}}\int_{-1}^0\nabla_3\psi^b\,\dx_3)_\epsi\) in \(L^p\). It was conjectured in \cite{BFM09}, for the membrane case, that  if one considers the behavior of \((\bcal^b_\epsi)_\epsi\) (in place of \((\bar\bcal^b_\epsi)_\epsi\)) without some kind of convexity hypothesis on the stored energy function, one is led to a nonlocal limit problem. We mention further that in \cite{GaZa07}, the authors characterize the asymptotic behavior of the functional considered in  \cite{GGLM02} assuming continuity but no convexity hypotheses on the stored energy function; however, in  \cite{GaZa07},  the behavior of  \((\bcal^a_\epsi)_\epsi\),
  \((\bcal^b_\epsi)_\epsi\), \((\bar\bcal^a_\epsi)_\epsi\),
and  \((\bar\bcal^b_\epsi)_\epsi\)
  is neglected (as in \cite{LDR95} for the membrane case). 

Similarly to \cite{GGLM02}, the limit model \eqref{Pell+} is coupled only if \(p>2\). However, given the non-standard scaling of the surface forces, which are absent in \cite{GGLM02}, our model includes a pseudo-coupling term, \(\bar a\, \hat G^b(0_\alpha)\cdot
\psi^a(0_3)\), which is independent of \(p\). This novel term represents an asymptotic balance between the applied surface forces on the bottom part of the multi-structure and the interaction at its junction by means of the trace
of the deformation on the top  part.
\end{remark}

In contrast with previous works in the nonlinear setting for multi-structures, in particular \cite{GaZa07,GGLM02}, we also characterize the limit problem for different asymptotic behaviors of the ratio \({h_\varepsilon}/{r_\varepsilon^2}\); precisely, for \(\ell=0\) and \(\ell=\infty\), where \(\ell \) is given by \eqref{ell} under additional hypotheses that we detail next. We obtain results that resemble those derived in \cite{GMMMS07} for the linear case.

In order to treat the \(\ell=\infty\) and \(\ell=0\) cases, we need to impose
a stronger coercivity hypothesis on \(W\) than that in \eqref{pgrowth}; precisely,
we assume that there is a positive constant, \(C\),  such that
for all \(\xi\in \R^{3\times 3}\), we have \vspace{-2mm}
\begin{equation}\label{coercrigid2}
\begin{aligned}
W(\xi) \geq \frac1C \dist^p(\xi,SO(3)\cup SO(3)A),
\end{aligned}
\end{equation}
where \(\II\)  is the identity matrix in \(\RR^{3\times 3}\),  $A$ is any other matrix in $ \RR^{3\times 3}$ such that \(A\) and \(\II\) are {\it strongly incompatible} (see \cite{Mat92, CM}), and   \(SO(3)=\{M\in \RR^{3\times 3}\!: MM^T=\II, \,  \det M=1\}\)
is the space of proper rotations in \(\RR^3\).
Note that  \eqref{coercrigid2} implies
that \(W(\xi) \geq \frac{1}{C'}|\xi|^p
- C'\) for some \(C'>0\) independent of \(\xi\); thus, if \eqref{coercrigid2}
holds,
then the lower bound in \eqref{pgrowth}
also holds. 
We observe further that  \eqref{coercrigid2} is a natural assumption for two-phase materials (see, for instance, \cite{MoMu07, Sve93}). For such materials, the set \(SO(3)\cup SO(3)A\) corresponds to the set of deformations that carry zero elastic energy; moreover,  each copy of \(SO(3)\), \(SO(3)\) and \(SO(3) A\), is called an energy well.
Note, however, that we are not imposing here that \(W\) vanishes on \(SO(3)\cup SO(3)A\). To treat  the \(\ell=\infty\) and \(\ell=0\) cases, we    assume only
that \vspace{-1mm}
\begin{align}
&W(\II)=0,\label{WI=0}
\end{align}
which means that the
reference configuration is a natural state.

 \begin{theorem}[\(\ell=\infty\)]\label{thm:elli}
 Let \(W:\R^{3\times 3} \to \R\) be a Borel function
satisfying
\eqref{pgrowth}, \eqref{coercrigid2}, and \eqref{WI=0}.
  Let \((\psi^a_\epsi,\psi^b_\epsi)_{\epsi>0}\) be a diagonal
infimizing sequence of the sequence of problems \eqref{Pepsi},
where   \((\ffi^a_{0,\epsi},\ffi^b_{0,\epsi})_{\epsi>0}\)
with \(\ffi^b_{0,\epsi}
\equiv (x_\alpha, h_\varepsilon x_3)\) satisfies
\eqref{bca}--\eqref{bcb} and \eqref{junction}
and where \eqref{forcesli} holds. Assume that \(p>2\), 
 \(\lim_{\varepsilon \to 0}{h_\varepsilon}^{p+1}/{r_\varepsilon^2}
= \infty\), and  \((G^b(r_\epsi \cdot))_{\epsi>0}\) is bounded in \(L^q(\omega^a;\RR^3)\).
 Let \(\psi^b\equiv (x_\alpha,0)\)
and \(\bcal^b \equiv (0_\alpha,1)\).
Then, \((\psi^b_\epsi, h_\epsi^{-1}\nabla_3
\psi^b_\epsi) \to (\psi^b, \bcal^b)\) in \(W^{1,p}(\Omega^b;\R^3)\times
L^p(\Omega^b;\R^{3})\). Moreover,   the sequence 
\((\bar \bcal^a_\epsi,\psi^a_\epsi)_{\epsi>0}\), where \(\bar\bcal^a_\epsi:={r_{\varepsilon}^{-1}}
\int_{\omega^a}\nabla_\alpha
\psi^a_\epsi\, \dx_\alpha\), is sequentially, weakly compact in
\(  L^p((0,L);\R^{3 \times 2}) \times W^{1,p}(\Omega^a;\R^3) \). If \((\bar\bcal^a, \psi^a)\) is a corresponding accumulation point, then \(\psi^a\in
\Phi_{\ell_\infty }^p\) and \((\bar\bcal^a,
\psi^a)\)  solves the 
minimization problem \vspace{-.5mm}
\begin{equation}\label{Pelli} 
\min
\big\{ E_{\ell_\infty }(\bar\bcal^a,\psi^a)\!: \, (
\bar\bcal^a, \psi^a  )\in
L^p((0,L);\R^{3 \times
2})\times \Phi_{\ell_\infty }^p \big\},
\tag{$
{\mathcal{P}}_{\ell_\infty}$}
\end{equation}
where, for \(\bar a:=|\omega^a|\), 
\(\ce\) the convex envelope of \(W\), and \(\bar f^a\), \(\bar g^a\), and \(\bar f^b\)  given by \eqref{barforces},
\vspace{-1mm}%
\begin{equation}
\label{Eelli}
\begin{aligned}
E_{\ell_\infty }(\bar\bcal^a,\psi^a):= &\, \bar
a\int_0^L \ce({\bar a^{-1}}\bar\bcal^a|\nabla_3\psi^a)\,\dx_3 - \int_0^L
\big(
\bar f^a\cdot\psi^a+ \bar g^a\cdot\psi^a + 
\Bcal^a:(\bar\bcal^a|0)\big)\,\dx_3 \\
&\quad - \int_{\omega^b} \big((\bar f^b_\alpha + 
g^{b,+}_\alpha -g^{b,-}_\alpha)\cdot x_\alpha  + G^b_3\big)
\,\dx_\alpha.
\end{aligned}
\end{equation}

\end{theorem}

\begin{remark}[on Theorem~\ref{thm:elli}]
\label{rmk:commentli}
(i) The restriction \(p>2\) in Theorem~\ref{thm:elli} is of a technical nature due to the
fact that  the limit condition \(\psi^a(0_3) = \psi^b(0_\alpha)
=0\)  may fail if \(p\leq2\). In this case, it seems
a very hard task to construct a recovery
sequence,
within our \(\Gamma\)-convergence analysis,
 that simultaneously satisfies
\eqref{junction} and cancels the
exploding coefficient in front of the
elastic energy in \(\Omega^b\). 
(ii) A key ingredient in our analysis of the  \(\ell=\infty\) and \(\ell=0\) cases is {(a \(p\)-version of) }the quantitative rigidity  estimate for scaled gradients
proved in \cite[Theorem~6]{FJM06} (also see Proposition~\ref{prop:thm6FJM}); this estimate together with the condition 
 \(\lim_{\varepsilon \to 0}{h_\varepsilon}^{p+1}/{r_\varepsilon^2} 
= \infty\) allows
us to \textit{properly} characterize the accumulation points in Theorem~\ref{thm:elli}.
(iii) In the same spirit as in  \cite{GMMMS07},  Theorem~\ref{thm:elli} shows that if \(r_\epsi^2 \ll h_\epsi^{p+1}\) with \(p>2\),  the limit behavior of the thin multi-structure is that of a rigid {plate} and a bent elastic string that is clamped at its lower extremity and satisfies a deformation condition at its upper extremity.
\end{remark}

Finally, we state the main theorem for the \(\ell=0\) case. The asymptotic behavior
of \eqref{Pepsi} when \(\ell=0\) is
encoded  in the  scaled problem
\(\frac{r_\epsi^2}{h_\epsi}\eqref{Pepsi}\);
this means that we will  consider an infimizing sequence
in \(\Phi_\epsi\) to the scaled energy
\(\frac{r_\epsi^2}{h_\epsi} \big(E^a_\varepsilon(\psi^a) + E^b_\varepsilon(\psi^b)\big)\). In
the linear case, this corresponds to
looking at the asymptotic behavior of
a scaled minimizing sequence (see, for
instance, \cite[Theorem~1-(iii)
and Corollary~1-(iii)]{GMMMS07}).

\begin{theorem}[\(\ell=0\)]\label{thm:ellz}
Let \(W:\R^{3\times 3} \to \R\) be
a Borel function
satisfying
\eqref{pgrowth}, \eqref{coercrigid2}, and \eqref{WI=0}. 
  Let \((\psi^a_\epsi,\psi^b_\epsi)_{\epsi>0}\)
be a diagonal
infimizing sequence of the sequence
of problems \(\frac{r_\epsi^2}{h_\epsi}\eqref{Pepsi}\),
where   \((\ffi^a_{0,\epsi},\ffi^b_{0,\epsi})_{\epsi>0}\)
with \(\ffi^a_{0,\epsi}
\equiv (r_\varepsilon x_\alpha, x_3)\)
satisfies
\eqref{bca}--\eqref{bcb} and \eqref{junction}
and where \eqref{forcesl0} holds. Assume
that \(p\leq2\), 
 \(\lim_{\varepsilon \to 0}{h_\varepsilon}/{r_\varepsilon^{p+2}}
= 0\), and  \((\frac{r_\epsi^2}{h_\epsi}G^b(r_\epsi \cdot))_{\epsi>0}\)
is bounded in \(L^q(\omega^a;\RR^3)\).
 Let \(\psi^a\equiv (0_\alpha,x_3)\)
and \(\bcal^a \equiv \II_\alpha\).
Then, \((r_\epsi^{-1}\grad_\alpha \psi^a_\epsi, 
\psi^a_\epsi) \to ( \bcal^a,\psi^a)\)
in \(
L^p(\Omega^a;\R^{3\times 2})\times W^{1,p}(\Omega^a;\R^3)\). Moreover,  
the sequence 
\((\psi^b_\epsi, \bar \bcal^b_\epsi)_{\epsi>0}\),
where \(\bar\bcal^b_\epsi:={h_{\varepsilon}^{-1}}
\int_{-1}^0\nabla_3
\psi^b_\epsi\, \dx_3\), is sequentially,
weakly compact in
\(   
W^{1,p}(\Omega^b;\R^3)\times 
L^p(\omega^b;\R^{3 }) \). If \((
\psi^b, \bar\bcal^b)\) is a corresponding accumulation
point, then \(\psi^b\in
\Phi_{\ell_0}^p\) and \((
\psi^b, \bar\bcal^b)\)  solves the 
minimization problem \vspace{-.5mm}
\begin{equation}\label{Pellz} 
\min
\big\{ E_{\ell_0 }(\psi^b,\bar \bcal^b)\!: \, (
\psi^b, \bar\bcal^b)\in  \Phi_{\ell_0}^p \times L^p(\omega^b;\R^{3 }) \big\},
\tag{$
{\mathcal{P}}_{\ell_0}$}
\end{equation}
where, being
 \(\qce\) the cross-quasiconvex-convex
envelope of \(W\) and  \(\bar f^a\), \(\bar g^a\), and
\(\bar f^b\)  given by \eqref{barforces},
\vspace{-1.mm}%
\begin{equation}
\label{Eell0torefer}
\begin{aligned}
E_{\ell_0}(
\psi^b, \bar\bcal^b):= &\int_{\omega^b}
 \qce(\nabla_\alpha \psi^b|\bar\bcal^b)\,\dx_\alpha
- \int_0^L
\big((
\bar f^a_3 + \bar g^a_3) x_3
+ 
\bar a(\Bcal^a_{11} + \Bcal^a_{22})\big)\,\dx_3 \\
&\quad - \int_{\omega^b} \big(\bar f^b\cdot\psi^b +
(g^{b,+}-g^{b,-})\cdot\psi^{b}+ G^b\cdot\bar\bcal^b\big)
\,\dx_\alpha.
\end{aligned}
\end{equation}
\end{theorem}

\begin{remark}[on Theorem~\ref{thm:ellz}]
\label{rmk:commentl0}
(i) The restriction \(p\leq2\) in Theorem~\ref{thm:ellz}
originates from a similar
technical difficulty mentioned in Remark~\ref{rmk:commentli}-(i).
However, the construction of the recovery
sequence in the \(\ell=0\) case is different
from the previous one and, in particular,
does not depend on the limit junction
condition. (ii) Also as in the previous case, the condition
 \(\lim_{\varepsilon \to 0}{h_\varepsilon}^{p+1}/{r_\varepsilon^2}
= \infty\) allows us to benefit from (a \(p\)-version of) the rigidity  estimate for scaled gradients
proved in \cite[Theorem~6]{FJM06} (also see Proposition~\ref{prop:thm6FJM}).
(iii)\ Finally, we observe that   Theorem~\ref{thm:ellz} shows that
if \( h_\epsi\ll r_\epsi^{p+2}\) with \(p\leq2\),  the limit behavior of
the thin multi-structure
is that of a rigid {beam} and a bent elastic membrane that  satisfies a deformation
condition on its boundary.

\end{remark}

\begin{remark}[on bending-torsion moments
in the limit models \eqref{Pell+}, \eqref{Pellz},
and \eqref{Pelli}]\label{bendingtorsion}
We observe that, in general,
the term
\(\bar\bcal^a\)
is not related to the one-dimensional
strain tensor of \(\psi^a\). Thus, \(\psi^a\)
and \(\bar\bcal^a\) must be regarded
as distinct macroscopic entities. Similarly,
\(\psi^b\)
and \(\bar\bcal^b\)  must be regarded
as distinct macroscopic entities.
We further observe that given the nature
of \(G^a\) and \(G^b\), \(\bar\bcal^a\)
accounts for bending and torsion moments
in the string, while \(\bar\bcal^b\)
accounts for bending  moments in the
membrane.
\end{remark}

This paper is organized as follows. In Section~\ref{Sect:prelim}, we recall the notions of convex, quasiconvex, and cross-quasiconvex-convex envelopes of a function and associated lower semicontinuity results that will be used throughout the paper. We also establish some preliminary results that are common to the three cases,  \(\ell\in\RR^+\),
\(\ell=\infty\), and \(\ell=0\). Next, in Section~\ref{Sect:l+}, we prove Theorem~\ref{thm:ellr+}. We also recover as a particular case the 3D-1D counterpart of the study in \cite{BFM03}, which was addressed in \cite{RiMSc} (see Section~\ref{Subs:rod}). Then, in Sections~\ref{Sect:li} and \ref{Sect:lz}, we prove Theorems~\ref{thm:elli} and \ref{thm:ellz}, respectively. Finally, in Section~\ref{Sect:forces}, we elaborate on variants of the models in Theorems~\ref{thm:ellr+}, \ref{thm:elli},
and \ref{thm:ellz} corresponding to instances where \(G^a\) or \(G^b\),
 inducing bending moments in the limit, is not present. In particular, we establish relationships with the models in \cite{ABP91,LDR95} (see Section~\ref{Subs:nobending}). We also discuss, in Section~\ref{Subs:divform}, the case where the system of applied forces is in divergence form as in \cite{GMMMS02, GMMMS07, MuSi99,
Musi00}. This divergence form allows for less regular body and surface density terms.

\section{Preliminary results
}\label{Sect:prelim}

In what follows, given a measurable set  \(A\subset\RR^n\),
we define  \(L^p_0(A;\R^l):= \{ u\in L^p(A;\R^l)\!:
\, \int_A u\,\dx = 0\}\) and we denote
by \(|A|\) its Lebesgue measure. 

An important argument within our analysis
relates to weakly lower semicontinuity properties
of integral functionals of the form
\vspace{-1.5mm}
\begin{equation*}
\begin{aligned}
(\psi,\bcal) \in W^{1,p}(\Omega;\RR^m)
\times L^p(\Omega;\RR^l) \mapsto I(\psi,\bcal):= \int_\Omega
\overline W(\grad \psi(x), \bcal(x))\,
\dx,
\end{aligned}
\end{equation*}
where \(\Omega\subset\RR^n\) is an open
and bounded set with Lipschitz boundary
and  \(\overline W: \RR^{m\times n} \times
\RR^l \to \RR\) is a Borel function  for which
there exists a positive constant, \(C \), such
that,
for all \((M,b) \in \RR^{m\times n} \times
\RR^l\), \vspace{-1mm}
\begin{equation}\label{growthcqc}
\begin{aligned}
- \frac{1}{C} \leq \overline W(M,b) \leq C( 1 + |M|^p + |b|^p). 
\end{aligned}
\end{equation}

It turns out (see \cite{LDR00, FKP94})
that the integral \(I\) above is sequentially
weakly lower semicontinuous in \(W^{1,p}(\Omega;\RR^m)
\times L^p(\Omega;\RR^l)\) if and only
if \(\overline W\) is cross-quasiconvex-convex;
that is,  setting \(Q:=(0,1)^n\), if
and only if for all  \((M,b) \in \RR^{m\times n}
\times
\RR^l\) and for all \(\theta \in W^{1,\infty}_0
(Q;\RR^m)\) and 
\(\eta \in L^\infty_0(Q;\RR^l)\),
we have \vspace{-1.mm}
\begin{equation}\label{defcqc}
\begin{aligned}
\overline W(M,b) \leq \int_{Q}
\overline W( M+ \grad \theta(x), b
+ \eta(x))\,\dx.
\end{aligned}
\end{equation}

It can be proved (see \cite[Proposition~4.4
and Corollary~4.6]{LDR00})  that if \(\overline
W\) is a cross-quasiconvex-convex function
satisfying \eqref{growthcqc}, then 
\(M\mapsto\overline W(M,b)\) is quasiconvex
for all \(b\in\RR^l\) fixed,  \(b
\mapsto \overline W(M,b)\) is convex
for all \(M\in\RR^{m\times n}\) fixed,
and there exists a positive constant,
\(C\), such that for all \((M_1,b_1), \,
(M_2,b_2) \in \RR^{m\times
n}
\times
\RR^l\),
we have \vspace{-.5mm}
\begin{equation*}
\begin{aligned}
&|\overline W(M_1,b_1) - \overline W(M_2,b_2)|\leq C(1 + |M_1|^{p-1} + |M_2|^{p-1}
+ |b_1|^{p-1} + |b_2|^{p-1})(|M_1- M_2|
+ |b_1 - b_2|).
\end{aligned}
\end{equation*}
Moreover, if the integral \(I\) above is not
  sequentially
weakly lower semicontinuous in \(W^{1,p}(\Omega;\RR^m)
\times L^p(\Omega;\RR^l)\), then its weak lower semicontinuous envelope in \(W^{1,p}(\Omega;\RR^m)
\times L^p(\Omega;\RR^l)\) has the following integral representation  (see \cite[Theorem~5.4]{FKP94}, \cite[Theorem~4.17]{LDR00}): \vspace{-1mm}
\begin{equation*}
\begin{aligned}
\int_\Omega \mathcal{QC}\overline W(\grad \psi(x), \bcal(x))\,
\dx 
\end{aligned}
\end{equation*}
for all \((\psi,\bcal) \in W^{1,p}(\Omega;\RR^m)
\times L^p(\Omega;\RR^l)\), where \(\mathcal{QC}\overline W\) is the cross-quasiconvex-convex
envelope of \(\overline W \); precisely,
 for all  \((M,b) \in \RR^{m\times
n}
\times
\RR^l\),
\vspace{-1mm}%
\begin{equation*}
\begin{aligned}
\mathcal{QC}\overline W(M,b) = \inf\bigg\{\int_{Q}
\overline W( M+ \grad \theta(x), b
+ \eta(x))\,\dx\!: \,  \theta \in W^{1,\infty}_0
(Q;\RR^m),\,\eta \in L^\infty_0(Q;\RR^l)\bigg\}.
\end{aligned}
\end{equation*}

In the \(l=m\) case, we can
associate to \(\overline W\) the function
\(\widetilde W:\RR^{m\times(n+1)}\to \RR\)
defined, for all   \((M,b) \in \RR^{m\times
n}
\times
\RR^m\),  by \(\widetilde W(M|b):= \overline
W(M,b)\). With this association in mind,
we have (see 
\cite[Corollary~4.21]{LDR00})
\begin{equation}
\label{eq:cxtyineq}
\begin{aligned}
\mathcal{C}\widetilde W(M|b) \leq
\mathcal{QC}\overline W(M,b) \leq \mathcal{Q}\widetilde W (M|b)\leq \widetilde
W(M|b) 
\end{aligned}
\end{equation}
for all \((M,b) \in \RR^{m\times
n}
\times
\RR^m\), where \(\mathcal{C}\widetilde W\) and \(\mathcal{Q}\widetilde
W\) are the convex and  quasiconvex
envelopes, respectively, of \(\widetilde
W\). 
In this paper, we do not
distinguish  \(\overline W\)
and \(\widetilde W\); in particular,
we  write \(\mathcal{QC}\widetilde W\) in place of \(\mathcal{QC}\overline W\).

We further observe that if \(n=1\),
then any cross-quasiconvex-convex is
convex; to see this, it suffices to
use \eqref{defcqc} with
\(M=\lambda M_1 + (1-\lambda) M_2\), \(b=\lambda b_1 + (1-\lambda) b_2\),
\begin{equation*}
\begin{aligned}
\theta(x):= \begin{cases}
(-\lambda M_1 + \lambda M_2)(x-1) &
\text{if } \lambda \leq x \leq 1\\
((1-\lambda) M_1 - (1-\lambda) M_2)x &
\text{if } 0 \leq x \leq \lambda,
\end{cases}  
\end{aligned}  
\end{equation*}
and \vspace{-1mm}
\begin{equation*}
\begin{aligned}
  \eta(x):= \begin{cases}
-\lambda b_1 + \lambda b_2 &
\text{if } \lambda \leq x \leq 1\\
(1-\lambda) b_1 - (1-\lambda) b_2
&
\text{if } 0 \leq x \leq \lambda,
\end{cases}
\end{aligned}  
\end{equation*}
for \((M_1,b_1), \, (M_2,b_2) \in \RR^{m\times 1}
\times
\RR^l\) and \(\lambda \in (0,1)\). The
converse implication follows  by Jensen's
inequality. Thus, for all \((M,b) \in \RR^{m\times
1}
\times
\RR^l\), we have \vspace{-.5mm}
\begin{equation}
\label{eq:cxty1d}
\begin{aligned}
\mathcal{C}\overline W(M,b) =
\mathcal{QC}\overline W(M,b).
\end{aligned}
\end{equation}

\begin{remark}\label{ccxqcx}
Given a Borel function   \(\widehat W: \RR^{d\times l} \times
\RR^m \to \RR\), we may associate  
 the function \(\overline W: \RR^{m\times
1}
\times
\RR^{dl}\) defined for \((M,b)\in \RR^{m\times
1}
\times
\RR^{dl} \), with \(b=(b_{11},...,b_{1l},...,
b_{d1},...,b_{dl})\), by \(\overline W(M,b):= \widehat W(\hat b, M)\), where \(\hat
b:= (b_{ij})_{1\leq i \leq d \atop 1\leq j \leq l} \in \RR^{d\times
l}\). Then, in view of \eqref{eq:cxty1d},
%
\(\mathcal{C}\widehat W(\hat b,M) = \mathcal{C}\overline W(M,b) =
\mathcal{QC}\overline W(M,b).\)
 
 In this paper, we will use  this remark
with \(d=m=3\), \(l=2\), and \(\widehat
W(M_\alpha,M_3):= W(M_\alpha|M_3)\)
for \(M=(M_\alpha|M_3)\in\RR^{3\times3}\).
Note that \(\mathcal{C} \widehat W(M_\alpha,M_3)= \mathcal{C} W(M_\alpha|M_3)\).\end{remark}

Finally, we recall the definition of
the function \(Q^\ast W\) introduced in \cite{BFM03}  to describe bending phenomena in thin plates: \vspace{-1.5mm} 
\begin{equation*}
\begin{aligned}
\mathcal{Q}^\ast W(M_\alpha| M_3) := \inf \bigg\{&
 \int_{(0,1)^3}W(M_\alpha + \nabla_\alpha \varphi(x)|\lambda \nabla_3 \varphi(x))\,\dx\!:\, \,\lambda\in \mathbb
R, \,\varphi \in W^{1,p}((0,1)^3;\mathbb R^3),
\\ 
 &\, \ffi(\cdot,x_3) \text{ is } (0,1)^2
 \text{-periodic for \aev\ \(x_3\) in
 (0,1)},\,    \lambda \int_{-{1}/{2}}^{{1}/{2}}\nabla_3
\varphi(x)\,\dx_3= M_3 \bigg\}
\end{aligned}
\end{equation*}
for \(M=(M_\alpha|M_3)\in
\mathbb R^{3\times 3}\). As proved in
 \cite[Proposition~A]{BFM09}, for all \(M=(M_\alpha|M_3)\in
\mathbb R^{3\times 3}\), we have 
\begin{equation}
\label{eq:Qast=Qhat}
\begin{aligned}
\mathcal{Q}^\ast W (M_\alpha|M_3)= \qce (M_\alpha|M_3).
\end{aligned}
\end{equation}

The following lemma allows us to characterize the accumulation
points of a diagonal infimizing sequence for the sequence
of problems \eqref{Pepsi}. Its proof is very similar to that of \cite[Proposition~2.1]{GGLM02}, for which
reason we will only highlight the necessary
modifications. We first introduce some
notation. 

Let \vspace{-.5mm}
\begin{equation}
\label{Aell+}
\begin{aligned}
\mathcal{A}_{l^+}^p:= \big\{ (\psi^a,
\psi^b) &\in W^{1,p}(\Omega^a;\R^3)
\times
 W^{1,p}(\Omega^b;\R^3)\!:\,
 \psi^a \hbox{ is  independent of }
x_\alpha, 
 \\
 &\quad \psi^b \hbox{ is  independent
of } x_3,
 \, \hbox{and for } p>2,\, \psi^a(0_3)=\psi^b(0_\alpha)
 \big\}
\end{aligned}
\end{equation}
and, for \(0<\epsi\leq 1\), let \vspace{-1mm}
\begin{equation}
\label{Aepsi}
\begin{aligned}
\mathcal{A}_\epsi:= \Big\{ ((\bar\bcal^a,\psi^a),
 ( \psi^b,\bar
\bcal^b)) &\in \big(L^p((0,L);\R^{3 \times
2}) \times W^{1,p}(\Omega^a;\R^3)\big)
\times
 \big(W^{1,p}(\Omega^b;\R^3)\times L^p(\omega^b;\R^{3})\big)\!:\,
\\
&\quad\frac{1}{r_{\varepsilon}} \int_{\omega^a}\nabla_\alpha
\psi^a(x_\alpha,\cdot)\, \dx_\alpha
=\bar \bcal^a(\cdot),\enspace
\frac{1}{h_\epsi}\int_{-1}^{0}\nabla_3
\psi^b(\cdot,x_3)\,\dx_3
=\bar \bcal^b(\cdot), \\
&\quad \psi^a(x_\alpha,0_3) = \psi^b(r_\epsi
x_\alpha,0_3) \hbox{ for \aev\ } x_\alpha\in
\omega^a
 \Big\}.
\end{aligned}
\end{equation}

\begin{lemma}\label{lem:l21GGLM}
Let $(\psi^a_\epsi)_{\epsi>0}\subset W^{1,p}(\Omega^a;\RR^3)$
and $(\psi^b_\epsi)_{\epsi>0}\subset W^{1,p}(\Omega^b;\RR^3)$  be such that
\vspace{-.5mm} %
\begin{equation*}
\begin{aligned}
& \sup_{\epsi>0} \|\psi^a_\epsi\|_{W^{1,p}(\Omega^a;\RR^3)}<\infty,
\quad \sup_{\epsi>0} \|r_{\epsi}^{-1}\grad_\alpha \psi^a_\epsi
\|_{L^p(\Omega^a;\RR^{3\times 2})}<\infty,\\
& \sup_{\epsi>0} \|\psi^b_\epsi\|_{W^{1,p}(\Omega^b;\RR^3)}<\infty,
\quad \sup_{\epsi>0} \|h_{\epsi}^{-1}\grad_3 \psi^b_\epsi
\|_{L^p(\Omega^b;\RR^3)}<\infty.
\end{aligned}
\end{equation*}
Then, the sequences  $(r_{\epsi}^{-1}\grad_\alpha
\psi^a_\epsi, \psi^a_\epsi)_{\epsi>0} $
and $(\psi^b_\epsi, h_{\epsi}^{-1}\grad_3 \psi^b_\epsi)_{\epsi>0}$ are
sequentially, weakly compact in \( L^p(\Omega^a;\allowbreak
\RR^{3\times
2})\times W^{1,p}(\Omega^a;\RR^3)\) and \(W^{1,p}(\Omega^b;\RR^3)\times L^p(\Omega^b;\RR^3)\),
respectively. Moreover, let \((\bcal^a,\psi^a)\) and    \((\psi^b, \bcal^b)\) be corresponding accumulation points; that is,  let \(\epsi_j\preccurlyeq
\epsi\) be such that \(((r_{\epsi_j}^{-1}\grad_\alpha
\psi^a_{\epsi_j}, \psi^a_{\epsi_j}),(\psi^b_{\epsi_j}, h_{{\epsi_j}}^{-1}\grad_3
\psi^b_{\epsi_j}))_{j\in\Nn} \)  converges
 to \(((\bcal^a,\psi^a),(\psi^b, \bcal^b))\) weakly in \(\big(L^p(\Omega^a;\RR^{3\times 2}) \times
W^{1,p}(\Omega^a;\RR^3)\big) \times
\big(W^{1,p}(\Omega^b;\RR^3)\times
L^p(\Omega^b;\RR^3) \big)\). Then,  \(\psi^a\)  is  independent of \( x_\alpha\),
\(\psi^b\)  is  independent of \( x_3\),  there exist  \(v^a\in L^p((0,L);W^{1,p}(\omega^a;\R^3)\cap
L^p_0(\omega^a;\R^3))\) and \(v^b\in L^p(\omega^b;W^{1,p}((-1,0);\R^3)\cap
L^p_0((-1,0);\R^3))\) such that \(\bcal^a = \grad_\alpha
v^a\) and \(\bcal^b = \grad_3
v^b\), and  \vspace{-1.2mm}
\begin{equation}
\label{eq:as24rod}
\lim_{j \to \infty}\int_{\omega^a}\psi^a_{\epsi_j}(x_\alpha,
0_3)\,\dx_\alpha=|\omega^a|\psi^a(0_3).
\end{equation}
Furthermore,
recalling \eqref{ell},
\begin{itemize}
\leftskip=-1em
\item [i)]
 if \(\ell\in\RR^+\)
or \(\ell=0\)  and if
\(p>2\), then we may extract
a
further  subsequence, \((\psi^a_{\epsi_{j_k}},
\psi^b_{\epsi_{j_k}} )_{k\in\NN}\),  for
which \vspace{-2.6mm}
\begin{equation}
\label{eq:as24GGLM}
\lim_{k\to \infty}\int_{\omega^a}\psi^b_{\epsi_{j_k}}
(r_{\varepsilon_{j_k}}
x_\alpha, 0_3)\,\dx_\alpha=|\omega^a|\psi^b(0_\alpha).
\end{equation}

\item [ii)] if \(\ell=\infty\),   
\(p>2\), and there exists a bounded sequence \((d_\epsi)_{\epsi>0}\) in
\(L^\infty(\Omega^b;\RR^3)\) such that
\( \Vert {h_\epsi}^{-1} \grad_3 \psi^b_\epsi
- d_\epsi\Vert ^p_{L^p(\Omega^b;\RR^3)}
\leq C \tfrac{r_\epsi^2}{h_\epsi}\), then we may extract
a
further  subsequence, \((\psi^a_{\epsi_{j_k}},
\psi^b_{\epsi_{j_k}} )_{k\in\NN}\),
 satisfying \eqref{eq:as24GGLM}.
\end{itemize}
In particular, if in addition to i)
or
ii),  we have \(((\bar\bcal^a_\epsi,\psi^a_\epsi),
 (\psi^b_\epsi,\bar
\bcal^b_\epsi))\in \A_\epsi \) for all
\(\epsi>0\), then
\((\psi^a, \psi^b) \in\A^p_{\ell_+}\),
 \( \bar\bcal^a_{\epsi_j}
\weakly \bar\bcal^a\) weakly in  \(L^p((0,L);\R^{3
\times 2})\), and \( \bar\bcal^b_{\epsi_j}
\weakly
\bar\bcal^b\) weakly in  \(L^p(\omega^b;\R^{3})\),
where \(\bar\bcal^a:=\int_{\omega^a}\bcal^a\,
\dx_\alpha\) and \(\bar\bcal^b:=\int_{-1}^{0}\bcal^b\,\dx_3
\).
\end{lemma}

\begin{proof} The proof regarding the
\(\ell\in\RR^+\) case  can be found in
 \cite[Proposition~2.1]{GGLM02}. Note that, independently of the value of \(\ell\), \eqref{eq:as24rod} follows
from the continuity of the trace with
respect to the weak convergence in \(W^{1,p}\).
We observe further that the arguments
in \cite[Proposition~2.1]{GGLM02} remain
valid for \(\ell=0\) (see \cite[(2.9) with \(N=3\)]{GGLM02}).
The \(\ell=\infty\) case also can  be treated 
as in \cite[Proposition~2.1]{GGLM02}  with the exception of the proof of  \cite[(2.9)]{GGLM02}. Precisely,
 we are left to prove that \vspace{-.5mm}
\begin{equation}
\label{eq:as2.10}
\begin{aligned}
\lim_{k\to\infty} \bigg| 
\int_{\omega^a}\big(\psi^b_{\epsi_{j_k}}
(r_{\varepsilon_{j_k}}
x_\alpha, 0)- \psi^b_{\epsi_{j_k}}
(r_{\varepsilon_{j_k}}
x_\alpha, \bar x_3)\big) \,\dx_\alpha \bigg| =0, 
\end{aligned}
\end{equation}
where \(\bar x_3 \) is a certain fixed
point in \((-1,0)\) (see  \cite[(2.5)]{GGLM02}).
To show  \eqref{eq:as2.10}, let \((d_\epsi)_{\epsi>0}\) be as in
\textit{ii)} and recall that \(\bar a =|\omega^a|\). Using H\"older's inequality and a change of variables,
we obtain \vspace{-.mm}
\begin{equation*}
\begin{aligned}
&\bigg| 
\int_{\omega^a}\big(\psi^b_{\epsi_{j_k}}
(r_{\varepsilon_{j_k}}
x_\alpha, 0_3)- \psi^b_{\epsi_{j_k}}
(r_{\varepsilon_{j_k}}
x_\alpha, \bar x_3)\big) \,\dx_\alpha
\bigg| \\
&\quad
= h_{\epsi_{j_k}}\bigg| 
\int_{\omega^a} \int_{\bar x_3}^0\Big[\big
(\tfrac1{h_{\epsi_k}}
\grad_3 \psi^b_{\epsi_{j_k}}
(r_{\varepsilon_{j_k}}
x_\alpha, x_3) - d_{\epsi_k}(r_{\varepsilon_{j_k}}
x_\alpha, x_3)\big )+d_{\epsi_k}(r_{\varepsilon_{j_k}}
x_\alpha, x_3)\Big] \,\dx_3\dx_\alpha
\bigg| \\
&\quad \leq\bar a^{\tfrac{p-1}{p}}  h_{\epsi_{j_k}}  \bigg( \frac{1}{r_{\epsi_{j_k}}^2}
\int_{r_{\varepsilon_{j_k}}\omega^a} \int_{-1}^0\big
|\tfrac1{h_{\epsi_{j_k}}}
\grad_3 \psi^b_{\epsi_{j_k}}
(
x_\alpha, x_3) - d_{\epsi_k}(
x_\alpha, x_3) \big|^p\,\dx_\alpha
\dx_3\bigg)^{\tfrac1p}+\bar ah_{\epsi_{j_k}} \Vert
d_{\epsi_k}\Vert_{\infty}\\
&\quad \leq C^{\tfrac1p}\bar a^{\tfrac{p-1}{p}}
h_{\epsi_{j_k}}^{\tfrac{p-1}{p}} +\bar ah_{\epsi_{j_k}}
\Vert
d_{\epsi_k}\Vert_{\infty},
\end{aligned}
\end{equation*}
from which \eqref{eq:as2.10} follows.
\end{proof}

\begin{remark}\label{onb's}
In view of Lemma~\ref{lem:l21GGLM},
we are led to investigate
whether the functions \(\bar\bcal^a\)
or \(
\bar\bcal^b\) in the limit problems, \eqref{Pell+}, \eqref{Pelli}, or \eqref{Pellz},
 belong
to a strict subspace of \(L^p((0,L);\R^{3
\times 2} )\) or \( L^p(\omega^b;\R^{3})
\), respectively. However,
it can be easily checked that \(\{\bar
\bcal^a\in L^p((0,L);\R^{3 \times 2})\!:
\,
\bar \bcal^a
= \int_{\omega^a} \grad_\alpha v^a\,\dx_\alpha
\hbox{ for some } v^a\in L^p((0,L);W^{1,p}(\omega^a;\R^3)\cap
L^p_0(\omega^a;\R^3) ) \}= L^p((0,L);\R^{3
\times 2}) \) and \(\{\bar \bcal^b\in
L^p(\omega^b;\R^{3})\!:
\, \bar \bcal^b
= \int_{-1}^0 \grad_3 v^b\,\dx_3 \allowbreak
\hbox{ for some } v^b\in L^p(\omega^b;W^{1,p}((-1,0);\R^3)\cap
L^p_0((-1,0);\R^3))
\}=L^p(\omega^b;\R^{3})\). Thus, the functions \(\bar\bcal^a\)
or \(
\bar\bcal^b\) in the limit problems
 are
indeed defined in the whole space \(L^p((0,L);\R^{3
\times 2} ) \) or \( L^p(\omega^b;\R^{3})
\), respectively. 
\end{remark}

In some proofs, to gain  regularity
regarding the integrand function, it will be convenient
to replace \(W\) by its quasiconvex
envelope, \(\mathcal{Q} W\). The next
lemma will enable us to do so without
loss of generality.

\begin{lemma}\label{lem:WbyQW}
Let  \(W:\RR^{3\times 3} \to \RR\)
be a Borel function satisfying \eqref{pgrowth}. Let \((\epsi_n)_{n\in\Nb}\) be a sequence
of positive
numbers convergent to zero, and let
\((\ell^a_n)_{n\in\Nn}\) and \((\ell^b_n)_{n\in\Nn}\)
be two sequences of positive numbers for which the corresponding limits
exist in \((0,\infty]\). Recall  \eqref{Fabepsi} and
let \vspace{-.5mm}
\begin{equation*}
\begin{aligned}
\calG_{\epsi_n}^a(\psi^a):= \int_{\Omega^a}
\qe({r_{\epsi_n}^{-1}}\nabla_\alpha
\psi^a|\nabla_3\psi^a)\,\dx \quad \hbox{and}
\quad %
\calG_{\epsi_n}^b(\psi^b):= \int_{\Omega^b}
\qe(\nabla_\alpha \psi^b|{h_{\epsi_n}^{-1}}
\nabla_3\psi^b)\,\dx.
\end{aligned}
\end{equation*}
Then, for
 all \(\big((\bar\bcal^a,\psi^a),
(\psi^b, \bar\bcal^b)\big) \in
\big(L^p((0,L);\R^{3 \times 2}) \times
W^{1,p}(\Omega^a;\R^3)\big)\times
\big( W^{1,p}(\Omega^b;\R^3)
\times L^p(\omega^b;\R^{3}) \big)  \),
we have
 \vspace{-1mm}
\begin{equation*}
\begin{aligned}
F^-((\bar\bcal^a,\psi^a), (\psi^b,
\bar\bcal^b)) = \calG((\bar\bcal^a,\psi^a), (\psi^b,
\bar\bcal^b)),
\end{aligned}
\end{equation*}
where
\vspace{-1mm} 
\begin{equation}
\label{F-}
\begin{aligned}
&F^-((\bar\bcal^a,\psi^a), (\psi^b,
\bar\bcal^b))\\
&\qquad:=\inf\Big\{\liminf_{n\to\infty}
\big(\ell^a_nF_{\epsi_n}^a (\psi^a_n) + \ell^b_nF_{\epsi_n}^b
(\psi^b_n)
\big)\!:\, ((\bar\bcal^a_n,\psi^a_n),(
\psi^b_n, \bar\bcal^b_n)
)\in \mathcal{A}_{\epsi_n} \hbox{ for
all } n\in\Nb,\\
& \hskip30mm \psi^a_n \weakly \psi^a
\hbox{ weakly in
} W^{1,p}(\Omega^a;\R^3),
\, 
 \bar\bcal^a_n \weakly \bar\bcal^a \hbox{
weakly in }
L^{p}((0,L);\R^{3\times
 2}),\\
& \hskip30mm \psi^b_n \weakly \psi^b
\hbox{ weakly in
} W^{1,p}(\Omega^b;\R^3),
\, 
 \bar\bcal^b_n \weakly \bar\bcal^b \hbox{
weakly in }
L^{p}(\omega^b;\R^3)\Big\}
\end{aligned}
\end{equation}
and \vspace{-1mm}
\begin{equation}
\label{GQW}
\begin{aligned}
&\calG((\bar\bcal^a,\psi^a), (\psi^b,
\bar\bcal^b))\\
&\qquad:=\inf\Big\{\liminf_{n\to\infty}
\big(\ell^a_n\calG_{\epsi_n}^a (\psi^a_n) +
\ell^b_n
\calG_{\epsi_n}^b (\psi^b_n)
\big)\!:\, ((\bar\bcal^a_n,\psi^a_n),(
\psi^b_n, \bar\bcal^b_n)
)\in \mathcal{A}_{\epsi_n} \hbox{ for
all } n\in\Nb,\\
& \hskip30mm \psi^a_n \weakly \psi^a
\hbox{ weakly in
} W^{1,p}(\Omega^a;\R^3),
\, 
 \bar\bcal^a_n \weakly \bar\bcal^a \hbox{
weakly in } 
L^{p}((0,L);\R^{3\times
 2}),\\
& \hskip30mm \psi^b_n \weakly \psi^b
\hbox{ weakly in
} W^{1,p}(\Omega^b;\R^3),
\, 
 \bar\bcal^b_n \weakly \bar\bcal^b \hbox{
weakly in }
L^{p}(\omega^b;\R^3)\Big\}.
\end{aligned}
\end{equation}
\end{lemma}

\begin{remark}\label{coefWbyQW}
In the \(\ell\in\RR^+\) and \(\ell=\infty\)
cases, we will use Lemma~\ref{lem:WbyQW}
with \(\ell_n^a=1\) and \(\ell_n^b = {h_{\epsi_n}}/
{r_{\epsi_n}^2}\) for all \(n\in\Nn\);
in the \(\ell=0\)
case, we will take \(\ell_n^a={r_{\epsi_n}^2}/
{h_{\epsi_n}}\) and \(\ell_n^b =
1\) for all \(n\in\Nn\).
\end{remark}

 \begin{proof}[Proof of Lemma~\ref{lem:WbyQW}]
We start by observing that if \(W\) satisfies \eqref{pgrowth}
or \eqref{coercrigid2}, then so does
\(\mathcal{Q}W\).

Because \(\qe \leq W\), the inequality
\(\calG \leq F^-\)
holds. To prove the converse inequality,
we will proceed
in several steps.

\textit{Step 1.} In this step, we prove
that for all
\(M=(M_\alpha|M_3) \in \R^{3\times 3}\)
and \(r,\,h>0\),
we have \vspace{-1mm}
\begin{equation}
\label{eq:QWepsi}
\begin{aligned}
\qe_r^a(M) = (\qe)_{r}^a(M) \quad \text{and}
\quad 
\qe_{h}^b(M) = (\qe)_{h}^b(M), 
\end{aligned}
\end{equation}
where
%
\(\tilde W_{r}^a(M) := \tilde W({r^{-1}}
M_\alpha|M_3)\) and
\(\tilde W_{h}^b(M)
:= \tilde W(M_\alpha|{h^{-1}}
M_3).\)
%

The proof of the second identity in
\eqref{eq:QWepsi}
can be found in \cite[Proposition~1.1]{BFM03}.
The first
identity in \eqref{eq:QWepsi} can be
proved similarly.


\textit{Step~2.} In this step, we show
that for fixed
\(n\in\Nb\) and for every
\(((\bar\bcal^a,\psi^a),(
\psi^b, \bar\bcal^b)
)\in \mathcal{A}_{\epsi_n}\), we can
find a sequence
\(((\bar\bcal^a_j,\psi^a_j),(
\psi^b_j, \bar\bcal^b_j)
)_{j\in\Nb}\subset \mathcal{A}_{\epsi_n}\)
such that
\(\psi^a_j \weakly \psi^a\) weakly in
 \(W^{1,p}(\Omega^a;\R^3)\), \( 
 \bar\bcal^a_j \weakly \bar\bcal^a\)
weakly in \(L^{p}((0,L);\R^{3\times
 2})\), \(\psi^b_j \weakly \psi^b\)
weakly in
 \(W^{1,p}(\Omega^b;\R^3)\),  \( 
 \bar\bcal^b_j \weakly \bar\bcal^b\)
weakly in \(L^{p}(\omega^b;\R^3)\),
 and \vspace{-1.5mm}
\begin{equation*}
\begin{aligned}
&\lim_{j\to\infty}  \int_{\Omega^a}
W({r_{\epsi_n}^{-1}}\nabla_\alpha
\psi^a_j|\nabla_3\psi^a_j)\,\dx =  \int_{\Omega^a}
\qe({r_{\epsi_n}^{-1}}\nabla_\alpha
\psi^a|\nabla_3\psi^a)\,\dx,\\
& \lim_{j\to\infty} \int_{\Omega^b}
W(\nabla_\alpha \psi^b_j|{h_{\epsi_n}^{-1}}\nabla_3\psi^b_j)\,\dx
= \int_{\Omega^b}
\qe(\nabla_\alpha \psi^b|{h_{\epsi_n}^{-1}}\nabla_3\psi^b)\,\dx.
\end{aligned}
\end{equation*}

We first observe  that
the
lower bound in \eqref{pgrowth} allows
us to  assume that \(W\geq0\) without loss
of generality. 
Invoking \eqref{pgrowth} once more,
\eqref{eq:QWepsi},  the relaxation
result
in \cite{AcFu84}, and  the
decomposition lemma \cite[Lemma~1.2]{FMP98},
we can find 
sequences \((\psi^a_k)_{k\in\Nb}\) and
\((\psi^b_k)_{k\in\Nb}\)
such that 
\(\psi^a_k \weakly \psi^a\) weakly in
 \(W^{1,p}(\Omega^a;\R^3)\),  \(\psi^b_k
\weakly \psi^b\)
weakly in
 \(W^{1,p}(\Omega^b;\R^3)\), \((|\grad
\psi^a_k|^p)_{k\in\Nb}
 \subset L^1(\Omega^a)\) and \((|\grad\psi^b_k|^p)_{k\in\Nb}
 \subset L^1(\Omega^b)\) are equi-integrable,
and \vspace{-1.5mm}
\begin{equation}\label{eq:relaxanddecomp}
\begin{aligned}
&\lim_{k\to\infty}  \int_{\Omega^a}
W({r_{\epsi_n}^{-1}}\nabla_\alpha
\psi^a_k|\nabla_3\psi^a_k)\,\dx =  \int_{\Omega^a}
\qe({r_{\epsi_n}^{-1}}\nabla_\alpha
\psi^a|\nabla_3\psi^a)\,\dx,\\
& \lim_{k\to\infty} \int_{\Omega^b}
W(\nabla_\alpha \psi^b_k|{h_{\epsi_n}^{-1}}\nabla_3\psi^b_k)\,\dx
= \int_{\Omega^b}
\qe(\nabla_\alpha \psi^b|{h_{\epsi_n}^{-1}}\nabla_3\psi^b)\,\dx.
\end{aligned}
\end{equation}
In particular, because \(((\bar\bcal^a,\psi^a),(
\psi^b, \bar\bcal^b)
)\in \mathcal{A}_{\epsi_n}\), we have \vspace{-.5mm}
\begin{equation*}
\begin{aligned}
&\bar\bcal^a_k:= \frac{1}{r_{\varepsilon_n}}
\int_{\omega^a}\nabla_\alpha
\psi^a_k(x_\alpha,\cdot)\, \dx_\alpha
\weakly_k \bar\bcal^a
\hbox{ weakly in } L^{p}((0,L);\R^{3\times
 2}),\\
& 
 \bar\bcal^b_k := \frac{1}{h_{\epsi_n}}\int_{-1}^{0}\nabla_3
\psi^b_k(\cdot,x_3)\,\dx_3
\weakly_k \bar\bcal^b \hbox{ weakly
in } L^{p}(\omega^b;\R^3).
\end{aligned}
\end{equation*}
To construct sequences that also
 satisfy 
the condition \eqref{junction}, we 
use the slicing method.
 Fix \(\tau>0\); because \((1+{r_{\epsi_n}^{-p}}|\nabla_\alpha
\psi^a|^p + |\nabla_3\psi^a|+{r_{\epsi_n}^{-p}}|\nabla_\alpha
\psi^a_k|^p + |\nabla_3\psi^a_k|^p)_{k\in\Nb}\)
and \((1+
|\nabla_\alpha \psi^b|^p + {h_{\epsi_n}^{-p}}|\nabla_3\psi^b|^p+
|\nabla_\alpha \psi^b_k|^p + {h_{\epsi_n}^{-p}}|\nabla_3\psi^b_k|^p)_{k\in\Nb}\)
are equi-integrable, there exists  \(\epsilon\in
(0,\tau)\)
 such that, for measurable   \(E\subset\R^3\) with \(|E|<\epsilon\), \vspace{-1.5mm} 
\begin{equation}
\label{eq:byequi}
\begin{aligned}
&\sup_{k\in\Nb}\bigg\{\int_{\Omega^a\cap
E} (1+{r_{\epsi_n}^{-p}}|\nabla_\alpha
\psi^a|^p + |\nabla_3\psi^a|+{r_{\epsi_n}^{-p}}|\nabla_\alpha
\psi^a_k|^p + |\nabla_3\psi^a_k|^p)\,\dx
\\ &\qquad+
\int_{\Omega^b\cap E} (1+
|\nabla_\alpha \psi^b|^p + {h_{\epsi_n}^{-p}}|\nabla_3\psi^b|^p+|\nabla_\alpha
\psi^b_k|^p + {h_{\epsi_n}^{-p}}|\nabla_3\psi^b_k|^p)\,\dx\bigg\}
< \tau. \end{aligned}
\end{equation}
For \(j\in \Nb\), fix \(\delta_j \in
\big(0, \tfrac{\epsilon}{|\omega^a|
+ |\omega^b|}\big)\) such that  \(\delta_j\to0^+\)
as
\(j\to\infty\), and let \(\phi_j \in
C^\infty(\R;[0,1])\)
be a smooth cut-off function such that
\(\phi_j(t)
= 0\)
for \(t\in (-\delta_j,\delta_j)\), 
\(\phi_j(t) =
1\) for \(|t|\geq 2\delta_j\), and
 \(\Vert\phi_j'\Vert_\infty \leq
2/{\delta_j}\).
Define, for \(k,\,j\in\Nb\), 
\begin{equation*}
\begin{aligned}
&\psi_{k,j}^a(x):= \phi_j(x_3)\psi^a_k(x)
+
(1-\phi_j(x_3))\psi^a(x), \quad x\in
\Omega^a,\\
&\bar\bcal^a_{k,j}(x_3):= \frac{1}{r_{\varepsilon_n}}
\int_{\omega^a}\nabla_\alpha
\psi^a_{k,j}(x_\alpha,x_3)\, \dx_\alpha,
\quad
x_3\in (0,L),\\
&\psi_{k,j}^b(x):= \phi_j(x_3)\psi^b_k(x)
+
(1-\phi_j(x_3))\psi^b(x), \quad x\in
\Omega^b,\\
&  \bar\bcal^b_{k,j} (x_\alpha):= \frac{1}{h_{\epsi_n}}\int_{-1}^{0}\nabla_3
\psi^b_{k,j}(x_\alpha,x_3)\,\dx_3, \quad
x_\alpha
\in \omega^b. \end{aligned}
\end{equation*}
Fix \(j\in\Nn\). It can be checked that
\begin{equation}
\label{eq:doublelimseq}
\begin{aligned}
&\psi^a_{k,j} \weakly_k \psi^a \hbox{
weakly in } W^{1,p}(\Omega^a;\R^3),\enspace
\bar\bcal^a_{k,j} \weakly_k \bar\bcal^a
\hbox{ weakly
in } L^{p}((0,L);\R^{3\times
 2}),\\
& \psi^b_{k,j} \weakly_k \psi^b \hbox{
weakly in } W^{1,p}(\Omega^b;\R^3),\enspace
 \bar\bcal^b_{k,j} \weakly_k \bar\bcal^b\hbox{
weakly
in } L^{p}(\omega^b;\R^3).
\end{aligned}
\end{equation}
Also, because \(\psi^a(x_\alpha,0_3) =
\psi^b(r_{\epsi_n}
x_\alpha, 0_3)\) for \aev\ \(x_\alpha
\in \omega^a\) and
\(\phi_j(0)=0\), we have \(\psi^a_{k,j}(x_\alpha,0_3)
=
\psi^b_{k,j}(r_{\epsi_n}
x_\alpha, 0_3)\) for \aev\ \(x_\alpha
\in \omega^a\); thus,
for all \(k\in\Nb\), \vspace{-1mm}
\begin{equation}
\label{eq:doubleseqAepsi}
\begin{aligned}
((\bar\bcal^a_{k,j},\psi^a_{k,j}),(
\psi^b_{k,j}, \bar\bcal^b_{k,j}))
\in \mathcal{A}_{\epsi_n}.
\end{aligned}
\end{equation}
 Moreover, in view of \eqref{pgrowth},
\eqref{eq:byequi},
and \(W\geq0\), \vspace{-.5mm}
\begin{equation*}
\begin{aligned}
&\int_{\Omega^a} W({r_{\epsi_n}^{-1}}\nabla_\alpha
\psi^a_{k,j}|\nabla_3\psi^a_{k,j})\,\dx
\\
&\quad =\int_0^{\delta_j}\int_{\omega^a}
W({r_{\epsi_n}^{-1}}\nabla_\alpha
\psi^a|\nabla_3\psi^a)\,\dx_\alpha\dx_3
+ \int_{\delta_j}^{2\delta_j}\int_{\omega^a}
W({r_{\epsi_n}^{-1}}\nabla_\alpha
\psi^a_{k,j}|\nabla_3\psi^a_{k,j})\,\dx_\alpha\dx_3\\
&\quad\quad + \int_{2\delta_j}^L\int_{\omega^a}
W({r_{\epsi_n}^{-1}}\nabla_\alpha
\psi^a_k|\nabla_3\psi^a_k)\,\dx_\alpha\dx_3\\
&\quad \leq C\tau +\ \frac{C}{\delta_j^p}
\int_{\Omega^a}
|\psi^a_k -\psi^a|^p\,\dx +\int_{\Omega^a}
W({r_{\epsi_n}^{-1}}\nabla_\alpha
\psi^a_k|\nabla_3\psi^a_k)\,\dx 
\end{aligned}
\end{equation*}
for some constant \(C\)  only depending 
on the constant
in \eqref{pgrowth} and on \(p\). Similarly,
\vspace{-.5mm}
\begin{equation*}
\begin{aligned}
&\int_{\Omega^b} W(\nabla_\alpha
\psi^b_{k,j}|{h_{\epsi_n}^{-1}}\nabla_3\psi^b_{k,j})\,\dx\leq
C\tau +\ \frac{C}{\delta^p_j} \int_{\Omega^b}
|\psi^b_k -\psi^b|^p\,\dx +\int_{\Omega^b}
W(\nabla_\alpha
\psi^b_k|{h_{\epsi_n}^{-1}}\nabla_3\psi^b_k)\,\dx.
\end{aligned}
\end{equation*}
Letting \(k\to\infty\) first, then \(j\to\infty\),
and
finally \(\tau\to0^+\) in the two last
estimates and using \eqref{eq:relaxanddecomp}, we conclude
that \vspace{-1mm}
\begin{equation}
\label{eq:almostQW}
\begin{aligned}
&\limsup_{j\to\infty} \limsup_{k\to\infty}
\int_{\Omega^a}
W({r_{\epsi_n}^{-1}}\nabla_\alpha
\psi^a_{k,j}|\nabla_3\psi^a_{k,j})\,\dx
\leq\int_{\Omega^a}
\qe({r_{\epsi_n}^{-1}}\nabla_\alpha
\psi^a|\nabla_3\psi^a)\,\dx,\\
&\limsup_{j\to\infty} \limsup_{k\to\infty}
\int_{\Omega^b}
W(\nabla_\alpha
\psi^b_{k,j}|{h_{\epsi_n}^{-1}}\nabla_3\psi^b_{k,j})\,\dx
\leq \int_{\Omega^b}
\qe(\nabla_\alpha \psi^b|{h_{\epsi_n}^{-1}}\nabla_3\psi^b)\,\dx.
\end{aligned}
\end{equation}

In view of  \eqref{eq:doublelimseq},
\eqref{eq:doubleseqAepsi},
\eqref{eq:almostQW},   the
metrizability of the weak convergence
on bounded sets
together with \eqref{pgrowth},
and invoking once more the relaxation
result 
in \cite{AcFu84} together with \eqref{eq:QWepsi},
 we can
find a subsequence \(k_j\prec k\) such
that \(((\tilde{\bar\bcal}^a_j,
\tilde\psi^a_j),(\tilde
\psi^b_j,\tilde {\bar\bcal}^b_j)):=
((\bar\bcal^a_{k_j,j},\psi^a_{k_j,j}),(
\psi^b_{k_j,j}, \bar\bcal^b_{k_j,j}))\),
\(j\in\Nb\),
satisfies
the requirements stated in Step~2.


\textit{Step~3.} In this step, we prove
that  \(\calG \geq
F^-\).

Let   \(((\bar\bcal^a,\psi^a),
(\psi^b, \bar\bcal^b)) \in
\big(L^p((0,L);\R^{3 \times 2}) \times
W^{1,p}(\Omega^a;\R^3)\big)\times
\big( W^{1,p}(\Omega^b;\R^3)
\times L^p(\omega^b;\R^{3}) \big) \)
be such that \(\calG((\bar\bcal^a,\psi^a),
(\psi^b, \bar\bcal^b)) <\infty\). Fix
\(\delta>0\), and
for
each \(n\in\Nb\), let \(((\bar\bcal^a_n,\psi^a_n),(
\psi^b_n, \bar\bcal^b_n)
)\in \mathcal{A}_{\epsi_n}\) be such
that  \( \psi^a_n
\weakly \psi^a\)  weakly in
\(W^{1,p}(\Omega^a;\R^3)\), \( \bar\bcal^a_n
\weakly
\bar\bcal^a\)
weakly in \(L^{p}((0,L);\R^{3\times
 2})\), \( \psi^b_n \weakly \psi^b\)
weakly in \(W^{1,p}(\Omega^b;\R^3)\),
 \(  \bar\bcal^b_n \weakly \bar\bcal^b\)
weakly in \(L^{p}(\omega^b;\R^3)\),
 and \vspace{-1mm}
\begin{equation*}
\begin{aligned}
\calG((\bar\bcal^a,\psi^a),
(\psi^b, \bar\bcal^b)) + \delta \geq
\liminf _{n\to\infty}
\big(\ell^a_n\calG_{\epsi_n}^a (\psi^a_n) +
\ell^b_n\calG_{\epsi_n}^b (\psi^b_n)
\big).
 \end{aligned}
\end{equation*}
Let 
\(n_k \prec n\) be such that \vspace{-1mm}
\begin{equation*}
\begin{aligned}
\liminf _{n\to\infty}
\big(\ell^a_n\calG_{\epsi_n}^a (\psi^a_n)
+
\ell^b_n\calG_{\epsi_n}^b (\psi^b_n)
\big) =
\lim _{k\to\infty}
\Big(\ell^a_{n_k}\calG_{\epsi_{n_k}}^a (\psi^a_{n_k})
+\ell^b_{n_k}
\calG_{\epsi_{n_k}}^b (\psi^b_{n_k})\Big)
.
 \end{aligned}
\end{equation*}
Fix  \(k\in\Nb\). By Step~2, 
there exists
a sequence  \(((\bar\bcal^a_{n_k,j},\psi^a_{n_k,j}),(
\psi^b_{n_k,j}, \bar\bcal^b_{n_k,j})
)_{j\in\Nb}\subset \mathcal{A}_{\epsi_{n_k}}\)
 such
that  \( \psi^a_{n_k,j}
\weakly_j \psi^a_{n_k}\)  weakly in
\(W^{1,p}(\Omega^a;\R^3)\), \( \bar\bcal^a_{n_k,j}
\weakly_j
\bar\bcal^a_{n_k}\)
weakly in \(L^{p}((0,L);\R^{3\times
 2})\), \( \psi^b_{n_k,j} \weakly_j
\psi^b_{n_k}\) weakly
in \(W^{1,p}(\Omega^b;\R^3)\),
 \(  \bar\bcal^b_{n_k,j} \weakly_j \bar\bcal^b_{n_k}\)
weakly in \(L^{p}(\omega^b;\R^3)\),
 and \vspace{-.5mm}
\begin{equation*}
\begin{aligned}
&\lim_{j\to\infty} \int_{\Omega^a} W({r_{\epsi_{n_k}}^{-1}}\nabla_\alpha
\psi^a_{n_k,j}|\nabla_3\psi^a_{n_k,j})\,\dx
=\calG_{\epsi_{n_k}}^a
(\psi^a_{n_k}),\\
& \lim_{j\to\infty} \int_{\Omega^b}
W(\nabla_\alpha \psi^b_{n_k,j}|{h_{\epsi_{n_k}}^{-1}}\nabla_3\psi^b_{n_k,j})\,\dx
=\calG_{\epsi_{n_k}}^b (\psi^b_{n_k}).
\end{aligned}
\end{equation*}
Hence, \vspace{-1mm}
\begin{equation}
\label{eq:S3energies}
\begin{aligned}
\calG((\bar\bcal^a,\psi^a),
(\psi^b, \bar\bcal^b)) + \delta \geq
 \lim _{k\to\infty}
\lim_{j\to\infty} \big(\ell^a_{n_k}F_{\epsi_{n_k}}^a
(\psi^a_{n_k,j})
+ \ell^b_{n_k}F_{\epsi_{n_k}}^b (\psi^b_{n_k,j})
\big),
\end{aligned}
\end{equation}
\begin{equation}
\label{eq:S3convlp}
\begin{aligned}
\lim _{k\to\infty}
\lim_{j\to\infty} \Vert \psi^a_{n_k,j}
- \psi^a\Vert_{L^p(\Omega^a;\R^3)}
= 0, \quad \lim _{k\to\infty}
\lim_{j\to\infty} \Vert \psi^b_{n_k,j}
- \psi^b\Vert_{L^p(\Omega^b;\R^3)}
= 0,
\end{aligned}
\end{equation}
and
\begin{equation}
\label{eq:S3weakvonvlp}
\begin{aligned}
&  \lim _{k\to\infty}
\lim_{j\to\infty}  \bar\bcal^a_{n_k,j}
= \bar\bcal^a
\hbox{ with
respect to the weak convergence in }
L^{p}((0,L);\R^{3\times
 2}),\\
 & \lim _{k\to\infty}
\lim_{j\to\infty}  \bar\bcal^b_{n_k,j}
= \bar\bcal^b
\hbox{ with
respect to the weak convergence in }
L^{p}(\omega^b;\R^3).
\end{aligned}
\end{equation}

Next, we observe that without
loss of generality, we may assume  that
\(\inf_{k\in\Nn} \ell^a_{n_k}, \, \inf_{k\in\Nn} \ell^b_{n_k}\geq
c>0\). 
Hence, by \eqref{pgrowth}, \vspace{-.5mm}
\begin{equation}\label{lboundsWbyQW}
\begin{aligned}
&\frac1C\Vert ({r_{\epsi_{n_k}}^{-1}}\nabla_\alpha
\psi^a_{n_k,j}|\nabla_3\psi^a_{n_k,j})
\Vert^p_{L^p(\Omega^a;\R^{3\times3})}
+ \frac1C\Vert (\nabla_\alpha \psi^b_{n_k,j}
|{h_{\epsi_{n_k}}^{-1}}\nabla_3\psi^b_{n_k,j})
\Vert^p_{L^p(\Omega^b;\R^{3\times3})}\\
&\quad \leq  \frac1{c} \Big({\ell^a_{n_k}}F_{\epsi_{n_k}}^a
(\psi^a_{n_k,j})
+ \ell^b_{n_k}F_{\epsi_{n_k}}^b (\psi^b_{n_k,j})\Big) + C(|\Omega^a| + |\Omega^b|).
\end{aligned}
\end{equation}

Because the weak topology is metrizable
on bounded sets,
 \eqref{eq:S3energies}--\eqref{lboundsWbyQW}
yield the existence of a diagonal sequence
  $((\bar\bcal^a_{n_k,j_k},
\psi^a_{n_k,j_k}), \allowbreak (
\psi^b_{n_k,j_k}, \bar\bcal^b_{n_k,j_k})
)_{k\in\Nb}$  satisfying
\(((\bar\bcal^a_{n_k,j_k},\psi^a_{n_k,j_k}),
\allowbreak(
\psi^b_{n_k,j_k}, \bar\bcal^b_{n_k,j_k})
)\in \mathcal{A}_{\epsi_{n_k}}\) for
all \(k\in\Nb\),
  \( \psi^a_{n_k,j_k}
\weakly_k \psi^a\)  weakly in
\(W^{1,p}(\Omega^a;\R^3)\), \( \bar\bcal^a_{n_k,j_k}
\weakly_k
\bar\bcal^a\)
weakly in \(L^{p}((0,L);\R^{3\times
 2})\), \( \psi^b_{n_k,j_k} \weakly_k
\psi^b\) weakly
in \(W^{1,p}(\Omega^b;\R^3)\),
 \(  \bar\bcal^b_{n_k,j_k} \weakly_k
\bar\bcal^b\) weakly
in \(L^{p}(\omega^b;\R^3)\),
 and realizing the double limit on the
right-hand side
 of \eqref{eq:S3energies}. Thus,
\vspace{-.5mm}%
\begin{equation}
\label{eq:almostQW-W}
\begin{aligned}
\calG((\bar\bcal^a,\psi^a),
(\psi^b, \bar\bcal^b)) + \delta &\geq
\lim _{k\to\infty}
 \big(\ell^a_{n_k}F_{\epsi_{n_k}}^a (\psi^a_{n_k,j_k})
+ \ell^b_{n_k}F_{\epsi_{n_k}}^b (\psi^b_{n_k,j_k})
\big) \\
&\geq \liminf _{n\to\infty}
 \big(\ell^a_{n}F_{\epsi_{n}}^a (\tilde\psi^a_{n})
+ \ell^b_{n}F_{\epsi_{n}}^b (\tilde\psi^b_{n})\big)
\geq F^-((\bar\bcal^a,\psi^a),
(\psi^b, \bar\bcal^b)),
\end{aligned}
\end{equation}
where \(((\tilde{\bar\bcal}^a_n,\tilde\psi^a_n),(
\tilde\psi^b_n, \tilde{\bar\bcal}^b_n)
):=((\bar\bcal^a_{n_k,j_k},
\psi^a_{n_k,j_k}),  (
\psi^b_{n_k,j_k}, \bar\bcal^b_{n_k,j_k})
)\) if \(n=n_k\), and \(((\tilde{\bar\bcal}^a_n,\tilde\psi^a_n),(
\tilde\psi^b_n, \tilde{\bar\bcal}^b_n)
):=((\bar\bcal^a_n,\allowbreak\psi^a_n),  (
\psi^b_n, \bar\bcal^b_n)
)\) if \(n\not=n_k\). Letting \(\delta\to0^+\)
in \eqref{eq:almostQW-W},
we obtain \(\calG\geq F^-\). This concludes
Step~3, as well
as the proof of Lemma~\ref{lem:WbyQW}.
\end{proof}

We conclude this section with a quantitative
result regarding approximations of
the scaled gradients in this paper,
\((r_\epsi^{-1}\grad_\alpha|\grad_3)\)
and \((\grad_\alpha|h_\epsi^{-1}\grad_3)\),
by appropriate matrices as in  \cite{FJM06}. In
what follows,  \(A_1\) and \(A_2\)
are two \textit{strongly incompatible} matrices in
the sense of \cite{Mat92, CM}. 

We first observe that the quantitative
geometric rigidity theorems \cite[Theorem~3.1]{FJM02}
for the single-well case, \(\mathcal{K}=SO(n)\),
and \cite[Theorem~1.2]{CM} for double-well case, \(\mathcal{K}=SO(n)A_1
\cup SO(n)A_2\), both proved for \(p=2\), hold for any \(p\in(1,\infty)\).  The  \(\mathcal{K}=SO(n)\) case was proved in \cite[Section~2.4]{CS06},
while the \(\mathcal{K}=SO(n)A_1
\cup SO(n)A_2\) case in \cite{DLS}.
Precisely, the following result holds.

\begin{theorem}\label{thm:GRbyFJM}
Let \(n\geq2\) and \(U\subset\RR^n\)
be a bounded Lipschitz domain. Assume
that   \(p\in(1,\infty)\) and that either  
 \(\mathcal{K}=SO(n)\) or  \(\mathcal{K}=SO(n)A_1
\cup SO(n)A_2\), where  \(A_1, \, A_2 \in \RR^{n\times n}\)
are   \textit{strongly incompatible}. Then,
there exists a positive constant, \(C_{U,p,\mathcal{K}}\),
such that for all \(v\in W^{1,p}(U;\RR^n)\),
we can find  \(M\in
\mathcal{K}\) satisfying
\begin{equation*}
\begin{aligned}
\Vert \grad v - M\Vert_{L^p(U;\RR^{n\times
n})} \leq C_{U,p,\mathcal{K}} \Vert \dist (\grad
v, \mathcal{K})\Vert_{L^p(U)}.
\end{aligned}
\end{equation*}
Moreover, the constant \(C_{U,p,\mathcal{K}}\) is
invariant under dilatations and can
be chosen uniformly for a family of
domains that are Bilipschitz equivalent
with  controlled Lipschitz constants.
\end{theorem}

Using Theorem~\ref{thm:GRbyFJM} and
arguing as in  \cite[Theorem~6]{FJM06},
the following result follows.

\begin{proposition}\label{prop:thm6FJM}
Let \(\omega\subset\RR^2\) be a bounded
Lipschitz domain,  \(I\subset \RR\)
an interval, and \(p\in(1,\infty)\).
Assume that either  
 \(\mathcal{K}=SO(3)\) or  \(\mathcal{K}=SO(3)A_1
\cup SO(3)A_2\), where  \(A_1, \, A_2 \in \RR^{3\times 3}\)
are   \textit{strongly incompatible}.
 Then, 
for all \(\psi \in W^{1,p}(\omega\times I;\RR^3)\), there exist constant matrices,
\(M^a\in \mathcal{K}\) and \(M^b\in \mathcal{K}\),
such that \vspace{-.5mm}
\begin{align*}
& \Vert (r_\epsi^{-1}\grad_\alpha \psi
|\grad_3 \psi) - M^a\Vert^p_{L^p(\omega\times
I;\RR^{3\times 3})} \leq \frac{C}{r_\epsi^p}
\Vert \dist ( (r_\epsi^{-1}\grad_\alpha \psi
|\grad_3 \psi) , \mathcal{K})\Vert^p_{L^p(\omega\times
I)},
\\
& \Vert (\grad_\alpha \psi
|h_\epsi^{-1}\grad_3 \psi) - M^b\Vert^p_{L^p(\omega\times
I;\RR^{3\times 3})} \leq \frac{C}{h_\epsi^p}
\Vert \dist ( (\grad_\alpha \psi
|h_\epsi^{-1}\grad_3 \psi) , \mathcal{K})\Vert^p_{L^p(\omega\times
I)},
\end{align*}
where \(C=C(\omega\times I,p, \mathcal{K})\) is a
 positive constant only  depending  on
 \(\omega\times I\),  \(p\), and \(\mathcal{K}\).

\end{proposition}

\section{Case $\ell\in\RR^+$}\label{Sect:l+}

In this section, we treat the $\ell\in\RR^+$ case. We start by establishing
 some
auxiliary results concerning this case in 
Section~\ref{Subs:l+ar}.
Then, in Section~\ref{Subs:l+proof},
we prove Theorem~\ref{thm:ellr+}. Finally,
in Section~\ref{Subs:rod}, we recover
the nonlinear string model with bending
moments and generalized boundary conditions
 in \cite{RiMSc}.

\subsection{Auxiliary results}\label{Subs:l+ar}

As in \cite{BFM03}, to individualize
the new variables
\(\bar\bcal^a\) and \(\bar\bcal^b\)
in the elastic
part of the total energy,  we introduce,
for \(0<\epsi\leq 1\), the functional
\(F_\epsi:
\big(
 L^p((0,L);\R^{3 \times 2}) \times W^{1,p}(\Omega^a;\R^3)\big)\times
\big( W^{1,p}(\Omega^b;\R^3)
\times L^p(\omega^b;\R^{3}) \big) \to
(-\infty,\infty]
\) defined by \vspace{-.5mm}
\begin{equation}
\label{Fepsi}
\begin{aligned}
F_\epsi((\bar\bcal^a,\psi^a), (\psi^b,
\bar\bcal^b)):=
\begin{cases}
\displaystyle
F_\epsi^a(\psi^a) +
\tfrac{h_\varepsilon}{r_\varepsilon^2}F_\epsi^b(\psi^b) & \hbox{if }  ((\bar\bcal^a,\psi^a),(
\psi^b, \bar\bcal^b)
)\in \mathcal{A}_\epsi\\
\infty & \hbox{otherwise,}
\end{cases}
\end{aligned}
\end{equation}
where \(F^a_\epsi\) and \(F^b_\epsi\)
are given by \eqref{Fabepsi}
and \(\A_\epsi\)  by \eqref{Aepsi}.

The next proposition is proved  in \cite[Proposition 3.1]{GGLM02} and provides a dense subspace of the space
\(\A^p_{\ell_+}\)
 introduced in \eqref{Aell+}. This density
result  will be  useful in
Theorem~\ref{Thm:FepsitoF} below, where
we prove an auxiliary
\(\Gamma\)-convergence result concerning
the sequence of functionals \((F_\epsi)_{\epsi>0}\). 
\begin{proposition}\label{prop:density}
Let  ${\mathcal A}_\ellp^p$ be the space
 defined in \eqref{Aell+}
and let \(V\) be the space defined by
\vspace{-.5mm}
\begin{equation}
\label{V}
\begin{aligned}
V = \big\{(\psi^a, \psi^b) \in W^{1,\infty}(\Omega^a;\mathbb
R^3)\times W^{1,\infty}&(\Omega^b;\mathbb
R^3)\!:\, \psi^a \hbox{ is independent
of } x_\alpha,
\\
  &\psi^b
\hbox{ is independent of }x_3,\, \psi^a(0_3)
= \psi^b(0_\alpha)\big\}.
\end{aligned}
\end{equation}
Then, $V$ is dense in ${\mathcal A}^p_{\ell_+}$
with
respect to the $W^{1,p}\times W^{1,p}$-norm.
\end{proposition} 

Next, we prove a relaxation result that
will be useful in  Theorem~\ref{Thm:FepsitoF} to establish an integral representation of the \(\Gamma\)-limit
of
the sequence  \((F_\epsi)_{\epsi>0}\).

\begin{lemma}\label{lem:relaxationresult}
 Let \(W:\R^{3\times 3} \to \R\) be
a Borel function
satisfying
\eqref{pgrowth},  let \(\A^p_\ellp\)
be given by \eqref{Aell+}, and let \(\bar a,\,
\ell\in\RR^+\). Then, for all $(\psi^a,
\psi^b)\in
{\A}^p_{\ell+}$ and $(\bar\bcal^a, \bar\bcal^b)\in
L^p((0,L);\mathbb
R^{3\times 2})\times L^p(\omega^b;\mathbb
R^3)$, the
functional \vspace{-1.5mm}
\begin{equation*}
\begin{aligned}
 J^p_{\ell_+}((\bar\bcal^a, \psi^a),(\psi^b,\bar\bcal
^b)):=\inf\bigg\{&\liminf_{j\to \infty}
\bigg( \bar{a}\int_0^L
W({\bar a}^{-1}\bar\bcal^a_j| \nabla_3
\psi_j^a)\,\dx_3
+\ell\int_{\omega^b}W(\nabla_\alpha
\psi_j^b| \bar\bcal^b_j)\,\dx_\alpha\bigg)\!:\\
&\enspace(\psi^a_j,
\psi^b_j)\in {\A}^p_{\ell_+},\, \psi_j^a \rightharpoonup
\psi^a \hbox{ in } W^{1,p}((0,L);\mathbb
R^3),\, \bar\bcal^a_j \rightharpoonup
\bar\bcal^a \hbox{
in }L^p((0,L);\mathbb R^{3\times 2}),\\
&\enspace  \psi_j^b \rightharpoonup
\psi^b \hbox{ in }
W^{1,p}(\omega^b;\mathbb R^3),\, \bar\bcal^b_j
\rightharpoonup
\bar\bcal^b \hbox{ in }L^p(\omega^b;\mathbb
R^{3})
 \bigg\}
\end{aligned}
\end{equation*}
coincides with \vspace{-1mm}
\begin{equation}\label{eq:relaxedee}
\begin{aligned}
\bar{a}\int_0^L\ce(\bar{a}^{-1}\bar\bcal^a|\nabla_3
\psi^a)\,\dx_3+\ell\int_{\omega^b}\qce(\nabla_\alpha
\psi^b|\bar\bcal^b)\,\dx_\alpha.
\end{aligned}
\end{equation}
\end{lemma}

\begin{proof}[Proof]
Fix \((\psi^a, \psi^b)\in
{\A}^p_{\ell+}$ and $(\bar\bcal^a, \bar\bcal^b)\in
L^p((0,L);\mathbb
R^{3\times 2})\times L^p(\omega^b;\mathbb
R^3)\). Because \(\ce\leq W\),  \(\qce\leq
W\), and the functional in \eqref{eq:relaxedee}
is sequentially weakly lower semincontinuous
in \(\big( L^p((0,L);\mathbb
R^{3\times 2})
\times W^{1,p}((0,L);\allowbreak\mathbb
R^3) \big) \times \big(W^{1,p}(\omega^b;\mathbb R^3) \times
L^p(\omega^b;\mathbb
R^{3}) \big)\) (see Section~\ref{Sect:prelim}),  we have \vspace{-1mm}
\begin{equation*}
\begin{aligned}
J^p_{\ell_+}((\bar\bcal^a, \psi^a),(\psi^b,\bar\bcal
^b)) \geq \bar{a}\int_0^L\ce(\bar{a}^{-1}\bar\bcal^a|\nabla_3
\psi^a)\,\dx_3+\ell\int_{\omega^b}\qce(\nabla_\alpha
\psi^b|\bar\bcal^b)\,\dx_\alpha.
\end{aligned}
\end{equation*}
To prove the converse inequality, we
start by observing
that in view of \cite[Theorem~4.17]{LDR00}
(see also 
\cite[Theorem~5.4]{FKP94}) and Remark~\ref{ccxqcx}, we
have \vspace{-.5mm}
\begin{align}
&\inf\bigg\{\liminf_{k\to \infty}\bar{a}\int_0^L
W({\bar
a}^{-1}\bar\bcal^a_k| \nabla_3 \psi_k^a)\,\dx_3\!:\,
\psi_k^a \rightharpoonup \psi^a \hbox{
in } W^{1,p}((0,L);\mathbb
R^3), \,
\bar\bcal^a_k \rightharpoonup
\bar\bcal^a \hbox{
in }L^p((0,L);\mathbb R^{3\times 2})\bigg\}
\nonumber
\\ &\quad ={\bar a}\int_0^L \ce({\bar
a}^{-1}\bar\bcal^a|
\nabla_3 \psi^a)\,\dx_3, \label{eq:one}\\
&\inf\bigg\{\liminf_{k\to \infty}\int_{\omega^b}W(\nabla_\alpha
\psi_k^b| \bar\bcal^b_k)\,\dx_\alpha\!:
\psi_k^b \rightharpoonup
\psi^b \hbox{ in } W^{1,p}(\omega^b;\mathbb
R^3),\, \bar\bcal^b_k
\rightharpoonup \bar\bcal^b \hbox{ in
}L^p(\omega^b;\mathbb
R^{3})\bigg\} \nonumber
 \\
 &\quad=\int_{\omega^b}\qce(\nabla_\alpha
\psi^b| \bar\bcal^b)\,\dx_\alpha. \label{eq:two}
\end{align} 
This completes the proof
in the case  in which $p\in(1, 2]$.

Assume now that $p>2$. By the  relaxation
results in
\eqref{eq:one} and \eqref{eq:two},
by  the  decomposition lemmas \cite[Lemma~2.1]{FMP98}
and \cite[Lemma~8.13]{FL07}, and by
\eqref{pgrowth},
we can find sequences 
\((\bar\bcal^a_k,\psi^a_k
)_{k\in\Nb}\subset  L^p((0,L);\mathbb
R^{3\times 2})
\times W^{1,p}((0,L);\mathbb
R^3)   \) and  
\((
\psi^b_k, \bar\bcal^b_k)_{k\in\Nb}\subset
W^{1,p}(\omega^b;\mathbb R^3) \times
L^p(\omega^b;\mathbb
R^{3})   \) such that
\(\psi^a_k \weakly \psi^a\) weakly in
 \(W^{1,p}((0,L);\R^3)\), \( 
 \bar\bcal^a_k \weakly \bar\bcal^a\)
weakly in \(L^{p}((0,L);\R^{3\times
 2})\), \(\psi^b_k \weakly \psi^b\)
weakly in
 \(W^{1,p}(\omega^b;\R^3)\),  \( 
 \bar\bcal^b_k \weakly \bar\bcal^b\)
weakly in \(L^{p}(\omega^b;\R^3)\),
  \((|\grad_3 \psi^a_k|^p+|\bar\bcal^a_k|^p)_{k\in\Nb}
 \subset L^1((0,L))\) and \((|\grad_\alpha\psi^b_k|^p+|\bar\bcal^b_k|^p)_{k\in\Nb}
 \subset L^1(\omega^b)\) are equi-integrable,
and \vspace{-1.5mm}
\begin{equation*}
\begin{aligned}
& \lim_{k\to \infty} \bar{a}\int_0^L
W({\bar
a}^{-1}\bar\bcal^a_k| \nabla_3 \psi_k^a)\,\dx_3
=\bar{a}\int_0^L\ce(\bar{a}^{-1}\bar\bcal^a|\nabla_3
\psi^a)\,\dx_3,\\
&\lim_{k\to \infty}\int_{\omega^b}W(\nabla_\alpha
\psi_k^b| \bar\bcal^b_k)\,\dx_\alpha
= \int_{\omega^b}\qce(\nabla_\alpha
\psi^b| \bar\bcal^b)\,\dx_\alpha.
\end{aligned}
\end{equation*}
To conclude, we are left to find sequences
that do not
increase the above integral limits and,
simultaneously,
satisfy the junction condition  $\psi^a_k(0_3)=\psi^b_k(0_\alpha)$.
This can be achieved by a similar slicing
argument to that used in Step~2 of the
proof of Lemma~\ref{lem:WbyQW}.
The main difference here is that instead
of just considering
one sequence of smooth cut-off function,
we consider
two; precisely, for a well-chosen sequence
\((\delta_j)_{j\in\NN}\)
of positive numbers convergent to zero,
we take  \((\phi_j^a)_{j\in\NN} \subset
 C^\infty(\R;[0,1])\) and \((\phi_j^b)_{j\in\NN}
\subset
 C^\infty(\R^2;[0,1])\)
 such that \(\phi_j^a(x_3)
= 0\)
for \(x_3\in (-\delta_j,\delta_j)\),
 \(\phi_j^a(x_3) =
1\) for \(|x_3|\geq 2\delta_j\),  \(\phi_j^b(x_\alpha)
= 0\)
for \(|x_\alpha|<\delta_j\),  \(\phi_j^b(x_\alpha)
=
1\) for \(|x_\alpha|\geq 2\delta_j\),
and
 \((\Vert\grad_3\phi_j^a\Vert_\infty
 + \Vert\grad_\alpha\phi_j^b\Vert_\infty)
\leq \tfrac{c}{\delta_j}\), where \(c>0\)
is independent
of \(j\in\NN\).
Then, defining \vspace{-1.mm}
\begin{equation*}
\begin{aligned}
&\psi_{k,j}^a(x_3):= \phi_j^a(x_3)\psi^a_k(x_3)
+
(1-\phi^a_j(x_3))\psi^a(x_3), \quad x_3\in
(0,L),\\
& \psi_{k,j}^b(x_\alpha):= \phi_j^b(x_\alpha)\psi^b_k(x)
+
(1-\phi_j^b(x_\alpha))\psi^b(x_\alpha),
\quad x_\alpha\in \omega^b,
\end{aligned}
\end{equation*}
and arguing as in Step~2 of the proof
of Lemma~\ref{lem:WbyQW},
we can find a subsequence \(k_j \prec k\) and \((\tilde \psi^a_j, \tilde
\psi^b_j) \in
\A^p_\ellp\), \(j\in\NN\), such that
\((\tilde \psi^a_j
)_{j\in\Nb} \), \((\bar\bcal^a_{k_j})_{j\in\NN}\),
    \((\tilde \psi^b_j)_{j\in\NN}\),
and \( (\bar\bcal^b_{k_j})_{j\in\Nb}\) 
are admissible sequences for \(J^p_{\ell_+}((\bar\bcal^a,
\psi^a),(\psi^b,\bar\bcal
^b))\) and \vspace{-1.5mm}
\begin{equation*}
\begin{aligned}
&\liminf_{j\to\infty} \bigg( \bar{a}\int_0^L
W({\bar a}^{-1}\bar\bcal^a_{k_j}| \nabla_3
\tilde \psi_j^a)\,\dx_3
+\ell\int_{\omega^b}W(\nabla_\alpha
\tilde \psi_j^b| \bar\bcal^b_{k_j})\,\dx_\alpha\bigg)\\
&\quad \leq\bar{a}\int_0^L\ce(\bar{a}^{-1}\bar\bcal^a|\nabla_3
\psi^a)\,\dx_3+\ell\int_{\omega^b}\qce(\nabla_\alpha
\psi^b|\bar\bcal^b)\,\dx_\alpha.
\end{aligned}
\end{equation*}
This completes the proof of Lemma~\ref{lem:relaxationresult}.
\end{proof}

Finally, we prove  a \(\Gamma\)-convergence
result
for the
sequence \((F_\epsi)_{\epsi>0}\) of
functionals defined by \eqref{Fepsi}.

\begin{theorem}\label{Thm:FepsitoF}
Let \(W:\R^{3\times 3} \to \R\) be a
Borel function satisfying
\eqref{pgrowth},  assume that \(\ell\in\RR^+\),
and let \(F_\epsi\) be given by \eqref{Fepsi}.
Then, \((F_\epsi)_{\epsi>0}\)
\(\Gamma\)-converges,
with respect to the weak topology in
\(
\big(L^p((0,L);\R^{3 \times 2}) \times
W^{1,p}(\Omega^a;\R^3)\big)\times
\big( W^{1,p}(\Omega^b;\R^3)
\times L^p(\omega^b;\R^{3}) \big) \),
to the functional
\(F:
\big(L^p((0,L);\R^{3 \times 2}) \times
W^{1,p}(\Omega^a;\R^3)\big)\times
\big( W^{1,p}(\Omega^b;\R^3)
\times L^p(\omega^b;\R^{3}) \big) \to
(-\infty,\infty]
\) defined by
\begin{equation*}
\begin{aligned}
F((\bar\bcal^a,\psi^a), (\psi^b, \bar\bcal^b)):=
\begin{cases}
\displaystyle
F^a(\bar\bcal^a,\psi^a) +\ell 
F^b(\psi^b,\bar\bcal^b) & \hbox{if }
 (\psi^a,
\psi^b)
\in \mathcal{A}^p_{l^+}\\
\infty & \hbox{otherwise,}
\end{cases}
\end{aligned}
\end{equation*}
where \(\A^p_{l^+}\) is given by \eqref{Aell+}
and, for \(\bar a=|\omega^a|\), \vspace{-1mm}
\begin{equation}
\label{Fab}
\begin{aligned}
F^a(\bar\bcal^a,\psi^a):= \bar a\int_0^L
\ce({\bar a^{-1}}\bar\bcal^a|\nabla_3\psi^a)\,\dx_3,
\quad %
F^b(\psi^b,\bar\bcal^b):= \int_{\omega^b}
\qce(\nabla_\alpha \psi^b|\bar\bcal^b)\,\dx_\alpha.
\end{aligned}
\end{equation}
\end{theorem}

\begin{proof}
To prove the claim,
it suffices to show
that  for any subsequence \(\epsi_n\prec\epsi\), the \(\Gamma\)-limit inferior,
\(F^-\), of \((F_{\epsi_n})_{n\in\Nb}\),
given by
\eqref{F-}
with \(\ell^a_n=1\) and \(\ell^b_n=
{h_{\epsi_n}}/{r_{\epsi_n}^2}\), coincides with \(F\)  for all \(\big((\bar\bcal^a,\psi^a),
(\psi^b, \bar\bcal^b)\big) \in
\big(L^p((0,L);\R^{3 \times 2}) \times
W^{1,p}(\Omega^a;\R^3)\big)\times
\big( W^{1,p}(\Omega^b;\R^3)
\times L^p(\omega^b;\R^{3}) \big)  \).
To show this identity, we will proceed
in several steps.

To simplify the notation, we
set 
\(\mu_n:={r_{\varepsilon_n}^{-1}}\),
\(\lambda_n:={h_{\varepsilon_n}^{-1}}\),
and 
\(\ell_n:={h_{\varepsilon_n}}/{r_{\varepsilon_n}^2}\).

\textit{Step 1.} In this step, we prove
that we may assume
without loss of generality that \(W
\) is  quasiconvex.

In fact, in view of \eqref{eq:cxtyineq},
we have  \(\mathcal{C}(\qe) = \ce\)
and \(\mathcal{QC}(\qe) = \qce\).
On the other hand, by Lemma~\ref{lem:WbyQW},
the \(\Gamma\)-limit inferior in
\eqref{F-} with \(\ell^a_n=1\) and \(\ell^b_n=
{h_{\epsi_n}}/{r_{\epsi_n}^2}\) remains
unchanged if we replace \(W\) by its
quasiconvex envelope,
\(\qe\). Thus, we may assume
 that \(W
\) is quasiconvex
function. In this case, by \eqref{pgrowth},
 we also have that \(W\) is \(p\)-Lipschitz continuous;
i.e., there
exists a positive constant, \(C\), such
that for all
\(\xi,\,\xi' \in \R^{3\times 3}\), we
have \vspace{-.5mm}
\begin{equation}
\label{eq:WpLip}
\begin{aligned}
|W(\xi) - W(\xi')| \leq C(1+ |\xi|^{p-1}
+\ |\xi'|^{p-1})|\xi
- \xi'|.
\end{aligned}
\end{equation}
%


\textit{Step 2.} In this step, we prove
that if \(F^-((\bar\bcal^a,\psi^a),
(\psi^b, \bar\bcal^b)) < \infty\), then
\((\psi^a,\psi^b ) \in \mathcal{A}^p_{l^+}\).

By definition of the \(\Gamma\)-limit
inferior, for all
\(n\in\Nb\), there exists  \(((\bar\bcal^a_n,\psi^a_n),(
\psi^b_n, \bar\bcal^b_n)
)\in \mathcal{A}_{\epsi_n} \) such that
 \(((\bar\bcal^a_n,\psi^a_n),(
\psi^b_n, \bar\bcal^b_n)
)_{n\in\Nb}\) weakly converges to \(((\bar\bcal^a,\psi^a),
(\psi^b, \bar\bcal^b) )\) in  \(
\big(L^p((0,L);\R^{3 \times 2}) \times
W^{1,p}(\Omega^a;\R^3)\big)\times
\big( W^{1,p}(\Omega^b;\R^3)
\times L^p(\omega^b;\R^{3}) \big) \)
and \(F^-((\bar\bcal^a,\psi^a),
(\psi^b, \bar\bcal^b)) = \liminf_{n\to\infty}\big(F_{\epsi_n}^a
(\psi^a_n) + \ell_nF_{\epsi_n}^b
(\psi^b_n) \big) \).

Because \(F^-((\bar\bcal^a,\psi^a),
(\psi^b, \bar\bcal^b)) < \infty\), extracting
a subsequence
if needed, we may assume that there
exists a positive
constant, \(\bar C\), such that for
all \(n\in\Nb\),
we have \(F_{\epsi_n}^a
(\psi^a_n) + \ell_nF_{\epsi_n}^b
(\psi^b_n) \leq \bar C\). Then, \eqref{pgrowth}
yields \vspace{-1.5mm}
\begin{equation*}
\begin{aligned}
\bar C \geq &\frac{1}{2C}\Big( \Vert
\mu_n \grad_\alpha\psi^a_n
\Vert^p_{L^p(\Omega^a;\R^{3\times 2})}
+ \Vert \grad_3\psi^a_n
\Vert^p_{L^p(\Omega^a;\R^{3})} +
\ell_n \Vert \grad_\alpha\psi^b_n
\Vert^p_{L^p(\Omega^a;\R^{3\times 2})}
+ \ell_n\Vert\lambda_n
\grad_3\psi^b_n
\Vert^p_{L^p(\Omega^a;\R^{3})} \Big)\\
&\quad- C(|\Omega^a| +\ell_n|\Omega^b|)
\end{aligned}
\end{equation*}
for all \(n\in\Nb\); this estimate, \eqref{ell},
and Lemma~\ref{lem:l21GGLM}
allow us to conclude that \((\psi^a,\psi^b
) \in \mathcal{A}^p_{l^+}\).

\textit{Step 3 (lower bound).} In this
step, we prove
that for all  \((\psi^a,\psi^b ) \in
\mathcal{A}^p_{l^+}\)
and $(\bar\bcal^a, \bar\bcal^b)\in L^p((0,L);\mathbb
R^{3\times 2})\times L^p(\omega^b;\mathbb
R^3)$, we have 
%
\(F^-((\bar\bcal^a,\psi^a),
(\psi^b, \bar\bcal^b)) \geq F^a(\bar\bcal^a,\psi^a)
+
\ell F^b(\psi^b,\bar\bcal^b) .\)
%
%

For each \(n\in\NN\), let \(((\bar\bcal^a_n,\psi^a_n),
\allowbreak(
\psi^b_n, \bar\bcal^b_n)
) \in \mathcal{A}_{\epsi_n}\)
 be such that  \vspace{-1.5mm}
\begin{equation*}
\begin{aligned}
&\psi^a_n \weakly \psi^a \hbox{ weakly
in
} W^{1,p}(\Omega^a;\R^3),
\, 
 \bar\bcal^a_n=\mu_n \int_{\omega^a}\nabla_\alpha
\psi^a_n\,\dx_\alpha \weakly \bar\bcal^a
\hbox{ weakly
in } L^{p}((0,L);\R^{3\times
 2}),\\
&  \psi^b_n \weakly \psi^b \hbox{ weakly
in
} W^{1,p}(\Omega^b;\R^3),
\, 
 \bar\bcal^b_n=\lambda_n \int_{-1}^0\nabla_3\psi^b_n\,
\dx_\alpha \weakly \bar\bcal^b \hbox{
weakly in } L^{p}(\omega^b;\R^3).
\end{aligned}
\end{equation*}

Using the inequality \(W\geq \ce\),
Fubini's lemma, and
Jensen's inequality, we obtain
\vspace{-1mm}%
\begin{equation}
\label{eq:lba1}
\begin{aligned}
  \int_{\Omega^a} W(\mu_n\nabla_\alpha
\psi^a_n|\nabla_3\psi^a_n)\,\dx &\geq
\int_0^L\int_{\omega^a}
\ce(\mu_n\nabla_\alpha
\psi^a_n|\nabla_3\psi^a_n)\,\dx_\alpha
\dx_3 \\
&\geq \bar a \int_0^L \ce\Big(\frac{\mu_n}{\bar
a}\int_{\omega^a}
\nabla_\alpha
\psi^a_n\,\dx_\alpha\Big|\frac{1}{\bar
a}\int_{\omega^a}\nabla_3\psi^a_n\,
\dx_\alpha \Big )\,\dx_3.
\end{aligned}
\end{equation}
 Next, we observe that the functional
\(\calG:L^p((0,L);\R^{3
\times 3}) \to \R\) defined by
%
\(\calG(v):=\int_0^L \ce(v(t))\,\dt\)
%
is sequentially lower semicontinuous
with respect to
the weak topology in \(L^p((0,L);\R^{3
\times 3})\) because \(\ce\) is a convex
function satisfying
the bounds in \eqref{pgrowth}. Consequently,
since
\begin{equation*}
\begin{aligned} 
\Big(\frac{\mu_n}{\bar a}\int_{\omega^a}
\nabla_\alpha
\psi^a_n\,\dx_\alpha\Big|\frac{1}{\bar
a}\int_{\omega^a}\nabla_3\psi^a_n\,
\dx_\alpha \Big ) \weakly \Big(\frac{1}{\bar
a}\bar\bcal^a
\Big| \grad_3 \psi^a\Big)
\end{aligned}
\end{equation*}
weakly in \(L^p((0,L);\R^{3
\times 3})\), we have
\begin{equation*}
\begin{aligned}
\liminf_{n\to\infty} \calG \Big(\frac{\mu_n}{\bar
a}\int_{\omega^a}
\nabla_\alpha
\psi^a_n\,\dx_\alpha\Big|\frac{1}{\bar
a}\int_{\omega^a}\nabla_3\psi^a_n\,
\dx_\alpha \Big ) \geq \calG \Big(\frac{1}{\bar
a}\bar\bcal^a
\Big| \grad_3 \psi^a\Big).
\end{aligned}
\end{equation*}
This estimate and \eqref{eq:lba1} entail
\vspace{-2mm}%
\begin{equation}
\label{eq:lba}
\begin{aligned}
\liminf_{n\to\infty} F^a_{\epsi_n} (\psi^a_n)\geq
\bar
a\int_0^L \ce({\bar a^{-1}}\bar\bcal^a|\nabla_3\psi^a)\,\dx_3
=  F^a(\bar\bcal^a,\psi^a). 
\end{aligned}
\end{equation}On the other hand, by \eqref{ell} and by   \cite[Theorem~1.2~(i)]{BFM03}
together with \eqref{eq:Qast=Qhat},
we have \vspace{-.5mm} 
\begin{equation}
\label{eq:lbb}
\begin{aligned}
\liminf_{n\to\infty} \ell_nF^b_{\epsi_n}
(\psi^b_n)&=\liminf_{n\to\infty}
\ell_n\int_{\Omega^b} W(\nabla_\alpha
\psi^b_n|\lambda_n\nabla_3\psi^b_n)\,\dx
\geq  \ell\int_{\omega^b}
\qce(\nabla_\alpha \psi^b|\bar\bcal^b)\,\dx_\alpha
=\ell F^b(\psi^b,\bar\bcal^b).
\end{aligned}
\end{equation}
From \eqref{eq:lba} and \eqref{eq:lbb},
we obtain \vspace{-1mm}
\begin{equation*}
\begin{aligned}
\liminf_{n\to\infty}  \big(F_{\epsi_n}^a
(\psi^a_n) +
\ell_nF_{\epsi_n}^b (\psi^b_n)
\big) &\geq \liminf_{n\to\infty} F^a_{\epsi_n}
(\psi^a_n)
+\liminf_{n\to\infty} \ell_nF^b_{\epsi_n}
(\psi^b_n) \geq
F^a(\bar\bcal^a,\psi^a)
+\ell F^b(\psi^b,\bar\bcal^b),
\end{aligned}
\end{equation*}
from which the conclusion follows by
taking the infimum
over all admissible sequences \(((\bar\bcal^a_n,\psi^a_n),
\allowbreak(
\psi^b_n, \bar\bcal^b_n)
)_{n\in\Nb}\) in the definition of \(F^-((\bar\bcal^a,\psi^a),
(\psi^b, \bar\bcal^b))\).
 
\smallskip
\textit{Step 4 (upper bound in terms
of the original
density \(W\) and for regular target
functions).} In this step,
we prove
that for all  $(\psi^a, \psi^b) \in
V$ (see \eqref{V})
and $(\bar\bcal^a,
\bar\bcal^b) \in C^1([0,L];\mathbb R^{3\times
2}) \times C^1(\overline\omega^b;\mathbb
R^3)$, we have \vspace{-1mm}
\begin{equation}
\label{eq:ubwrtW}
\begin{aligned}
F^-((\bar\bcal^a,\psi^a),
(\psi^b, \bar\bcal^b)) \leq \bar{a}\int_0^L
W({\bar
a}^{-1}\bar\bcal^a|\nabla_3\psi^a)\,\dx_3
+ \ell\int_{\omega^b}W(\nabla_\alpha
\psi^b| \bar\bcal^b)\,\dx_\alpha .
\end{aligned}
\end{equation}

The proof of \eqref{eq:ubwrtW}
follows closely
that of \cite[Proposition~4.1]{GGLM02}.
For \(n\in\NN\), define 
\begin{equation*}
\psi^a_n(x) :=
\begin{cases}
\big(\frac{r_{\varepsilon_n}}{\bar
a} \bar\bcal^a(r_{\varepsilon_n})x_\alpha^T
+ \psi^a(r_{\varepsilon_n})\big)\frac{x_3}{r_{\varepsilon_n}}
+ \psi^b(r_{\varepsilon_n} x_\alpha)\frac{r_{\varepsilon_n}-x_3}{r_{\varepsilon_n}}
 &\hbox{if } x = (x_\alpha, x_3) \in \omega^a\times
(0, r_{\varepsilon_n})\\
\frac{r_{\varepsilon_n}}{\bar a}\bar\bcal^a(x_3)x_\alpha^T
+ \psi^a(x_3)
&\hbox{if } x= (x_\alpha, x_3) \in \omega^a \times
(r_{\varepsilon_n},
L)
\end{cases}
\end{equation*}
and \vspace{-1.5mm}
\begin{equation*}
\psi^b_n(x):=  h_{\varepsilon_n} x_3\bar\bcal^b(x_\alpha)
+ \psi^b(x_\alpha)  \hbox{ if } x \in
\Omega^b.
\end{equation*}
Note that \({h_{\varepsilon_n}^{-1}}\int_{-1}^{0}\nabla_3
\psi^b_n\,\dx_3=
\bar\bcal^b\) 
and, defining \(\bar\bcal^a_n:= {r_{\epsi_n}^{-1}}
\int_{\omega^a}\nabla_\alpha
\psi^a_n\,\dx_\alpha\), it can be easily
checked that
\(((\bar\bcal^a_n,\psi^a_n),(\psi_n^b,\bar\bcal^b))\in
\A_{\epsi_n}\) for all \(n\in\NN\);
moreover,
%
\(\psi^b_n \to \psi^b \hbox{ in } W^{1,p}(\Omega^b;\mathbb
R^3) \)
%
and \(\grad_\alpha\psi^b_n \to \grad_\alpha\psi^b\)
pointwise in \(\Omega^b\).
Thus, the continuity of \(W\) (see \eqref{eq:WpLip}
in Step~1), Lebesgue's
dominated convergence
theorem, and \eqref{pgrowth} yield \vspace{-.5mm}
\begin{equation}
\label{eq:oneb}
\lim_{n\to\infty}\int_{\Omega^b}W\left(\nabla_\alpha
\psi^{b}_n| h_{\varepsilon_n}^{-1}\nabla_3
\psi^b_n\right)\dx = \lim_{n\to\infty}\int_{\Omega^b}W(\nabla_\alpha
\psi^{b}_n| \bar\bcal^b)\dx= \int_{\omega^b}W(\nabla_\alpha
\psi^b| \bar\bcal^b)\,\dx_\alpha.
\end{equation}

On the other hand, for all \(x = (x_\alpha,
x_3) \in \omega^a\times (0, r_{\varepsilon_n})\),
we have the
following pointwise estimates:
\begin{equation*}
\begin{aligned}
&|\psi^a_n(x)| \leq r_{\varepsilon_n}
{\bar
a}^{-1}|x_\alpha| 
|\bar\bcal^a(r_{\varepsilon_n})| + |
\psi^a(r_{\varepsilon_n})|
+2|\psi^b(r_{\varepsilon_n} x_\alpha)|
\leq C(\Vert \bar\bcal^a\Vert_\infty
+ \Vert \psi^a\Vert_\infty +  \Vert
\psi^b\Vert_\infty),\\
& |\grad_\alpha \psi^a_n(x)| \leq  r_{\varepsilon_n}
{\bar
a}^{-1}|\bar\bcal^a(r_{\varepsilon_n})|+2
r_{\epsi_n}|\grad_\alpha
\psi^b(r_{\varepsilon_n} x_\alpha)|\leq
r_{\varepsilon_n}C(\Vert \bar\bcal^a\Vert_\infty
 +  \Vert\grad_\alpha  \psi^b\Vert_\infty),
\\
& |\grad_3 \psi^a_n(x)| = \big|{\bar
a}^{-1} \bar\bcal^a(r_{\varepsilon_n})x_\alpha^T
+ {r_{\epsi_n}^{-1}}  \psi^a(r_{\varepsilon_n})
-
{r_{\epsi_n}^{-1}}  \psi^b(r_{\varepsilon_n}
x_\alpha) \big| \leq C(\Vert \bar\bcal^a\Vert_\infty
 + {\rm Lip}(\psi^a) + {\rm Lip}(\psi^b)),
\end{aligned}
\end{equation*}
where in the last estimate we used the
identity \(\, - {r_{\epsi_n}^{-1}}\psi^a(0_3)
+  {r_{\epsi_n}^{-1}} \psi^b(0_\alpha)
=0\) (see \eqref{V}), and where \(C\)
is a
positive constant independent of \(n\).
These estimates, the definition
of \(\psi^a_n\), and \eqref{pgrowth}
entail \vspace{-1.5mm}%
\begin{equation*}
\begin{aligned}
&\psi^a_n \to \psi^a \hbox{ in } W^{1,p}(\Omega^a;\mathbb
R^3), \quad \bar\bcal^a_n= \frac{1}{r_{\epsi_n}}
\int_{\omega^a}\nabla_\alpha
\psi^a_n\,\dx_\alpha \to \bar\bcal^a
\hbox{ in } L^p((0,L);\mathbb R^3),\\
& \lim_{n\to\infty}\int_0^{r_{\epsi_n}}\!\!\!
\int_{\omega^a} W \left(r_{\varepsilon_n}^{-1}\nabla_\alpha
\psi^a_n| \nabla_3 \psi^a_n\right)\dx_\alpha\dx_3
=0.
\end{aligned}
\end{equation*}
From this last limit and arguing as in \eqref{eq:oneb}, we obtain \vspace{-1mm}
\begin{equation}
\begin{aligned}
\label{eq:onea}
&\lim_{n\to\infty}\int_{\Omega^a}W \left(r_{\varepsilon_n}^{-1}\nabla_\alpha
\psi^a_n| \nabla_3 \psi^a_n\right)\dx
\\
&\quad = \lim_{n\to\infty}\int_{\Omega^a}
W \left(\bar a^{-1}\bar\bcal^a(x_3)|r_{\varepsilon_n}
\bar a^{-1}\grad_3
\bar \bcal^a(x_3)x_\alpha^T
+\nabla_3 \psi^a(x_3)\right)\chi_{(r_{\varepsilon_n},L)}(x_3)\,\dx
\\
 &\quad= \bar{a}\int_0^L W({\bar
a}^{-1}\bar\bcal^a|\nabla_3\psi^a)\,\dx_3.
\end{aligned}
\end{equation}

Using the definition of \(F^-((\bar\bcal^a,\psi^a),
(\psi^b, \bar\bcal^b))\), \eqref{eq:onea},
\eqref{ell},
and \eqref{eq:oneb}, we conclude Step~4.

%

%
\textit{Step 5 (Upper bound in terms
of the original
density \(W\)).} In this step, we prove
that \eqref{eq:ubwrtW}
holds for all   $(\psi^a, \psi^b) \in
\A_\ellp$
and $(\bar\bcal^a,
\bar\bcal^b) \in L^p((0,L);\mathbb R^{3\times
2}) \times
L^p(\omega^b;\mathbb
R^3)$.

The claim in this step follows from Steps~1 and 4, Proposition~\ref{prop:density},
the density  of $C^1([0,L];\mathbb
R^3)$ and $C^1(\overline\omega^b,\mathbb
R^3)$ in $L^p((0,L);\mathbb
R^3)$ and $L^p(\omega^b;\mathbb R^3),$
 respectively,
with
respect to the $L^p$-strong convergence,
the sequential lower semicontinuity
of \(F^-\) with respect to the weak
convergence in  \(
\big(L^p((0,L);\R^{3 \times 2}) \times
W^{1,p}(\Omega^a;\R^3)\big)\times
\big( W^{1,p}(\Omega^b;\R^3)
\times L^p(\omega^b;\R^{3}) \big) \),
 \eqref{pgrowth}, and Lebesgue's
dominated convergence theorem.

\textit{Step 6 (Upper bound).} In this
step, we prove
that for all  \((\psi^a,\psi^b ) \in
\mathcal{A}^p_{l^+}\)
and $(\bar\bcal^a, \bar\bcal^b)\in L^p((0,L);\mathbb
R^{3\times 2})\times L^p(\omega^b;\mathbb
R^3)$, we have
\(F^-((\bar\bcal^a,\psi^a),
(\psi^b, \bar\bcal^b)) \leq F^a(\bar\bcal^a,\psi^a)
+\ell 
F^b(\psi^b,\bar\bcal^b)\).

The claim in this step is an immediate
consequence of Step~5, the sequential
lower semicontinuity of \(F^-\) with
respect to the weak
convergence in  \(
\big(L^p((0,L);\R^{3 \times 2}) \times
W^{1,p}(\Omega^a;\R^3)\big)\times
\big( W^{1,p}(\Omega^b;\R^3)
\times L^p(\omega^b;\R^{3}) \big) \),
and  Lemma~\ref{lem:relaxationresult}.
\end{proof}

We conclude this section by proving
a lemma that allows
us to address the boundary conditions
in the minimization
problem \eqref{Pepsi}. The proof uses
some of the ideas
in \cite[Lemma~2.2]{BFM03} and is based
on a slicing
argument.

\begin{lemma}\label{lem:ontrace}
Let \(W:\R^{3\times 3} \to \R\) be a
Borel function satisfying
\eqref{pgrowth},  let \(\kappa\in\R\),
and assume that \(\ell\in\RR^+\).
Fix \((\psi^a,\psi^b ) \in \mathcal{A}^p_{l^+}\)
and $(\bar\bcal^a, \bar\bcal^b)\in L^p((0,L);\mathbb
R^{3\times 2})\times L^p(\omega^b;\mathbb
R^3)$.  For each
\(n\in\Nb\), let \(((\bar\bcal^a_n,\psi^a_n),(
\psi^b_n, \bar\bcal^b_n)
)\in \mathcal{A}_{\epsi_n}\) be such
that  \( \psi^a_n
\to \psi^a\)   in
\(L^{p}(\Omega^a;\R^3)\), \( \bar\bcal^a_n
\weakly \bar\bcal^a\)
weakly in \(L^{p}((0,L);\R^{3\times
 2})\), \( \psi^b_n \to \psi^b\)  in
\(L^{p}(\Omega^b;\R^3)\),
 \(  \bar\bcal^b_n \weakly \bar\bcal^b\)
weakly in \(L^{p}(\omega^b;\R^3)\),
 and \vspace{-.5mm}
\begin{equation}\label{eq:limL}
\begin{aligned}
\lim_{n\to\infty}
  \Big(F_{\epsi_{n}}^a (\psi^a_{n})
+ \tfrac{h_{\varepsilon_n}}{r_{\varepsilon_n}^2}F_{\epsi_{n}}^b (\psi^b_{n})\Big)
= \kappa.
\end{aligned}
\end{equation}
For each
\(n\in\Nb\), let  \(\ffi^a_n,\, \ffi^a \in W^{1,p}(\Omega^a;\R^3)\) and 
 \(\ffi^b_n, \, \ffi^b \in W^{1,p}(\Omega^b;\R^3)\)
be such that \vspace{-.5mm}
\begin{equation}\label{eq:assonbc}
\begin{aligned}
&\ffi^a_n \weakly \ffi^a \hbox{ weakly
in } W^{1,p}(\Omega^a;\R^3),
\enspace \big(|r_{\epsi_n}^{-1}\grad_\alpha
\ffi^a_n|^p
+ |\grad_3\ffi^a_n|^p
\big)_{n\in\Nb}\subset L^1(\Omega^a)
\hbox{ is equi-integrable},
\\  
&\ffi^b_n \weakly \ffi^b \hbox{ weakly
in } W^{1,p}(\Omega^b;\R^3),\enspace
\big(|\grad_\alpha \ffi^b_n|^p + |h_{\epsi_n}^{-1}\grad_3\ffi^b_n|^p
\big)_{n\in\Nb}\subset L^1(\Omega^b)
\hbox{ is equi-integrable}.
\end{aligned}
\end{equation}
Assume further that \(\psi^a = \ffi^a\)
on \(\Gamma^a\) and \(\psi^b = \ffi^b\)
on \(\Gamma^b
\). Then, there exist subsequences \(r_{\epsi_{n_k}}
\prec r_{\epsi_n}\), \(h_{\epsi_{n_k}}
\prec h_{\epsi_n}\), \(\ffi^a_{{n_k}}
\prec \ffi^a_n\), and \(\ffi^b_{{n_k}}
\prec \ffi^b_n\), and a sequence 
\(((\tilde{\bar\bcal}^a_k,\tilde\psi^a_k),(\tilde
\psi^b_k, \tilde{\bar\bcal}^b_k)
)_{k\in\Nb}\) satisfying
\(((\tilde{\bar\bcal}^a_k,\tilde\psi^a_k),(\tilde
\psi^b_k, \tilde{\bar\bcal}^b_k)
) \in \mathcal{A}_{\epsi_{n_k}}\), \(\tilde\psi^a_k
= \ffi^a_{n_k}\)
on \(\Gamma^a\), and \(\tilde\psi^b_k
= \ffi^b_{n_k}\)
on \(\Gamma^b\)  for all \(k\in\Nb\);
moreover,
\(\tilde\psi^a_k \weakly \psi^a\) weakly
in
 \(W^{1,p}(\Omega^a;\R^3)\), \( 
\tilde{\bar\bcal}^a_k \weakly \bar\bcal^a\)
weakly in \(L^{p}((0,L);\R^{3\times
 2})\), \(\tilde \psi^b_k \weakly \psi^b\)
weakly in
 \(W^{1,p}(\Omega^b;\R^3)\),  \( \tilde{\bar\bcal}^b_k
\weakly \bar\bcal^b\) weakly in \(L^{p}(\omega^b;\R^3)\),
 and \vspace{-.5mm}
 \begin{equation*}
\begin{aligned}
\limsup_{k\to\infty}
 \Big(F_{\epsi_{n_k}}^a (\tilde\psi^a_{k})
+ \tfrac{h_{\varepsilon_{n_k}}}{r_{\varepsilon_{n_k}}^2}F_{\epsi_{n_k}}^b (\tilde\psi^b_{k})\Big)
\leq \kappa.
\end{aligned}
\end{equation*}
\end{lemma}

\begin{proof}
Without
loss of generality,  we may assume 
that \(W\geq 0\) by
the lower bound
in \eqref{pgrowth} and that  \(r_{\epsi_n} \leq 1\) and \(h_{\epsi_n} \leq 1\) for all \(n\in\NN\).
By  \eqref{eq:assonbc}, there
exist \(\vartheta^a
\in L^p((0,L);\R^{3\times  2}) \) and
\(\vartheta^b
\in L^{p}(\omega^b;\R^3) \) such that,
up to a not relabeled
subsequence, we have \vspace{-.5mm}
\begin{equation*}
\begin{aligned}
&{r_{\epsi_n}^{-1}}\int_{\omega^a}\grad_\alpha
\ffi^a_n
\,\dx_\alpha \weakly \vartheta^a \hbox{
weakly in } L^p((0,L);\R^{3\times
 2}), \quad \grad_\alpha \ffi^a_n
\to 0 \hbox{  in } L^{p}(\Omega^a;\R^{3\times
2}), \\
& {h_{\epsi_n}^{-1}}\int_{-1}^0\grad_3
\ffi^b_n
\,\dx_3 \weakly \vartheta^b \hbox{ weakly
in } L^p(\omega^b;\R^3),
\quad \grad_3 \ffi^b_n
\to 0 \hbox{  in } L^{p}(\Omega^b;\R^3).
\end{aligned} 
\end{equation*}
Consequently, \(\grad_\alpha \ffi^a
= 0\) and \(\grad_3
\ffi^b = 0\); thus, \(\ffi^a\) is independent
of \(x_\alpha\)
and \(\ffi^b\) is independent of \(x_3\).

Next, we define two sequences of positive
Radon measures,
\((\nu^a_n)_{n\in\Nb} \subset \M(0,L)\)
and \((\nu^b_n)_{n\in\Nb} \subset \M(\omega^b)\),
by setting, for  \(n\in\Nb\),  \(B^a\in
\calB(0,L)\),
and
 \(B^b\in \calB(\omega^b)\),
\vspace{-.5mm}%
\begin{equation*}
\begin{aligned}
 \nu^a_n(B^a):= &\int_{\omega^a \times
B^a} \big(1+ |r_{\epsi_n}^{-1}\grad_\alpha
\psi^a_n|^p+ |r_{\epsi_n}^{-1}\grad_\alpha
\ffi^a_n|^p
+ |\grad_3\psi^a_n|^p+ |\grad_3\ffi^a_n|^p+
 |\grad_3\psi^a|^p+
 |\grad_3\ffi^a|^p\big) \,\dx\\
\end{aligned}
\end{equation*}
 and
\vspace{-.5mm}%
\begin{equation*}
\begin{aligned}
 \nu^b_n(B^b):= &\int_{ B^b\times (-1,0)}
\big(1+ |\grad_\alpha
\psi^b_n|^p+ |\grad_\alpha
\ffi^b_n|^p
+ |h_{\epsi_n}^{-1}\grad_3\psi^b_n|^p+
|h_{\epsi_n}^{-1}\grad_3\ffi^b_n|^p+  |\grad_\alpha\psi^b|^p+
 |\grad_\alpha\ffi^b|^p\big)
\,\dx.
\end{aligned}
\end{equation*}

Using \eqref{pgrowth}, \eqref{eq:limL},
and \eqref{eq:assonbc},
we may extract subsequences \((\nu^a_{n_j})_{j\in\Nb}
\)
and \((\nu^b_{n_j})_{j\in\Nb} \)
of \((\nu^a_n)_{n\in\Nb} \)
and \((\nu^b_n)_{n\in\Nb} \),
respectively, such that \(\nu^a_{n_j}
\weaklystar \nu^a\)
weakly-\(\star\) in \(\M(0,L)\) and
\(\nu^b_{n_j}
\weaklystar \nu^b\)
weakly-\(\star\) in \(\M(\omega^b)\)
for some \(\nu^a\in
\M(0,L)\) and \(\nu^b\in
\M(\omega^b)\).

Fix \(\tau>0\). Because \((1+|{r_{\epsi_{n_j}}^{-1}}\nabla_\alpha
\ffi^a_{n_j}|^p + |\nabla_3\ffi^a_{n_j}|^p+|\nabla_3\ffi^a|^p
+ |\nabla_3\psi^a|^p)_{j\in\Nb}\) and
\((1+
|\grad_\alpha \ffi^b_{n_j}|^p + |h_{\epsi_{n_j}}^{-1}\grad_3\ffi^b_{n_j}|^p+|\nabla_\alpha\ffi^b|^p
+ |\nabla_\alpha\psi^b|^p)_{j\in\Nb}\)
are equi-integrable, there exists  \(\epsilon\in
(0,\tau)\)
 such that for every measurable set
 \(B\subset\R^3\) with \(|B|
< \epsilon\), we have  \vspace{-.5mm}
\begin{equation}
\label{eq:byequibc}
\begin{aligned}
&\sup_{j\in\Nb}\bigg\{\int_{\Omega^a\cap
B} (1+|{r_{\epsi_{n_j}}^{-1}}\nabla_\alpha
\ffi^a_{n_j}|^p + |\nabla_3\ffi^a_{n_j}|^p+|\nabla_3\ffi^a|^p
+ |\nabla_3\psi^a|^p)\,\dx \\ &\qquad+
\int_{\Omega^b\cap
B} (1+
|\grad_\alpha \ffi^b_{n_j}|^p + |h_{\epsi_{n_j}}^{-1}\grad_3\ffi^b_{n_j}|^p+|\nabla_\alpha\ffi^b|^p
+ |\nabla_\alpha\psi^b|^p)\,\dx\bigg\}< \tau. \end{aligned}
\end{equation}

For \(t>0\), let \(A^b_t:= \{x\in\omega^b\!:\,
\dist(x,\partial\omega^b)>t\}\). Fix
\(\eta=\eta(\tau)>0\)
such that  \(r_{\epsi_{n_j}}\omega^a
\subset A^b_\eta\)
for all \(j\in\Nb\) and \(\big(|\omega^a\times
[L-\eta,L)|
+ |(\omega^b\backslash\overline
A^b_\eta)\times(-1,0)|\big)<\epsilon\);
for \(\delta\in
(0,\tfrac\eta2)\),
define the subsets \(I^a_\delta:= (L-\eta-\delta,
L-\eta+2\delta)\)
and \(I^b_\delta:= A^b_{\eta-2\delta}
\backslash \overline
 A^b_{\eta+\delta}\), and consider smooth
cut-off functions
 \(\phi^a_{\eta,\delta} \in C^\infty_c(\R;[0,1])\)
and
 \(\phi^b_{\eta,\delta} \in C^\infty_c(\R^2;[0,1])\)
such that
 \(\phi^a_{\eta,\delta}(x_3) =0\) if
\(|x_3|>L-\eta+\delta\),
 \(\phi^a_{\eta,\delta}(x_3) =1\) if
\(|x_3|<L-\eta\),
 \(\phi^b_{\eta,\delta}(x_\alpha) =0\)
if \(x_\alpha
\not\in
 A^b_{\eta-\delta}\), and  \(\phi^b_{\eta,\delta}(x_\alpha)=
1\) if \(x_\alpha \in
 A^b_{\eta}\). Because both the length
of the interval
\(I^a_\delta\) and the thickness of
the strip \(I^b_\delta\)
 are of the order \(\delta\), there
exists a constant
 \(C\), independent of \(\delta\), such
that \(\Vert\grad_3
 \phi^a_{\eta,\delta}\Vert_{\infty}
+\Vert\grad_\alpha
\phi^b_{\eta,\delta}\Vert_{\infty} \leq
 C/\delta\).
Finally,
define \vspace{-.0mm}
\begin{equation*}
\begin{aligned}
& \vartheta^a_{j,\delta,\eta}(x):= \psi^a_{n_j}(x)
\phi^a_{\eta,\delta}(x_3)
+ (\ffi^a_{n_j}(x) - \ffi^a(x_3) + \psi^a(x_3))(1-
\phi^a_{\eta,\delta}(x_3)),
\\
& \vartheta^b_{j,\delta,\eta}(x):= \psi^b_{n_j}(x)
\phi^b_{\eta,\delta}(x_\alpha)
+ (\ffi^b_{n_j}(x) - \ffi^b(x_\alpha)
+ \psi^b(x_\alpha))(1-
\phi^b_{\eta,\delta}(x_\alpha)).
\end{aligned}
\end{equation*}
Because \(\phi^a_{\eta,\delta}(L)=0\),
\(\ffi^a(L) =
\psi^a(L)\), and \( \phi^b_{\eta,\delta}
= 0\) and \(\ffi^b
=
\psi^b\) on \(\partial \omega^b\), we
have \vspace{-.5mm}
\begin{equation}
\label{eq:barpsijbc}
\begin{aligned}
\vartheta^a_{j,\delta,\eta} = \ffi^a_{n_j}
\hbox{ on
} \omega^b\times \{L\} \,\hbox{ and
}\, \vartheta^b_{j,\delta,\eta}
= \ffi^b_{n_j} \hbox{ on
} \partial\omega^b\times (-1,0).
\end{aligned}
\end{equation}Also, for \aev\ \(x_\alpha\in \omega^a\),
 \(\phi^a_{\eta,\delta}(0)=\phi^b_{\eta,\delta}(r_{\epsi_n}
x_\alpha)=1\) and \(\psi^a_{n_j}(x_\alpha,0)
= \psi^b_{n_j}
(r_{\epsi_n} x_\alpha,0)\) by the definition
of \(\mathcal{A}_{\epsi_n}\); hence,
 for \aev\ \(x_\alpha\in \omega^a\),
\vspace{-.5mm}%
\begin{equation}
\label{eq:barpsijjunct}
\begin{aligned}
\vartheta^a_{j,\delta,\eta}(x_\alpha,0)
= \vartheta^b_{j,\delta,\eta}
(r_{\epsi_n} x_\alpha,0).
\end{aligned}
\end{equation}
Moreover, \vspace{-.5mm}
\begin{equation*}
\begin{aligned}
& \grad_\alpha \vartheta^a_{j,\delta,\eta}
= \phi^a_{\eta,\delta}
\grad_\alpha  \psi^a_{n_j} +(1- \phi^a_{\eta,\delta})\grad_\alpha
 \ffi^a_{n_j},\\
& \grad_3 \vartheta^a_{j,\delta,\eta}
= \phi^a_{\eta,\delta}
\grad_3  \psi^a_{n_j} +(1- \phi^a_{\eta,\delta})(\grad_3\ffi^a_{n_j}
-\grad_3\ffi^a + \grad_3\psi^a)  + \grad_3
 \phi^a_{\eta,\delta}(\psi^a_{n_j} 
-\psi^a+ \ffi^a-
\ffi^a_{n_j}  ),\\
& \grad_\alpha \vartheta^b_{j,\delta,\eta}
= \phi^b_{\eta,\delta}
\grad_\alpha  \psi^b_{n_j} +(1- \phi^b_{\eta,\delta})(\grad_\alpha\ffi^b_{n_j}
-\grad_\alpha\ffi^b + \grad_\alpha\psi^b)  
+(\psi^b_{n_j}  -\psi^b+ \ffi^b-
\ffi^b_{n_j}  )\otimes 
\grad_\alpha\phi^b_{\eta,\delta} ,\\
&  \grad_3 \vartheta^b_{j,\delta,\eta}
= \phi^b_{\eta,\delta}
\grad_3\psi^b_{n_j} +(1- \phi^b_{\eta,\delta})\grad_3\ffi^b_{n_j},
\end{aligned}
\end{equation*}
and, passing to the limit as \(j\to\infty\),
\begin{equation}
\label{eq:limitsbarpsij}
\begin{aligned}
& \vartheta^a_{j,\delta,\eta} \to_j
\psi^a \hbox{ in
} L^p(\Omega^a;\R^3),\enspace {r_{\epsi_{n_j}}^{-1}}\int_{\omega^a}\grad_\alpha
\vartheta^a_{j,\delta,\eta}
\,\dx_\alpha \weakly_j \phi^a_{\eta,\delta}
\bar\bcal^a +  (1- \phi^a_{\eta,\delta})\vartheta^a
\hbox{ weakly in } L^p((0,L);\R^{3\times
 2}),   \\
& \vartheta^b_{j,\delta,\eta} \to_j
\psi^b \hbox{ in
} L^p(\Omega^b;\R^3), \enspace {h_{\epsi_{n_j}}^{-1}}\int_{-1}^0\grad_3
\vartheta^b_{j,\delta,\eta}\dx_3
\weakly_j \phi^b_{\eta,\delta}
\bar\bcal^b +  (1- \phi^b_{\eta,\delta})\vartheta^b
\hbox{ weakly in } L^p(\omega^b;\R^3).
\end{aligned}
\end{equation}

Next, we estimate the energies \(F^a_{\epsi_{n_j}}\)
and \(F^b_{\epsi_{n_j}}\) at \(\vartheta^a_{j,\delta,\eta}\)
and \(\vartheta^b_{j,\delta,\eta}\),
respectively. Because
\(W\geq0\), using \eqref{pgrowth}, \eqref{eq:byequibc},
and the definitions of \(\vartheta^a_{j,\delta,\eta}\)
 and \(\nu^a_{n_j}\), we have
\vspace{-1mm}%
\begin{equation*}
\begin{aligned}
F^a_{\epsi_{n_j}}(\vartheta^a_{j,\delta,\eta})
&\leq
F^a_{\epsi_{n_j}}(
\psi^a_{n_j} ) + \int_{\omega^a \times
(L-\eta+\delta,
L)} W( r_{\epsi_{n_j}}^{-1} \grad_\alpha
 \ffi^a_{n_j}|\grad_3\ffi^a_{n_j}
-\grad_3\ffi^a + \grad_3\psi^a)\,\dx\\
&\qquad +\int_{\omega^a \times (L-\eta,
L-\eta + \delta)}
W( r_{\epsi_{n_j}}^{-1} \grad_\alpha
 \vartheta^a_{j,\delta,\eta}
\,|\grad_3\vartheta^a_{j,\delta,\eta}
\,)\,\dx \\
&\leq F^a_{\epsi_{n_j}}(
\psi^a_{n_j} ) + C\tau+ C\bigg( \nu^a_{n_j}(I^a_\delta
) + \frac{1}{\delta^p}
\int_{\omega^a\times I^a_\delta} |\psi^a_{n_j}
 -\psi^a+
\ffi^a- \ffi^a_{n_j} |^p\,\dx\bigg).
\end{aligned}
\end{equation*}
A similar estimate holds for \(F^b_{\epsi_{n_j}}(\vartheta^b_{j,\delta,\eta})
\); thus, fixing \(\eta\) and \(\delta\)
and letting
\(j\to\infty\), from the convergences
in the assumption
of Lemma~\ref{lem:ontrace} and \eqref{ell},
we obtain
\vspace{-1mm}%
\begin{equation*}
\begin{aligned}
&\limsup_{j\to\infty} \Big( F^a_{\epsi_{n_j}}
(\vartheta^a_{j,\delta,\eta})
+ \tfrac{h_{\varepsilon_{n_j}}} {r_{\varepsilon_{n_j}}^2}F^b_{\epsi_{n_j}}(\vartheta^b_{j,\delta,\eta})
\Big)
\\
&\quad\leq \limsup_{j\to\infty} \Big(
F^a_{\epsi_{n_j}}(
\psi^a_{n_j} ) + \tfrac{h_{\varepsilon_{n_j}}}
{r_{\varepsilon_{n_j}}^2}F^b_{\epsi_{n_j}}(
\psi^b_{n_j} ) \Big) + C\tau + C\limsup_{j\to\infty}
\big( \nu^a_{n_j}(I^a_\delta) + \nu^b_{n_j}(I^b_\delta)\big)
\\
&\qquad + \frac{C}{\delta^p}\limsup_{j\to\infty}
\bigg(
\int_{\omega^a\times I^a_\delta} |\psi^a_{n_j}
 -\psi^a+
\ffi^a- \ffi^a_{n_j} |^p\,\dx + \int_{
I^b_\delta \times
(-1,0)} |\psi^b_{n_j}  -\psi^b+
\ffi^b- \ffi^b_{n_j} |^p\,\dx  \bigg)
\\
&\quad \leq \kappa + C\tau + C\big(\nu^a(\overline
I^a_\delta)
+ \nu^b(\overline I^b_\delta) \big).
\end{aligned}
\end{equation*}
Letting \(\delta\to0^+\), it follows
that
\vspace{-1mm}%
\begin{equation*}
\begin{aligned}
\limsup_{\delta\to0^+}\limsup_{j\to\infty}
\Big( F^a_{\epsi_{n_j}}(\vartheta^a_{j,\delta,\eta})
+ \tfrac{h_{\varepsilon_{n_j}}}{r_{\varepsilon_{n_j}}^2} F^b_{\epsi_{n_j}}(\vartheta^b_{j,\delta,\eta})
\Big)
\leq \kappa + C\tau + C\big(\nu^a(\{L-\eta\})
+ \nu^b(\partial A^b_\eta) \big).
\end{aligned}
\end{equation*}
Hence, choosing sequences \((\eta_k)_{k\in\Nb}\)
and
 \((\delta_k)_{k\in\Nb}\)  such that
\(\eta_k\to 0^+\) as \(k\to\infty\)
and, for all \(k\in\Nb\), \(\nu^a(\{L-\eta_k\})
= \nu^b(\partial A^b_{\eta_k}) = 0\)
and
 \(0<2\delta_k
< \eta_k\),  we obtain
\vspace{-1mm}%
\begin{equation*}
\begin{aligned}
\limsup_{k\to\infty}\limsup_{j\to\infty}
\Big( F^a_{\epsi_{n_j}}(\vartheta^a_{j,\delta_k,\eta_k})
+ \tfrac{h_{\varepsilon_{n_j}}}{r_{\varepsilon_{n_j}}^2}F^b_{\epsi_{n_j}}(\vartheta^b_{j,\delta_k,\eta_k})
\Big)
\leq \kappa + C\tau.
\end{aligned}
\end{equation*}
Thus, letting \(\tau\to0^+\), we
have
\vspace{-1mm}%
\begin{equation}\label{eq:limsupL}
\begin{aligned}
\limsup_{k\to\infty}\limsup_{j\to\infty}
\Big( F^a_{\epsi_{n_j}}(\vartheta^a_{j,\delta_k,\eta_k})
+ \tfrac{h_{\varepsilon_{n_j}}}{r_{\varepsilon_{n_j}}^2} F^b_{\epsi_{n_j}}(\vartheta^b_{j,\delta_k,\eta_k})
\Big)
\leq \kappa.
\end{aligned}
\end{equation}
Finally,  \eqref{pgrowth}
and \eqref{eq:limsupL}
imply that   \(r_{\epsi_{n_j}}^{-1}\int_{\omega^a}\grad_\alpha
\vartheta^a_{j,\delta_k,\eta_k}
\,\dx_\alpha\) and  \( h_{\epsi_{n_j}}^{-1}\int_{-1}^0\grad_3
\vartheta^b_{j,\delta_k,\eta_k}\dx_3\)
admit bounds in
\(L^p((0,L);\R^{3\times
 2})\) and \(L^p(\omega^b;\R^3)\), respectively,
that
 are independent of \(k\) and \(j\).
Because \(\phi^a_{\eta_k,\delta_k}
 \to1\) in \(L^p(0,L)\) and \(\phi^b_{\eta_k,\delta_k}
 \to1\) in \(L^p(\omega^b)\) and because
the weak topology
 is metrizable on bounded sets, \eqref{eq:barpsijbc},
 \eqref{eq:barpsijjunct}, \eqref{eq:limitsbarpsij},
 \eqref{eq:limsupL}, and \eqref{pgrowth}
yield the existence
 of a sequence \((j_k)_{k\in\Nb}\) 
such that \vspace{-1mm}
\begin{equation*}
\begin{aligned}
\tilde\psi^a_k:= \vartheta^a_{j_k,\delta_k,\eta_k},
\enspace
\tilde{\bar\bcal}^a_k :={r_{\epsi_{n_{j_k}}}^{-1}}\int_{\omega^a}\grad_\alpha
\vartheta^a_{j_k,\delta_k,\eta_k}
\,\dx_\alpha, \enspace \tilde\psi^b_k:= \vartheta^b_{j_k,\delta_k,\eta_k},
\enspace \hbox{and }\,
\tilde{\bar\bcal}^b_k :={h_{\epsi_{n_{j_k}}}^{-1}}\int_{-1}^0\grad_3
\vartheta^a_{j_k,\delta_k,\eta_k}
\,\dx_3
\end{aligned}
\end{equation*}
satisfy the requirements.\end{proof}

\subsection{Proof of Theorem~\ref{thm:ellr+}}\label{Subs:l+proof}

\begin{proof}[Proof of Theorem~\ref{thm:ellr+}]
Let \((\psi^a_\epsi,\psi^b_\epsi)_{\epsi>0}\)
be as in the statement of Theorem~\ref{thm:ellr+}; that is, a sequence
in \(W^{1,p}(\Omega^a;{\mathbb
        R}^3) \times W^{1,p}(\Omega^b;{\mathbb
        R}^3)\) satisfying \( \psi^a_\epsi =
\ffi^a_{0,\epsi}\) on \(\Gamma^a\),
\(
\psi^b_\epsi =   \ffi^b_{0,\epsi}\)
on  \(\Gamma^b\), \(\psi^a_\epsi(x_\alpha,0)
=
\psi^b_\epsi(r_\varepsilon x_\alpha,0)\)
for \aev\
\(x_\alpha\in\omega^a\), and \vspace{-1mm}
\begin{equation}
        \label{eq:diaginf}
        \begin{aligned}
                E^a_\epsi(\psi^a_\epsi) + E^b_\epsi(\psi^b_\epsi)
                < \inf_{(\psi^a,\psi^b) \in
                        \Phi_\epsi}
                \big(E^a_\varepsilon(\psi^a) + E^b_\varepsilon(\psi^b)\big)
                + \rho(\epsi),
        \end{aligned}
\end{equation}
where \(\rho\) is a non-negative
function satisfying \(\rho(\epsi) \to
0\) as \(\epsi\to0^+\) and  \(E^a_\epsi\), \(E^b_\epsi\), and \(\Phi_\epsi\) are given by \eqref{Eabepsi} and \eqref{Phiepsi}. Note that by \eqref{Laepsi},
\eqref{Lbepsi}, and 
\eqref{forcesl+}, we have
\vspace{-1mm}%
\begin{equation*}
        \begin{aligned}
                L^a_\varepsilon(\psi^a_\epsi)=\int_{\Omega^a}
                f^a\cdot\psi^a_\epsi\,\dx +
                \int_{S^a} g^a\cdot\psi^a_\epsi\,\d{\HH}^2(x)
                +
                \int_{\Omega^a}
                \Bcal^a:\Big(\frac{1}{r_\varepsilon}\nabla_\alpha\psi^a_\epsi
                \Big|0
                \Big)\,\dx
        \end{aligned}
\end{equation*}
and
\vspace{-2mm}%
\begin{equation*}
        \begin{aligned}
                L^b_\varepsilon(\psi^b_\epsi)=&\int_{\Omega^b}
                f^b\cdot\psi^b_\epsi\,\dx + 
                \int_{\omega^b\backslash r_\varepsilon
                        \overline\omega^a}
                (g^{b,+}\cdot \psi^{b,+}_\epsi -g^{b,-}
                \cdot\psi^{b,-}_\epsi)\,\dx_\alpha
                \\&+
                \int_{\omega^b\backslash r_\varepsilon\overline\omega^a}
                G^b\cdot\Big(\frac{\psi^{b,+}_\epsi-\psi^{b,-}_\epsi}
                {h_\varepsilon}\Big)
                \,\dx_\alpha
                -
                \int_{r_\varepsilon \omega^a}\hat
                g^{b,-}\cdot\psi^{b,-} _\epsi\,\dx_\alpha
                -
                \frac{1}{h_\varepsilon}\int_{r_\varepsilon
                        \omega^a}\hat
                G^b\cdot\psi^{b,-}_\epsi\,\dx_\alpha.
        \end{aligned}
\end{equation*}
Also, recalling that  \(\bar\bcal^a_\epsi={r_{\varepsilon}^{-1}}
\int_{\omega^a}\nabla_\alpha
\psi^a_\epsi\, \dx_\alpha\) and \(\bar\bcal^b_\epsi=
{h_\epsi^{-1}}\int_{-1}^{0}\nabla_3
\psi^b_\epsi\,\dx_3
\), \vspace{-.0mm}
%
\begin{align}
        \int_{\Omega^a}
        \Bcal^a:\Big(\frac{1}{r_\varepsilon}\nabla_\alpha\psi^a_\epsi
        \Big|0
        \Big)\,\dx = 
        \int_{0}^L
        \Bcal^a:(\bar\bcal^a_\epsi|0)\,\dx_3, \,\,
        \int_{\omega^b\backslash r_\varepsilon\overline\omega^a}
        G^b\cdot\Big(\frac{\psi^{b,+}_\epsi-\psi^{b,-}_\epsi}
        {h_\varepsilon}\Big)
        \,\dx_\alpha = \int_{\omega^b\backslash r_\varepsilon\overline\omega^a}
        G^b\cdot
        \bar\bcal^b_\epsi\,
        \dx_\alpha, \label{Lbcala}
\end{align}
%
and, using \eqref{junction} and a change
of variables, %
\begin{equation}\label{Ljunction}
        \begin{aligned}
                \displaystyle{\frac{1}{h_\epsi} \int_{r_\epsi
                                \omega^a}
                        \hat G^b(x_\alpha)\cdot \psi^{b,-}_\epsi(x_\alpha)\,\dx_\alpha}
              & = 
                \displaystyle{\frac{1}{h_\epsi} \int_{r_\epsi
                                \omega^a}
                        \hat G^b(x_\alpha)\cdot (\psi^{b,+}_\epsi(x_\alpha)
                        + \psi^{b,-}_\epsi(x_\alpha)- \psi^{b,+}_\epsi(x_\alpha))\,\dx_\alpha
                } 
                \\
                &=\displaystyle{\frac{1}{h_\epsi} \int_{r_\epsi
                                \omega^a}
                        \hat G^b(x_\alpha)\cdot \psi^{b,+}_\epsi(x_\alpha)\,
                        \dx_\alpha - \int_{r_\epsi \omega^a}\hat
                        G^b(x_\alpha)\cdot \bar\bcal^b_\epsi(x_\alpha)\,
                        \dx_\alpha
                }\\
                &=\displaystyle{\frac{r_\epsi^2}{h_\epsi}
                        \int_{ \omega^a}
                        \hat G^b(r_\epsi x_\alpha)\cdot \psi^a_\epsi(x_\alpha,0)\,\dx_\alpha
                        -\int_{r_\epsi \omega^a}\hat G^b(x_\alpha)\cdot
                        \bar\bcal^b_\epsi(x_\alpha)\,
                        \dx_\alpha .}
        \end{aligned}
\end{equation}

Because \(0_\alpha\) is a Lebesgue point
of \(|\hat G^b|^q\),
 the Vitali--Lebesgue  theorem yields \(\hat
G(r_\epsi\cdot) \to \hat G(0_\alpha)\)
in \(L^q(\mathfrak{C};\RR^3)\);
in particular, in \(L^q(\omega^a;\RR^3) \) because \(\omega^a \subset \mathfrak{C}\) by hypothesis.
Then, taking \((\ffi^a_{\epsi,0},\ffi^b_{\epsi,0})\)
as a test
function on the right-hand side of \eqref{eq:diaginf},
from \eqref{pgrowth},  \eqref{bca}--\eqref{bcb},
Holder's inequality, the continuity
of the trace (from \(W^{1,p}\) into
\(L^p\)), and the fact that \(\ell\in\RR^+\),   we
conclude that \vspace{-.5mm}
\begin{equation*}
        \begin{aligned}
                \sup_{\epsi>0}\big(E^a_\epsi(\psi^a_\epsi)
                + E^b_\epsi(\psi^b_\epsi) \big)<\infty.
        \end{aligned}
\end{equation*}
This estimate, \eqref{pgrowth},
Young's  inequality,
Poincar\'e's inequality together with
\eqref{bca}--\eqref{bcb},  the continuity
of the trace (from \(W^{1,p}\) into
\(L^p\)), and the fact that \(\ell\in\RR^+\) yield
\vspace{-1mm}%
\begin{equation*}
        \begin{aligned}
                \sup_{\epsi>0}\Big(\|\psi^a_\epsi\|_{W^{1,p}(\Omega^a;\RR^3)}
                + \|r_{\epsi}^{-1}\grad_\alpha \psi^a_\epsi
                \|_{L^p(\Omega^a;\RR^{3\times 2})} +
                \|\psi^b_\epsi\|_{W^{1,p}(\Omega^b;\RR^3)}
                + \|h_{\epsi}^{-1}\grad_3 \psi^b_\epsi
                \|_{L^p(\Omega^b;\RR^3)}\Big)<\infty.
        \end{aligned}
\end{equation*}
Thus, the sequences 
\((\bar \bcal^a_\epsi,\psi^a_\epsi)_{\epsi>0}\)
and \((\psi^b_\epsi,\bar
\bcal^b_\epsi)_{\epsi>0}\) are sequentially, weakly compact
in
\(L^p((0,L);\R^{3
        \times
        2}) \times W^{1,p}(\Omega^a;\R^3)  \)
and \(W^{1,p}(\Omega^b;\R^3)\times L^p(\omega^b;\R^{3})\),
respectively. Let \((\bar\bcal^a, \psi^a)\)
and \(( \psi^b,\bar
\bcal^b)\) be corresponding accumulation
points. By Lemma~\ref{lem:l21GGLM},
\((\psi^a,\psi^b)\in
\A_{\ell_+}^p\) (see \eqref{Aell+}). Moreover,
 \( \psi^a = \ffi^a_{0}\) on
\(\Gamma^a\) and \(  \psi^b =   \ffi^b_{0}\)
on  \(\Gamma^b\) by the continuity
of the trace.  Hence,
\((\psi^a,\psi^b)\in \Phi_{\ell_+}^p\)(see \eqref{Phipl+}). 

Next, we show that
\vspace{-2mm}%
\begin{equation}
        \label{eq:limtotalforce}
        \begin{aligned}
                \lim_{\epsi\to0^+} \Big(L^a_\epsi(\psi^a_\epsi)
                + \frac{h_\epsi}{r_\epsi^2}
                L^b_\epsi(\psi^b_\epsi)
                \Big) = L^a(\bar\bcal^a,\psi^a) + \ell
                L^b(\psi^a,\psi^b, \bar\bcal^b),
        \end{aligned}
\end{equation}
where, for  \(\bar f^a\), \(\bar g^a\), and
\(\bar f^b\) given
by \eqref{barforces}, \vspace{-1mm} 
\begin{equation*}
        \begin{aligned}
                L^a(\bar\bcal^a,\psi^a):= \int_0^L
                \big(
                \bar f^a\cdot\psi^a+ \bar g^a\cdot\psi^a
                + 
                \Bcal^a:(\bar\bcal^a|0)\big)\,\dx_3
        \end{aligned}
\end{equation*}
and \vspace{-1mm}
\begin{equation*}
        \begin{aligned}
                L^b(\psi^a,\psi^b, \bar\bcal^b):=\int_{\omega^b}
                \big(\bar f^b\cdot\psi^b
                +
                (g^{b,+}-g^{b,-})\cdot\psi^{b}
                + G^b\cdot\bar\bcal^b\big)
                \,\dx_\alpha - \frac{\bar a}{\ell}\,
                \hat G^b(0_\alpha)\cdot \psi^a(0_3).
        \end{aligned}
\end{equation*}

By \eqref{Lbcala},
the equality \vspace{-1mm}
\begin{equation*}
        \begin{aligned}
                \lim_{\epsi\to0^+} L^a_\epsi(\psi^a_\epsi)
                = L^a(\bar\bcal^a,\psi^a)
        \end{aligned}
\end{equation*}
is an immediate consequence of the convergence
\(\psi^a_\epsi\weakly
\psi^a\) weakly in \(W^{1,p}(\Omega^a;\RR^3)\)
together with the continuity of the
trace and of the convergence  \(\bar\bcal^a_\epsi\weakly
\bar\bcal^a\) weakly in \(L^p((0,L);\RR^{3\times2}
)\).

Similarly, in view of \eqref{Lbcala},
\eqref{Ljunction}, \eqref{ell}, and because   \(\psi^b_\epsi\weakly
\psi^b\) weakly in \(W^{1,p}(\Omega^b;\RR^3)\),
\(g^{b,\pm}\chi_{\omega^b
        \backslash r_\epsi\overline \omega^a}
\to g^{b,\pm}\)
in \(L^q(\omega^b;\RR^3)\), \(G^b\chi_{\omega^b
        \backslash r_\epsi\overline \omega^a}
\to G^b\)
in \(L^q(\omega^b;\RR^3)\), \(\bcal^b_\epsi\weakly \bar\bcal^b\)
weakly  in \(L^p(\omega^b;\RR^3)\),
\(\hat g^{b,-}\chi_{ r_\epsi \omega^a}
\to 0\)
in \(L^q(\omega^a;\RR^3)\), 
\(\hat G^b \chi_{r_\epsi \omega^a} \to
0\) in  \(L^q(\omega^b;\RR^3)\), and
\(\hat
G(r_\epsi\cdot) \to \hat G(0_\alpha)\)
in \(L^q({\omega^a};\RR^3)\), it follows
that \vspace{-1mm}
\begin{equation*}
        \begin{aligned}
                \lim_{\epsi\to0^+} L^b_\epsi(\psi^b_\epsi)
                = L^b(\psi^a,\psi^b,
                \bar\bcal^b).
        \end{aligned}
\end{equation*}  

Consequently, using \eqref{ell} once
more, we conclude that \eqref{eq:limtotalforce}
holds.

To simplify the notation, in the remaining
part of the
proof, we set \vspace{-.5mm}
\begin{equation*}
        \begin{aligned}
                \mathcal{X}:=\big(
                L^p((0,L);\R^{3 \times
2}) \times W^{1,p}(\Omega^a;\R^3)\big)\times
                \big( W^{1,p}(\Omega^b;\R^3)
\times
                L^p(\omega^b;\R^{3})
                \big) .
        \end{aligned}
\end{equation*}

Let
us now introduce, for \(0<\epsi \leq 1\),
the functionals
\(\mathcal{E}_\epsi:
\mathcal{X} \to
(-\infty,\infty]
\) and 
\(\mathcal{E}_{\ell_+}:
\mathcal{X} \to
(-\infty,\infty]
\) defined by \vspace{-1.5mm}
\begin{equation*}
        \begin{aligned}
                \mathcal{E}_\epsi((\bar\bcal^a,\psi^a), (\psi^b,
                \bar\bcal^b)):=
                \begin{cases}
                        \displaystyle
                        E_\epsi^a(\psi^a) +
                        E_\epsi^b(\psi^b) & \hbox{if }  ((\bar\bcal^a,\psi^a),(
                        \psi^b, \bar\bcal^b)
                        )\in \mathcal{A}_\epsi \hbox{ and }
                        (\psi^a,\psi^b) \in
                        \Phi_\epsi\\
                        \infty & \hbox{otherwise}
                \end{cases}
        \end{aligned}
\end{equation*}
and
\vspace{-1mm}%
\begin{equation*}
        \begin{aligned}
                \mathcal{E}_{\ell_+}((\bar\bcal^a,\psi^a),
                (\psi^b, \bar\bcal^b)):=
                \begin{cases}
                        \displaystyle
                        E_{\ell_+}((\bar\bcal^a,\psi^a), (\psi^b,\bar
                        \bcal^b)) & \hbox{if }   (\psi^a,\psi^b)
                        \in
                        \Phi_{\ell_+}^p\\
                        \infty & \hbox{otherwise,}
                \end{cases}
        \end{aligned}
\end{equation*}
respectively, where, we recall, \(\A_\epsi\),
\(\Phi_\epsi\),  \(E_{\ell_+}\), and
\(\Phi_{\ell_+}^p\)
are given by \eqref{Aepsi}, \eqref{Phiepsi},
\eqref{Eell+}, and   \eqref{Phipl+}.

Note that if \( ((\bar\bcal^a,\psi^a),(
\psi^b, \bar\bcal^b)
)\in \mathcal{A}_\epsi\) and  \((\psi^a,\psi^b)
\in
\Phi_\epsi\), then \(\mathcal{E}_\epsi((\bar\bcal^a,\psi^a),
(\psi^b, \bar\bcal^b)) = F_\epsi^a(\psi^a)
+
\frac{h_\epsi}{r_\epsi^2}F_\epsi^b(\psi^b) -L^a_\epsi(\psi^a_\epsi)
- \frac{h_\epsi}{r_\epsi^2}
L^b_\epsi(\psi^b_\epsi) \) (see \eqref{Eabepsi});
also,
if  \((\psi^a,\psi^b) \in \Phi_{\ell_+}^p\),
then \(\mathcal{E}_{\ell_+}((\bar\bcal^a,\psi^a),
(\psi^b,\bar \bcal^b)) = F^a(\psi^a)
+
\ell F^b(\psi^b) -L^a(\bar\bcal^a,\psi^a)
- \ell L^b(\psi^a,\psi^b, \bar\bcal^b) \) (see
\eqref{Eell+} and \eqref{Fab}).

We claim that \((\mathcal{E}_\epsi)_{\epsi>0}\)
\(\Gamma\)-converges
to \(\mathcal{E}_{\ell_+}\) with respect to the weak
topology in \(\mathcal{X} \).
As we showed at
the beginning of this proof, \((\mathcal{E}_\epsi)_{\epsi>0}\)
is equi-coercive with respect to the
weak topology in \(\mathcal{X} \).
Thus,  if the claim
holds, then  Theorem~\ref{thm:ellr+}
immediately follows
(see  \cite[Proposition~8.16, Theorem~7.8, and
Corollary~7.20]{DM93}).
To prove the claim, it suffices to show
that given any
subsequence \(\epsi_n\prec\epsi\), the
\(\Gamma\)-lower
limit  of \((\mathcal{E}_{\epsi_n})_{n\in\NN}\)
coincides with \(\mathcal{E}_{\ell_+}\)
(see \cite[Chapter~8]{DM93}).

We first  show that given   \( ((\bar\bcal^a_n,\psi^a_n),(\psi^b_n,
\bar\bcal^b_n))_{n\in\NN} \subset \calX\) and \( ((\bar\bcal^a,\psi^a),(\psi^b,
\bar\bcal^b)) \in \calX\)
such that \( ((\bar\bcal^a_n,\psi^a_n),\allowbreak(\psi^b_n,
\bar\bcal^b_n))\weakly ((\bar\bcal^a,\psi^a),\allowbreak
(\psi^b, \bar\bcal^b)) \) weakly in
\( \calX\), we have \vspace{-1mm}%
\begin{equation}
        \label{eq:liminfcalF}
        \begin{aligned}
                \mathcal{E}_{\ell_+}((\bar\bcal^a,\psi^a),
                (\psi^b, \bar\bcal^b))
                \leq \liminf_{n\to\infty} \mathcal{E}_{\epsi_n}
                ((\bar\bcal^a_n,\psi^a_n),(\psi^b_n,
                \bar\bcal^b_n)).
        \end{aligned}
\end{equation}
To prove \eqref{eq:liminfcalF}, we may
assume that  the  lower
limit on the right-hand
side of \eqref{eq:liminfcalF} is actually
a limit and is finite, extracting a
subsequence if necessary.  Then,  \(((\bar\bcal^a_n,\psi^a_n),(\psi^b_n,
\bar\bcal^b_n)) \in \A_{\epsi_n}\) and
\((\psi^a_n,\psi^b_n)
\in  \Phi_{\epsi_n} \) for all \(n\in\NN\).
In particular,
\(\psi^a_n = \ffi^a_{0,\epsi_n}\) on
\(\Gamma^a\) and \( \psi^b_n =   \ffi^b_{0,\epsi_n}\)
on \( \Gamma^b\). Thus, \(\psi^a = \ffi^a_{0}\)
on \(\Gamma^a\) and \(\psi^b =   \ffi^b_{0}\)
on \( \Gamma^b\). Invoking \eqref{eq:limtotalforce}
and Theorem~\ref{Thm:FepsitoF}, we deduce
\eqref{eq:liminfcalF}.

To conclude, we prove that given   \( ((\bar\bcal^a,\psi^a),(\psi^b,
\bar\bcal^b)) \in \calX\), there exists
a sequence \((({\bar\bcal}^a_n,
\psi^a_n),\allowbreak
(\psi^b_n,
{\bar\bcal}^b_n))_{n\in\NN} \allowbreak\subset
\calX\)   such that \( (({\bar\bcal}^a_n,\psi^a_n),(\psi^b_n,
{\bar\bcal}^b_n))\weakly ((\bar\bcal^a,\psi^a),(\psi^b,
\bar\bcal^b)) \) weakly in \( \calX\)
and
\vspace{-1mm}%
\begin{equation}
        \label{eq:limsupcalF}
        \begin{aligned}
                \mathcal{E}_{\ell_+}((\bar\bcal^a,\psi^a),
                (\psi^b, \bar\bcal^b))
                =\liminf_{n\to\infty} \mathcal{E}_{\epsi_n}
                ((\bar\bcal^a_n,\psi^a_n),(\psi^b_n,
                \bar\bcal^b_n)).
        \end{aligned}
\end{equation}
To establish \eqref{eq:limsupcalF},
the only non-trivial
case is the case in which \((\psi^a,\psi^b)
\in
\Phi_{\ell_+}^p\). Then, by Theorem~\ref{Thm:FepsitoF},
we can find a sequence   \(((\hat{\bar\bcal}^a_n,
\hat\psi^a_n),\allowbreak
(\hat\psi^b_n,
\hat{\bar\bcal}^b_n))_{n\in\NN} \subset
\A_{\epsi_n}\) such
that 
\( ((\hat{\bar\bcal}^a_n,
\hat\psi^a_n),\allowbreak
(\hat\psi^b_n,
\hat{\bar\bcal}^b_n))\weakly ((\bar\bcal^a,\psi^a),(\psi^b,
\bar\bcal^b)) \) weakly in \( \calX\)
and
\vspace{-1.5mm}%
\begin{equation*}
        \begin{aligned}
                F^a(\psi^a) +
                \ell F^b(\psi^b)
                =\lim_{n\to\infty} \Big(F_{\epsi_n}
                (\hat\psi^a_n)+\frac{h_{\epsi_n}}{r_{\epsi_n}^2}
                F_{\epsi_n} (\hat\psi^b_n)\Big).
        \end{aligned}
\end{equation*}
By Lemma~\ref{lem:ontrace}, we can find
a subsequence
\(\epsi_{n_k} \prec \epsi_n\) and a
sequence  
\(((\tilde{\bar\bcal}^a_k,\tilde\psi^a_k),(\tilde
\psi^b_k, \tilde{\bar\bcal}^b_k)
)_{j\in\Nb} \subset \mathcal{A}_{\epsi_{n_k}}\)
satisfying
\((\tilde\psi^a_k,\tilde\psi^b_k) \in
\Phi_{\epsi_{n_k}}\)
for all \(k\in\Nb\), 
\(((\tilde{\bar\bcal}^a_k,\tilde\psi^a_k),(
\tilde \psi^b_k,
\tilde{\bar\bcal}^b_k))  \weakly ((\bar\bcal^a,
\psi^a),(
\psi^b, \bar\bcal^b))\) weakly in
\(\calX\), 
and \vspace{-1mm}
\begin{equation}\label{eq:limsupcalF2}
        \begin{aligned}
                \limsup_{k\to\infty}
                \Big(F_{\epsi_{n_k}}^a (\tilde\psi^a_{k})
                + \frac{h_{\epsi_{n_k}}}{r_{\epsi_{n_k}}^2} F_{\epsi_{n_k}}^b (\tilde\psi^b_{k})\Big)
                \leq F^a(\psi^a) +
                \ell F^b(\psi^b).
        \end{aligned}
\end{equation}
Then, defining \((({\bar\bcal}^a_n,
\psi^a_n),\allowbreak
(\psi^b_n,
{\bar\bcal}^b_n)) :=((\tilde{\bar\bcal}^a_k,\tilde\psi^a_k),(\tilde
\psi^b_k, \tilde{\bar\bcal}^b_k)
) \)   if \(n=n_k\) and \((({\bar\bcal}^a_n,
\psi^a_n),\allowbreak
(\psi^b_n,
{\bar\bcal}^b_n)):= ((\bar\bcal^a,\psi^a),\allowbreak(\psi^b,
\bar\bcal^b))\) if \(n\not= n_k\), from
\eqref{eq:liminfcalF}, \eqref{eq:limtotalforce},
and \eqref{eq:limsupcalF2}, in this
order, we obtain
\begin{equation*}
        \begin{aligned}
                \mathcal{E}_{\ell_+}((\bar\bcal^a,\psi^a),
                (\psi^b,\bar \bcal^b)) &= F^a(\psi^a)
                +
                \ell F^b(\psi^b) -L^a(\bar\bcal^a,\psi^a)
                - \ell L^b(\psi^b,
                \bar\bcal^b)\\
                &\leq \liminf_{n\to\infty} \mathcal{E}_{\epsi_n}
                ((\bar\bcal^a_n,\psi^a_n),(\psi^b_n,
                \bar\bcal^b_n))
                \leq \limsup_{k\to\infty} \mathcal{E}_{\epsi_{n_k}}
                ((\tilde{\bar\bcal}^a_k,\tilde\psi^a_k),(\tilde
                \psi^b_k, \tilde{\bar\bcal}^b_k)
                )\\
                & = \limsup_{k\to\infty} \Big(F_{\epsi_{n_k}}^a(\tilde\psi^a_k)
                +
                \frac{h_{\epsi_{n_k}}}{r_{\epsi_{n_k}}^2}F_{\epsi_{n_k}}^b(\tilde
                \psi^b_k) -L^a_{\epsi_{n_k}}(\tilde
                \psi^a_k) - \frac{h_{\epsi_{n_k}}}{r_{\epsi_{n_k}}^2}
                L^b_{\epsi_{n_k}}(\tilde
                \psi^b_k) \Big)
                \\
                &\leq \limsup_{k\to\infty} \Big(F_{\epsi_{n_k}}^a
                (\tilde\psi^a_{k})
                + \frac{h_{\epsi_{n_k}}}{r_{\epsi_{n_k}}^2}F_{\epsi_{n_k}}^b (\tilde\psi^b_{k})\Big)
                -L^a(\bar\bcal^a,\psi^a) - \ell L^b(\psi^a,\psi^b,
                \bar\bcal^b)
                \\
                &\leq \mathcal{E}_{\ell_+}((\bar\bcal^a,\psi^a),
                (\psi^b,\bar \bcal^b)), 
        \end{aligned}
\end{equation*}
which proves \eqref{eq:limsupcalF}.
\end{proof}
\subsection{The string case}\label{Subs:rod}
 
Here, we recover the analysis of a 
 nonlinear string model with bending-torsion
moments and generalized boundary conditions
 that was carried out in \cite{RiMSc},
which provides the 3D-1D counterpart
of the study in \cite{BFM03} under more
general boundary conditions.
Roughly speaking, it corresponds to
consider the problem  \eqref{Ptepsi}
disregarding the terms in \(\Omega^b_\epsi\)
and setting a deformation condition
on \(r_\epsi \omega^a\times \{0,L\}\);
that is, on both of the extremities of
the thin tube-shaped domain \(\Omega^a_\epsi\).  
 After a similar change of variables
 and re-scaling described in the Introduction, we are then led
 to the study of the re-scaled problem
\begin{equation}\label{Pepsirod} 
\inf
\big\{ E^a_\varepsilon(\psi^a)\!:\,
\psi^a  \in
 \Phi^a_\epsi\big\} ,
\tag{$
{\mathcal{P}}^a_\epsi$}
\end{equation}
where, for \(\Gamma^a_0:=\omega^a\times\{0\}\)
and \(\Gamma^a_L:=\omega^a\times\{L\}\),
%
\(\Phi^a_\epsi := \big\{ \psi^a\in
W^{1,p}(\Omega^a;{\mathbb
R}^3) \!: \,\, \psi^a = 
\ffi^a_{0,\epsi} \hbox{ on } \Gamma^a_0
\cup \Gamma^a_L  \big\}\)
%
and, as above, \vspace{-1.5mm}
\begin{equation*}
\begin{aligned}
E^a_\varepsilon(\psi^a) =\int_{\Omega^a}
W({r_\varepsilon^{-1}}\nabla_\alpha
\psi^a|\nabla_3\psi^a)\,\dx - \int_{\Omega^a}
f^a\cdot\psi^a\,\dx -
\int_{S^a} g^a\cdot\psi^a\,\d{\HH}^2(x)
-
\int_{\Omega^a}
\Bcal^a:\Big(\frac{1}{r_\varepsilon}\nabla_\alpha\psi^a
\Big|0
\Big)\,\dx
\end{aligned}
\end{equation*}
with  \(W:\R^{3\times 3} \to \R\)  a
Borel function
satisfying
\eqref{pgrowth},   \(f^a\in L^q(\Omega^a;\RR^3)\),
\(g^a\in L^q(S^a;\RR^3)\), and \(\Bcal^a\in L^q((0,L);\RR^{3\times 3})\).

We further assume that there exists  \(\ffi^a_0 \in W^{1,p}(\Omega^a;\R^3)\)
 satisfying
\eqref{bca}.
As observed before,  the function \(\ffi^a_{0,\epsi}(x)
=
(r_\varepsilon x_\alpha, x_3) \)  corresponding
to the
clamped case, which is commonly considered
in the literature,
satisfies \eqref{bca}.

Addressing the extremity \(\Gamma^a_0\)
in an analogous way as we treated the extremity \(\Gamma^a_L\) in the 
previous two subsections, we find implicit
in the arguments in those two subsections
the proof of the following result.

\begin{theorem}\label{thm:rod1} For
\(0<\epsi\leq 1\), let  \(\mathcal{E}^a_\epsi,
\, \mathcal{E}^a:
L^p((0,L);\R^{3 \times 2}) \times W^{1,p}(\Omega^a;\R^3)
 \to (-\infty,\infty]\) be the functionals defined, for
\((\bar\bcal^a,\psi^a) \in L^p((0,L);\R^{3
\times 2}) \times W^{1,p}(\Omega^a;\R^3)
\),   by \vspace{-.5mm}
\begin{equation*}
\begin{aligned}
\mathcal{E}^a_\epsi(\bar\bcal^a,\psi^a):=
\begin{cases}
\displaystyle
E_\epsi^a(\psi^a)  & \hbox{if }  \psi^a
 \in
 \Phi^a_\epsi \hbox{ and } \frac{1}{r_{\varepsilon}}
\int_{\omega^a}\nabla_\alpha
\psi^a\, \dx_\alpha
=\bar \bcal^a\\
\infty & \hbox{otherwise}
\end{cases} 
\end{aligned}
\end{equation*}
and  \vspace{-1.mm}
\begin{equation*}
\begin{aligned}
\mathcal{E}^a(\bar\bcal^a,\psi^a):=
\begin{cases}
\displaystyle
E^a(\bar\bcal^a,\psi^a) & \hbox{if }
  \psi^a
\in
\Phi^a \\
\infty & \hbox{otherwise,}
\end{cases}
\end{aligned}
\end{equation*}
respectively, where \vspace{-.mm}
\begin{equation}\label{Phirod}
\begin{aligned}
\Phi^a := \big\{ \psi^a\in
W^{1,p}(\Omega^a;{\mathbb
R}^3) \!:& \,\, \psi^a = 
\ffi^a_{0} \hbox{ on } \Gamma^a_0
\cup \Gamma^a_L \hbox{ and } \,
  \psi^a \hbox{ is  independent of }
x_\alpha \big\}
\end{aligned}
\end{equation}
and
\vspace{-1mm}\begin{equation*}
\begin{aligned}
E^a(\bar\bcal^a,\psi^a):= &\, \bar
a\int_0^L \ce({\bar a^{-1}}\bar\bcal^a|\nabla_3\psi^a)\,\dx_3
  - \int_0^L
\big(
\bar f^a\cdot\psi^a+ \bar g^a\cdot\psi^a
+ 
\Bcal^a:(\bar\bcal^a|0)\big)\,\dx_3
\end{aligned}
\end{equation*}
with
\(\bar a = |\omega^a|\), \(\bar f^a(x_3)=
\int_{\omega^a} f^a(x)\,\dx_\alpha\),
\(
\bar g^a(x_3)= \int_{\partial\omega^a}
g^a(x)\,\d\HH^1(x_\alpha)\).
Then,  \((\mathcal{E}^a_\epsi)_{\epsi>0}\)
\(\Gamma\)-converges to \(\mathcal{E}^a\)
 with respect to the
weak
topology in \(L^p((0,L);\R^{3 \times
2}) \times W^{1,p}(\Omega^a;\R^3) \).
\end{theorem}

As a corollary to Theorem~\ref{thm:rod1},
we derive a nonlinear string model where
the applied surface forces induce a
bending-torsion effect:

\begin{theorem}\label{thm:rod2}
Let \(W:\R^{3\times 3} \to \R\) be a
Borel function
satisfying
\eqref{pgrowth} and let \((\psi^a_\epsi)_{\epsi>0}\)
be a diagonal
infimizing sequence of the sequence
of problems \eqref{Pepsirod},
where      
\((\ffi^a_{0,\epsi})_{\epsi>0}\)
satisfies
\eqref{bca}. 
  Then, the sequence
\((\bar \bcal^a_\epsi,\psi^a_\epsi)_{\epsi>0}\),
where \(\bar\bcal^a_\epsi:={r_{\varepsilon}^{-1}}
\int_{\omega^a}\nabla_\alpha
\psi^a_\epsi\, \dx_\alpha\), is sequentially,
weakly compact in
\(  L^p((0,L);\R^{3 \times 2}) \times
W^{1,p}(\Omega^a;\R^3) \). If \((\bar\bcal^a, \psi^a)\)
is an  
accumulation point, then \((\bar\bcal^a,
\psi^a)\in
 L^p((0,L);\R^{3 \times
2}) \times \Phi^a\) and it  solves the
minimization problem
\begin{equation}\label{Prod} 
\min
\big\{ E^a(\bar\bcal^a,\psi^a)\!:\, \psi^a\in \ffi^a_0 + W^{1,p}_0((0,L);\RR^3), \, \bar\bcal^a \in L^p((0,L);\R^{3 \times
2}) \big\}  .
\tag{$
{\mathcal{P}}^a$}
\end{equation}
\end{theorem}

\begin{remark}
(i) As before, in general, the term \(\bar\bcal^a\)
is not related to the one-dimensional
strain tensor of \(\psi^a\). Thus, \(\psi^a\)
and \(\bar\bcal^a\) must be regarded
as distinct macroscopic entities. Moreover, given the nature
of \(G^a\), \(\bar\bcal^a\)
accounts for bending and torsion moments
in the string. (ii)
If \(\Bcal^a\equiv0\), which means that the
term \(G^a\) in the surface applied forces with
a non-standard  order
of scaling  magnitude is not present,
then the model \eqref{Prod} reduces
to
\vspace{-.5mm}%
\begin{equation}\label{Prodnb} 
\min
\bigg\{\bar
a\int_0^L \mathcal{C}W_0(\nabla_3\psi^a)\,\dx_3
  - \int_0^L
\big(
\bar f^a\cdot\psi^a+ \bar g^a\cdot\psi^a
\big)\,\dx_3\!:\, \psi^a\in \ffi^a_0 + W^{1,p}_0((0,L);\RR^3)\bigg\} ,
\tag{$
{\widehat{\mathcal{P}}}^a$}
\end{equation}
where \(\mathcal{C}W_0\) is the convex
envelop of the function \(W_0:\RR^3\to\RR\)
defined for \(\zeta\in\RR^3\) by
\begin{equation}\label{eq:W0}
\begin{aligned}
W_0(\zeta) := \inf_{b^a\in \RR^{3\times
2}} W(b^a|\zeta).
\end{aligned}
\end{equation}
The model \eqref{Prodnb} is the 1D counterpart
of the model derived in \cite{LDR95}
and, in essence, coincides with the
model derived in \cite{ABP91}.
We note, however, that in  \cite{ABP91}, the physical condition ``\(\det \grad\psi^a >0\)" of non-interpenetration of matter is addressed.\end{remark}

\section{Case $\ell=\infty$}\label{Sect:li}
 
 

This section is devoted to the proof
of Theorem~\ref{thm:elli}, where Proposition~\ref{prop:thm6FJM}
and the results in Subsection~\ref{Subs:rod}
play an important role.

\begin{proof}[Proof of Theorem~\ref{thm:elli}]
Let \((\psi^a_\epsi,\psi^b_\epsi)_{\epsi>0}\)
be as in the statement of Theorem~\ref{thm:elli};
that is, a sequence
in \(W^{1,p}(\Omega^a;{\mathbb
R}^3) \times W^{1,p}(\Omega^b;{\mathbb
R}^3)\) satisfying \( \psi^a_\epsi =
\ffi^a_{0,\epsi}\) on \(\Gamma^a\),
\(
 \psi^b_\epsi =   \ffi^b_{0,\epsi}\)
on  \(\Gamma^b\), \(\psi^a_\epsi(x_\alpha,0)
=
\psi^b_\epsi(r_\varepsilon x_\alpha,0)\)
for \aev\
\(x_\alpha\in\omega^a\), and
\vspace{-1mm}%
\begin{equation}
\label{eq:diaginfi}
\begin{aligned}
E^a_\epsi(\psi^a_\epsi) + E^b_\epsi(\psi^b_\epsi)
< \inf_{(\psi^a,\psi^b) \in
 \Phi_\epsi}
 \big(E^a_\varepsilon(\psi^a) + E^b_\varepsilon(\psi^b)\big)
+ \rho(\epsi),
\end{aligned}
\end{equation}
where \(\rho\) is a non-negative
function satisfying \(\rho(\epsi) \to
0\) as \(\epsi\to0^+\) and  \(E^a_\epsi\), \(E^b_\epsi\), and \(\Phi_\epsi\)
are given by \eqref{Eabepsi} and \eqref{Phiepsi}. By  \eqref{Laepsi},
 \eqref{Lbepsi},  
\eqref{forcesli}, \eqref{Lbcala}, 
and \eqref{Ljunction}, we have \vspace{-1mm}
\begin{equation*}
\begin{aligned}
E^a_\epsi(\psi^a_\epsi) = F^a_\epsi(\psi^a_\epsi)
- L^a_\epsi(\psi^a_\epsi) \enspace \text{and
} \enspace E^b_\varepsilon(\psi^b_\epsi) =
\frac{h_\epsi}{r_\epsi^2} F^b_\varepsilon(\psi_\epsi
^b) - \frac{h_\epsi}{r_\epsi^2}L^b_\varepsilon(\psi^b_\epsi), 
\end{aligned}
\end{equation*}
where \(F^a_\epsi\) and \(F^b_\epsi\)
are given by \eqref{Fabepsi} and 
\vspace{-1mm}%
\begin{align*}
L^a_\varepsilon(\psi^a_\epsi)&=\int_{\Omega^a}
f^a\cdot\psi^a_\epsi\,\dx +
\int_{S^a} g^a\cdot\psi^a_\epsi\,\d{\HH}^2(x)
+
\int_{0}^L
\Bcal^a:(\bar\bcal^a_\epsi|0)\,\dx_3,\\
\frac{h_\epsi}{r_\epsi^2}L^b_\varepsilon(\psi^b_\epsi)&=\int_{\Omega^b}
f^b\cdot\psi^b_\epsi\,\dx + 
\int_{\omega^b\backslash r_\varepsilon
\overline\omega^a}
(g^{b,+}\cdot \psi^{b,+}_\epsi -g^{b,-}
\cdot\psi^{b,-}_\epsi)\,\dx_\alpha
+
\int_{\omega^b\backslash
r_\varepsilon\overline\omega^a}
G^b\cdot
\bar\bcal^b_\epsi\,
\dx_\alpha
\\&\quad-
\int_{r_\varepsilon \omega^a}\hat
g^{b,-}\cdot\psi^{b,-} _\epsi\,\dx_\alpha
-
\frac{r_\epsi^2}{h_\epsi}
\int_{ \omega^a}
\hat G^b(r_\epsi \cdot)\cdot \psi^a_\epsi(\cdot,0)\,\dx_\alpha
+\int_{r_\epsi \omega^a}\hat G^b\cdot
\bar\bcal^b_\epsi\,
\dx_\alpha .
\end{align*}

Because 
\((G^b(r_\epsi \cdot))_{\epsi>0}\)
is bounded in \(L^q(\omega^a;\RR^3)\),
taking \((\ffi^a_{\epsi,0},\ffi^b_{\epsi,0})\)
with \(\ffi^b_{\epsi,0}\equiv (x_\alpha,
h_\epsi x_3)\) as a test
function on the right-hand side of \eqref{eq:diaginfi},
from \eqref{pgrowth}, \eqref{bca}, 
\eqref{WI=0}, Holder's inequality, the continuity
of the trace (from \(W^{1,p}\) into
\(L^p\)), and the fact that \(r_\epsi^2/h_\epsi
\to0\), we
conclude that 
\begin{equation*}
\begin{aligned}
\sup_{\epsi>0}\big(E^a_\epsi(\psi^a_\epsi)
+ E^b_\epsi(\psi^b_\epsi) \big)<\infty.
\end{aligned}
\end{equation*}
From this estimate, \eqref{coercrigid2},
Young's  inequality,
 Poincar\'e's inequality together with
\eqref{bca}--\eqref{bcb},  the continuity
of the trace (from \(W^{1,p}\) into
\(L^p\)), and the fact that \(\mathcal{K}:=SO(3)
\cup SO(3) A\)
is a compact subset of \(\RR^{3\times
3}\),
we obtain
\begin{equation}\label{coercdist}
\begin{aligned}
\sup_{\epsi>0}\Big(&\Vert \dist ( (r_\epsi^{-1}\grad_\alpha
\psi^a_\epsi
|\grad_3 \psi^a_\epsi) , \mathcal{K})\Vert^p_{L^p(\Omega^a)} +\frac{h_\epsi}{r_\epsi^2}  \Vert \dist ( (\grad_\alpha
\psi^b_\epsi
|h_\epsi^{-1}\grad_3 \psi^b_\epsi) , \mathcal{K})\Vert^p_{L^p(\Omega^b)}\Big)<\infty.
\end{aligned}
\end{equation}
In particular,
using the fact that \(\mathcal{K}\)
is a compact subset of \(\RR^{3\times
3}\) and Poincar\'e's inequality together with
\eqref{bca}--\eqref{bcb} once more, and because \(\sup_{\epsi>0}
{r_\epsi^2}/{h_\epsi}<\infty\), we 
have also
\vspace{-1mm}%
\begin{equation*}
\begin{aligned}
\sup_{\epsi>0}\Big( & \|\psi^a_\epsi\|_{W^{1,p}(\Omega^a;\RR^3)}
+ \|r_{\epsi}^{-1}\grad_\alpha \psi^a_\epsi
\|_{L^p(\Omega^a;\RR^{3\times 2})} \\
&
+
 \|\psi^b_\epsi\|_{W^{1,p}(\Omega^b;\RR^3)}
+ \|h_{\epsi}^{-1}\grad_3 \psi^b_\epsi
\|_{L^p(\Omega^b;\RR^3)} +\tfrac{h_\epsi}{r_\epsi^2}
 \|h_{\epsi}^{-1}\grad_3 \psi^b_\epsi
 - d_\epsi\|^p_{L^p(\Omega^b;\RR^3)}
\Big)<\infty,
\end{aligned}
\end{equation*}
where \(d_\epsi\) is the third column
of the map \(D_\epsi:\Omega^b \to \mathcal{K}\)
satisfying
\begin{equation*}
\begin{aligned}
 \dist
( (\grad_\alpha
\psi^b_\epsi
|h_\epsi^{-1}\grad_3 \psi^b_\epsi) ,
\mathcal{K}) = | (\grad_\alpha
\psi^b_\epsi
|h_\epsi^{-1}\grad_3 \psi^b_\epsi) -
D_\epsi|.
\end{aligned}
\end{equation*}
Thus, the sequences 
\((\bar \bcal^a_\epsi,\psi^a_\epsi)_{\epsi>0}\)
and \((\psi^b_\epsi,\bar
\bcal^b_\epsi)_{\epsi>0}\) are sequentially,
weakly compact
in
\(L^p((0,L);\R^{3
\times
2}) \times W^{1,p}(\Omega^a;\R^3)  \)
and \(W^{1,p}(\Omega^b;\R^3)\times L^p(\omega^b;\R^{3})\),
respectively. Let \((\bar\bcal^a, \psi^a)\)
and \(( \psi^b,\bar
\bcal^b)\) be corresponding accumulation
points. By Lemma~\ref{lem:l21GGLM},
\((\psi^a,\psi^b)\in
\A_{\ell_+}^p\) (see \eqref{Aell+}). Moreover,
we have \( \psi^a = \ffi^a_{0}\) on
\(\Gamma^a\) and \(  \psi^b =   \ffi^b_{0}\)
on  \(\Gamma^b\) by the continuity
of the trace.  Hence,
\((\psi^a,\psi^b)\in \Phi_{\ell_+}^p\) (see \eqref{Phipl+}).

Invoking now \eqref{coercdist} and 
Proposition~\ref{prop:thm6FJM},  there exists a sequence \((M^b_\epsi)_{\epsi>0}\subset
\mathcal{K}\) of constant matrices such that
\vspace{-1mm}%
\begin{equation*}
\begin{aligned}
\Vert (\grad_\alpha
\psi^b_\epsi
|h_\epsi^{-1}\grad_3 \psi^b_\epsi) - M^b_\epsi  \Vert^p_{L^p(\Omega^b;\RR^{3\times 3})} \leq C \frac{r_\epsi^2}{h_\epsi^{p+1}}.
\end{aligned}
\end{equation*}
Extracting a subsequence if necessary,
we have that \(M^b_\epsi \to M^b\) in
\(\RR^{3\times 3}\) for some \(M^b\in
\mathcal{K}\). Then, because \(\lim_{\varepsilon \to 0}{r_\varepsilon^2}
/{h_\varepsilon}^{p+1}= 0\) by hypothesis, we have 
\begin{equation*}
\begin{aligned}
 (\grad_\alpha
\psi^b_\epsi
|h_\epsi^{-1}\grad_3 \psi^b_\epsi) \to
M^b \text{ in } {L^p(\Omega^b;\RR^{3\times
3})}.
\end{aligned}
\end{equation*}
In particular, \((\grad_\alpha
\psi^b
|\grad_3 \psi^b) \equiv (M^b_\alpha|0)\).
Then, using the fact that  \(  \psi^b =   \ffi^b_{0} = (x_\alpha,0)\)
on  \(\Gamma^b\), it follows that \(M^b_\alpha
= \II_\alpha\) and  \(  \psi^b
\equiv  (x_\alpha,0)\);  because either
 \(M^b\in
SO(3)\) or \(M^b\in
SO(3)A\) and \(\II\) and \(A\) are strongly
incompatible, we conclude that   \(M^b = \II\).
Consequently, \(\bcal^b = \bar\bcal^b\equiv \II_3\).
Moreover, since \(\psi^b(0_\alpha) =0
\), we conclude  that \(\psi^a\in 
\Phi^p_{\ell_\infty }\) (see \eqref{Phipli}).

Next, we observe that, arguing as in the proof of
Theorem~\ref{thm:ellr+} and using the
fact that \(\lim_{\varepsilon
\to 0}{r_\epsi^2}/{h_\epsi} = 0\), we
conclude that if 
\(\psi^a_\epsi\weakly
\psi^a\) weakly in \(W^{1,p}(\Omega^a;\RR^3)\),
\(\psi^b_\epsi\weakly
\psi^b\) weakly in \(W^{1,p}(\Omega^b;\RR^3)\),
\(\bar\bcal^a_\epsi\weakly
\bar\bcal^a\) weakly in \(L^p((0,L);\RR^{3\times2}
)\),
and \(\bar\bcal^b_\epsi\weakly
\bar\bcal^b\) weakly in \(L^p(\omega^b;\RR^{3}
)\), then \vspace{-.5mm}
\begin{equation*}
\begin{aligned}
\lim_{\epsi\to0^+} \big(L^a_\epsi(\psi^a_\epsi)
+ 
\tfrac{h_\epsi}{r_\epsi^2}L^b_\epsi(\psi^b_\epsi)
\big) =&  \int_0^L
\big(
\bar f^a\cdot\psi^a+ \bar g^a\cdot\psi^a
+ 
\Bcal^a:(\bar\bcal^a|0)\big)\,\dx_3 \\&+ \int_{\omega^b}
\big(\bar f^b\cdot\psi^b
+
(g^{b,+}-g^{b,-})\cdot\psi^{b}
+ G^b\cdot\bar\bcal^b\big)
\,\dx_\alpha ;
\end{aligned}
\end{equation*}
in particular, for  \(\psi^b\equiv (x_\alpha,0)\)
and  \(\bar \bcal^b \equiv (0_\alpha,1)\),
and  \(\bar f^a\), \(\bar g^a\), and
\(\bar f^b\)  given by \eqref{barforces}, we have
\vspace{-.5mm}%
\begin{equation}\label{eq:limtotalforcei2}
\begin{aligned}
\lim_{\epsi\to0^+} \big(L^a_\epsi(\psi^a_\epsi)
+ 
\tfrac{h_\epsi}{r_\epsi^2}L^b_\epsi(\psi^b_\epsi)
\big) =&
 \int_0^L
\big(
\bar f^a\cdot\psi^a+ \bar g^a\cdot\psi^a
+ 
\Bcal^a:(\bar\bcal^a|0)\big)\,\dx_3 \\&+ \int_{\omega^b} \big((\bar f^b_\alpha
+ 
g^{b,+}_\alpha -g^{b,-}_\alpha)\cdot
x_\alpha  + G^b_3\big)
\,\dx_\alpha.
\end{aligned}
\end{equation}

As in the proof of Theorem~\ref{thm:ellr+},
we set %
\begin{equation*}
\begin{aligned}
\mathcal{X}:=\big(
 L^p((0,L);\R^{3 \times 2}) \times W^{1,p}(\Omega^a;\R^3)\big)\times
\big( W^{1,p}(\Omega^b;\R^3) \times
L^p(\omega^b;\R^{3})
\big) 
\end{aligned}
\end{equation*}
and, recalling \eqref{Aepsi}, \eqref{Phiepsi}, \eqref{Eelli}, and \eqref{Phipli},  we introduce, for \(0<\epsi \leq 1\),
the functionals
 \(\mathcal{E}_\epsi:
\mathcal{X} \to
(-\infty,\infty]
\) and 
 \(\mathcal{E}_{\ell_\infty }:
\mathcal{X} \to
(-\infty,\infty]
\) defined by
\begin{equation*}
\begin{aligned}
\mathcal{E}_\epsi((\bar\bcal^a,\psi^a),
(\psi^b,
\bar\bcal^b)):=
\begin{cases}
\displaystyle
E_\epsi^a(\psi^a) +
E_\epsi^b(\psi^b) & \hbox{if }  ((\bar\bcal^a,\psi^a),(
\psi^b, \bar\bcal^b)
)\in \mathcal{A}_\epsi \hbox{ and }
(\psi^a,\psi^b) \in
\Phi_\epsi\\
\infty & \hbox{otherwise}
\end{cases}
\end{aligned}
\end{equation*}
and
\vspace{-1mm}%
\begin{equation*}
\begin{aligned}
\mathcal{E}_{\ell_\infty }((\bar\bcal^a,\psi^a),
(\psi^b, \bar\bcal^b)):=
\begin{cases}
\displaystyle
E_{\ell_\infty }(\bar\bcal^a,\psi^a) & \hbox{if }   \psi^a
\in
\Phi_{\ell_\infty }^p, \,   \psi^b\equiv (x_\alpha,0), \hbox{ and }
\bar\bcal^b \equiv (0_\alpha,1) \\
\infty & \hbox{otherwise,}
\end{cases}
\end{aligned}
\end{equation*}
respectively.

We claim that \((\mathcal{E}_\epsi)_{\epsi>0}\)
\(\Gamma\)-converges
to \(\mathcal{E}_{\ell_\infty }\) with
respect to the
weak
topology in \(\mathcal{X} \).
As we showed at
the beginning of this proof, \((\mathcal{E}_\epsi)_{\epsi>0}\)
is equi-coercive with respect to the
weak topology in \(\mathcal{X} \).
Thus, 
if the claim
holds, then  Theorem~\ref{thm:elli}
 follows.
Moreover, 
to prove the claim, it suffices to show
that given any
subsequence \(\epsi_n\prec\epsi\), the
\(\Gamma\)-limit  of \((\mathcal{E}_{\epsi_n})_{n\in\NN}\)
coincides with \(\mathcal{E}_{\ell_\infty }\).

We first  show that given  \( ((\bar\bcal^a_n,\psi^a_n),(\psi^b_n,
\bar\bcal^b_n))_{n\in\NN}\subset \calX\) and \( ((\bar\bcal^a,\psi^a),(\psi^b,
\bar\bcal^b)) \in \calX\) such that \( ((\bar\bcal^a_n,\psi^a_n),\allowbreak(\psi^b_n,
\bar\bcal^b_n))\weakly ((\bar\bcal^a,\psi^a),\allowbreak
(\psi^b, \bar\bcal^b)) \) weakly in
\( \calX\), we have \vspace{-.5mm}
\begin{equation}
\label{eq:liminfcalFi}
\begin{aligned}
\mathcal{E}_{\ell_\infty }((\bar\bcal^a,\psi^a),
(\psi^b, \bar\bcal^b))
\leq \liminf_{n\to\infty} \mathcal{E}_{\epsi_n}
((\bar\bcal^a_n,\psi^a_n),(\psi^b_n,
\bar\bcal^b_n)).
\end{aligned}
\end{equation}

 To prove \eqref{eq:liminfcalFi}, we may
assume  that  the  lower
limit on the right-hand
side of \eqref{eq:liminfcalF} is actually
a limit and is finite, extracting a
subsequence if necessary.  Then,  \(((\bar\bcal^a_n,\psi^a_n),(\psi^b_n,
\bar\bcal^b_n)) \in \A_{\epsi_n}\),  \((\psi^a_n,\psi^b_n)
\in  \Phi_{\epsi_n} \),  and \(\mathcal{E}_{\epsi_n}
((\bar\bcal^a_n,\psi^a_n),(\psi^b_n,
\bar\bcal^b_n)) = F^a_{\epsi_n}(\psi^a_n)
+\frac{h_{\epsi_n}}{r_{\epsi_n}^2} F^b_{\epsi_n}(\psi^b_n) - L^a_{\epsi_n}(\psi^a_n) -  \frac{h_{\epsi_n}}{r_{\epsi_n}^2}L^b_{\epsi_n}(\psi^b_n) \) for all \(n\in\NN\).
Consequently, arguing as above, we conclude
that \(\psi^a \in \Phi_{\ell_\infty }^p\), \(\psi^b\equiv (x_\alpha,0)\),  and \(\bar\bcal^b \equiv (0_\alpha,1)\); thus,
\vspace{-1mm}%
\begin{equation*}
\begin{aligned}
\mathcal{E}_{\ell_\infty }((\bar\bcal^a,\psi^a),
(\psi^b, \bar\bcal^b))= 
&\, \bar
a\int_0^L \ce({\bar a^{-1}}\bar\bcal^a|\nabla_3\psi^a)\,\dx_3
- \int_0^L
\big(
\bar f^a\cdot\psi^a+ \bar g^a\cdot\psi^a
+ 
\Bcal^a:(\bar\bcal^a|0)\big)\,\dx_3 \\
&\quad - \int_{\omega^b} \big((\bar f^b_\alpha
+ 
g^{b,+}_\alpha -g^{b,-}_\alpha)\cdot
x_\alpha  + G^b_3\big)
\,\dx_\alpha. 
\end{aligned}
\end{equation*}
As proved in \eqref{eq:lba}, we have
\vspace{-1mm}%
\begin{equation*}
\begin{aligned}
  \bar
a\int_0^L \ce({\bar a^{-1}}\bar\bcal^a|
\nabla_3\psi^a)\,\dx_3 \leq \liminf_{n\to\infty} \int_{\Omega^a}W (r_{\varepsilon_n}^{-1}\nabla_\alpha
\psi^a_n| \nabla_3 \psi^a_n)\,\dx =
 \liminf_{n\to\infty}
 F^a_{\epsi_n} (\psi^a_n).
\end{aligned}
\end{equation*}
This inequality, the fact that \(W\geq0\)
by \eqref{coercrigid2}, and  \eqref{eq:limtotalforcei2} yield
  \eqref{eq:liminfcalFi}.

To conclude, we prove that given   \( ((\bar\bcal^a,\psi^a),(\psi^b,
\bar\bcal^b)) \in \calX\), there exists
a sequence \((({\bar\bcal}^a_n,
\psi^a_n),\allowbreak
(\psi^b_n,
{\bar\bcal}^b_n))_{n\in\NN}\allowbreak \subset
\calX\)   such that \( (({\bar\bcal}^a_n,\psi^a_n),(\psi^b_n,
{\bar\bcal}^b_n))\weakly ((\bar\bcal^a,\psi^a),(\psi^b,
\bar\bcal^b)) \) weakly in \( \calX\)
and
\vspace{-1mm}%
\begin{equation}
\label{eq:limsupcalFi}
\begin{aligned}
\mathcal{E}_{\ell_\infty }((\bar\bcal^a,\psi^a),
(\psi^b, \bar\bcal^b))
=\lim_{n\to\infty} \mathcal{E}_{\epsi_n}
((\bar\bcal^a_n,\psi^a_n),(\psi^b_n,
\bar\bcal^b_n)).
\end{aligned}
\end{equation}
To establish \eqref{eq:limsupcalFi},
the only non-trivial
case is the case in which  \(\psi^a \in \Phi_{\ell_\infty }^p\),
\(\psi^b\equiv (x_\alpha,0)\),  and
\(\bar\bcal^b \equiv (0_\alpha,1)\). 
Assume that these three conditions hold
and, from now on,  also assume that \(p>2\), which
implies that \(\psi^a(0_3) = \psi^b(0_\alpha)= 0\).

Let \(\phi\in C^\infty_c(\RR;[0,1])\)
be a smooth cut-off function such that
\(\phi(t)=1\) if \(|t|\leq \tfrac{L}5\),
and 
\(\phi(t)=0\) if \(|t|\geq \tfrac{L}4\).
Define, for \(0<\epsi\leq 1\) and \(x=(x_\alpha,
x_3) \in \Omega^a\),
\begin{equation*}
\begin{aligned}
&\tilde \ffi^a_{\epsi,0}(x):= (r_\epsi x_\alpha,
x_3) \phi(x_3) + \ffi^a_{\epsi,0}(x)
(1- \phi(x_3)), \enspace
\tilde \ffi^a_{0}(x_3):= (0_\alpha,
x_3) \phi(x_3) + \ffi^a_{0}(x_3)
(1- \phi(x_3)).
\end{aligned}
\end{equation*}
Because 
\(\ffi^b_{0,\epsi}(x)
= (x_\alpha,
h_{\epsi} x_3)\),   \((\ffi^a_{\epsi,0},\ffi^b_{\epsi,0})\) satisfies
\eqref{bca}--\eqref{bcb}, and  \(\phi(0_3) = 1\), we deduce
that   \((\tilde \ffi^a_{\epsi,0},\ffi^b_{\epsi,0})\) satisfies
\eqref{bca}--\eqref{bcb},  \eqref{junction},
and \(\tilde \ffi^a_{\epsi,0} \weakly
\tilde \ffi^a_{0}
\) weakly in \(W^{1,p}(\Omega^a;\RR^3)\).
Also, note  that
 \( \psi^a(0_3) = 0 =\tilde \ffi^a_{0}(0_3)
\) and,  because \(\psi^a\in \Phi^p_{\ell_\infty}\) and \(\phi(L)=0\),
\( \psi^a(L) = \ffi^a_{0}(L) =\tilde \ffi^a_{0}(L)
\); in particular, \(\psi^a \in \tilde\Phi^a\), where \(\tilde\Phi^a\) is given by \eqref{Phirod} with \(\ffi^a_0\) replaced by \(\tilde\ffi^a_0\). 

Invoking Theorem~\ref{thm:rod1} and
\eqref{pgrowth}, we can find a sequence
\(({\bar\bcal}^a_n,
\psi^a_n)_{n\in\NN} \subset
L^p((0,L);\R^{3 \times 2}) \times W^{1,p}(\Omega^a;\R^3)\)   such that \( ({\bar\bcal}^a_n,\psi^a_n)\weakly (\bar\bcal^a,\psi^a) \) weakly in
\(L^p((0,L);\R^{3 \times 2}) \times W^{1,p}(\Omega^a;\R^3)\), \(\psi^a_n
= \tilde\ffi^a_{0,\epsi_n}\) on  \(\omega^a\times
\{0,L\}\), \(r^{-1}_{\epsi_n} \int_{\omega^a}
\grad_\alpha \psi^a_n \,\dx_\alpha =
\bar \bcal^a_n\), and \vspace{-1mm}
\begin{equation}\label{ubcasei}
\begin{aligned}
\lim_{n\to\infty} E^a_{\epsi_n}(\psi^a_n) =  \bar
a\int_0^L \ce({\bar a^{-1}}\bar\bcal^a|\nabla_3\psi^a)\,\dx_3
- \int_0^L
\big(
\bar f^a\cdot\psi^a+ \bar g^a\cdot\psi^a
+ 
\Bcal^a:(\bar\bcal^a|0)\big)\,\dx_3.
\end{aligned}
\end{equation}
Note that the condition  \(\psi^a_n
= \tilde\ffi^a_{0,\epsi_n}\) on  \(\omega^a\times
\{0,L\}\) implies that \(\psi^a_n(x_\alpha
, 0_3) =(r_{\epsi_n} x_\alpha,0_3) \)
for \aev\ \(x_\alpha\in\omega^a\) and \( \psi^a_n =
\ffi^a_{\epsi_n,0}\) on \(\Gamma^a= \omega^a
\times \{L\} \).  Finally,  define \(\psi^b_n(x) = (x_\alpha,
h_{\epsi_n} x_3) \) and \(\bar \bcal^b_n
=(0_\alpha,1)\). Recalling that 
\(\ffi^b_{0,\epsi_n}(x)
= (x_\alpha,
h_{\epsi_n} x_3)\), we conclude that
\(((\bar\bcal^a_n,\psi^a_n),(
\psi^b_n, \bar\bcal^b_n)
)\in \mathcal{A}_{\epsi_n}\),   
\((\psi^a_n,\psi^b_n) \in
\Phi_{\epsi_n}\), and, by \eqref{WI=0}
and \eqref{eq:limtotalforcei2},
\vspace{-2mm}%
\begin{equation}\label{ubcasei1}
\begin{aligned}
\lim_{n\to\infty} E^b_{\epsi_n}(\psi^b_n)
=- \lim_{n\to\infty} \tfrac{h_{\epsi_n}}{r_{\epsi_n}^2}L^b_{\epsi_n}(\psi^b_n)=
- \int_{\omega^b}
\big((\bar f^b_\alpha
+ 
g^{b,+}_\alpha -g^{b,-}_\alpha)\cdot
x_\alpha  + G^b_3\big)
\,\dx_\alpha.
\end{aligned}
\end{equation}
From \eqref{ubcasei} and \eqref{ubcasei1},
we obtain
 \eqref{eq:limsupcalFi}.
\end{proof} 

\begin{remark}
        \label{doublewells}
Instead of $\II$ and $A$,  it is possible
to consider
any other two  strongly incompatible matrices,  $A_1$ and $A_2$,  in \eqref{coercrigid2}.
In this case, 
a result similar to Theorem~\ref{thm:elli}
 holds subjected to prescribing an appropriate deformation condition (related to either $A_1$ or $A_2$) on \(\Gamma^b\) and adjusting
\eqref{WI=0} accordingly.
\end{remark}

\section{{Case $\ell=0$}}\label{Sect:lz}

In this section, we prove Theorem~\ref{thm:ellz} concerning the  \(\ell=0\) case. This is done in Section~\ref{Subs:lzproof} after we have established  some preliminary results in Section~\ref{Subs:lzau}.

Note that by  \eqref{Laepsi},
 \eqref{Lbepsi},  
\eqref{forcesl0}, \eqref{Lbcala}, 
and \eqref{Ljunction}, we have
\vspace{-1mm}%
\begin{equation*}
\begin{aligned}
\frac{r_\epsi^2}{h_\epsi} E^a_\epsi(\psi^a) = \frac{r_\epsi^2}{h_\epsi}F^a_\epsi(\psi^a)
- \frac{r_\epsi^2}{h_\epsi}L^a_\epsi(\psi^a) \enspace \text{and
} \enspace \frac{r_\epsi^2}{h_\epsi}E^b_\varepsilon(\psi^b) =
 F^b_\varepsilon(\psi^b) - L^b_\varepsilon(\psi^b), 
\end{aligned}
\end{equation*}
where \(F^a_\epsi\) and \(F^b_\epsi\)
are given by \eqref{Fabepsi} and,  for \(((\bar\bcal^a,\psi^a),
(\psi^b, \bar\bcal^b)) \in \A_\epsi\) (see \eqref{Aepsi}), \vspace{-1mm}%
\begin{equation*}
\begin{aligned}
\frac{r_\epsi^2}{h_\epsi}L^a_\varepsilon(\psi^a)&=\int_{\Omega^a}
f^a\cdot\psi^a\,\dx +
\int_{S^a} g^a\cdot\psi^a\,\d{\HH}^2(x)
+
\int_{0}^L
\Bcal^a:(\bar\bcal^a|0)\,\dx_3,\\
L^b_\varepsilon(\psi^b)&=\int_{\Omega^b}
f^b\cdot\psi^b\,\dx + 
\int_{\omega^b\backslash r_\varepsilon
\overline\omega^a}
(g^{b,+}\cdot \psi^{b,+} -g^{b,-}
\cdot\psi^{b,-})\,\dx_\alpha
+
\int_{\omega^b\backslash
r_\varepsilon\overline\omega^a}
G^b\cdot
\bar\bcal^b\,
\dx_\alpha
\\&\quad-
\int_{r_\varepsilon \omega^a}\hat
g^{b,-}\cdot\psi^{b,-}\,\dx_\alpha
-
\frac{r_\epsi^2}{h_\epsi}
\int_{ \omega^a}
\hat G^b(r_\epsi \cdot)\cdot \psi^a(\cdot,0)\,\dx_\alpha
+\int_{r_\epsi \omega^a}\hat G^b\cdot
\bar\bcal^b\,
\dx_\alpha. 
\end{aligned}
\end{equation*}

\subsection{Auxiliary results}\label{Subs:lzau}

We start by proving a convenient version of Lemma~\ref{lem:WbyQW} that allows us to assume that \(W\) is quasiconvex; thus, in particular, continuous in view of \eqref{pgrowth}.

\begin{lemma}\label{lem:WbyQW2}
The statement of Lemma~\ref{lem:WbyQW}
remains valid if we replace the set \(\mathcal{A}_\epsi\)  in
\eqref{Aepsi} by the set
\vspace{-1mm}%
\begin{equation}
\label{tAepsi}
\begin{aligned}
\widetilde{\mathcal{A}}_\epsi:= \big\{ ((\bar\bcal^a,\psi^a),
 ( \psi^b,\bar
\bcal^b)) &\in \mathcal{A}_\epsi\!:\,
\psi^a= \ffi^a_{\epsi,0} \text{ on }
\Gamma^a
 \big\}.
\end{aligned}
\end{equation}
\end{lemma} 

\begin{proof}
The proof of Lemma~\ref{lem:WbyQW2}
is mainly that of Lemma~\ref{lem:WbyQW}.
We only need to adapt Step~3 of the
latter to incorporate the boundary
condition \(\psi^a= \ffi^a_{\epsi,0} \) on \(\Gamma^a\) in the definition
of \(\widetilde \A_\epsi\), which we
detail next.

Let \(\mathcal{\widetilde G}\) and \(\widetilde F^-\) be
given by \eqref{GQW} and \eqref{F-},
respectively, with \(\A_\epsi\) replaced
by \(\widetilde \A_\epsi\).
As in the beginning of Step~3 of  the
proof
of Lemma~\ref{lem:WbyQW}, to prove that  \(\mathcal{\widetilde
G} \geq \widetilde F^-\),
 take   \(((\bar\bcal^a,\psi^a),
(\psi^b, \bar\bcal^b)) \in
\big(L^p((0,L);\R^{3 \times 2}) \times
W^{1,p}(\Omega^a;\R^3)\big)\times
\big( W^{1,p}(\Omega^b;\R^3)
\times L^p(\omega^b;\R^{3}) \big) \)
 satisfying \(\widetilde \calG((\bar\bcal^a,\psi^a),
(\psi^b, \bar\bcal^b)) <\infty\). Fix
\(\delta>0\), and
for
each \(n\in\Nb\), let \(((\bar\bcal^a_n,\psi^a_n),(
\psi^b_n, \bar\bcal^b_n)
)\in \widetilde {\mathcal{A}}_{\epsi_n}\) be such
that  \( \psi^a_n
\weakly \psi^a\)  weakly in
\(W^{1,p}(\Omega^a;\R^3)\), \( \bar\bcal^a_n
\weakly
\bar\bcal^a\)
weakly in \(L^{p}((0,L);\R^{3\times
 2})\), \( \psi^b_n \weakly \psi^b\)
weakly in \(W^{1,p}(\Omega^b;\R^3)\),
 \(  \bar\bcal^b_n \weakly \bar\bcal^b\)
weakly in \(L^{p}(\omega^b;\R^3)\),
 and
\vspace{-1mm} %
\begin{equation*}
\begin{aligned}
\widetilde \calG((\bar\bcal^a,\psi^a),
(\psi^b, \bar\bcal^b)) + \delta \geq
\liminf _{n\to\infty}
\big(\ell^a_n\calG_{\epsi_n}^a (\psi^a_n)
+
\ell^b_n\calG_{\epsi_n}^b (\psi^b_n)
\big).
 \end{aligned}
\end{equation*}
Let 
\(n_k \prec n\) be a subsequence for
which
\vspace{-1mm}\begin{equation*}
\begin{aligned}
\liminf _{n\to\infty}
\big(\ell^a_n\calG_{\epsi_n}^a (\psi^a_n)
+
\ell^b_n\calG_{\epsi_n}^b (\psi^b_n)
\big) =
\lim _{k\to\infty}
\Big(\ell^a_{n_k}\calG_{\epsi_{n_k}}^a
(\psi^a_{n_k})
+\ell^b_{n_k}
\calG_{\epsi_{n_k}}^b (\psi^b_{n_k})\Big)
.
 \end{aligned}
\end{equation*}
Fix  \(k\in\Nb\). By Step~2 of the proof
of Lemma~\ref{lem:WbyQW}, 
there exists
a sequence  \(((\bar\bcal^a_{n_k,j},\psi^a_{n_k,j}),(
\psi^b_{n_k,j}, \bar\bcal^b_{n_k,j})
)_{j\in\Nb}\subset \mathcal{A}_{\epsi_{n_k}}\)
 such
that  \( \psi^a_{n_k,j}
\weakly_j \psi^a_{n_k}\)  weakly in
\(W^{1,p}(\Omega^a;\R^3)\), \( \bar\bcal^a_{n_k,j}
\weakly_j
\bar\bcal^a_{n_k}\)
weakly in \(L^{p}((0,L);\R^{3\times
 2})\), \( \psi^b_{n_k,j} \weakly_j
\psi^b_{n_k}\) weakly
in \(W^{1,p}(\Omega^b;\R^3)\),
 \(  \bar\bcal^b_{n_k,j} \weakly_j \bar\bcal^b_{n_k}\)
weakly in \(L^{p}(\omega^b;\R^3)\),
 and \vspace{-1mm}
\begin{equation*}
\begin{aligned}
&\lim_{j\to\infty} \int_{\Omega^a} W({r_{\epsi_{n_k}}^{-1}}\nabla_\alpha
\psi^a_{n_k,j}|\nabla_3\psi^a_{n_k,j})\,\dx
=\calG_{\epsi_{n_k}}^a
(\psi^a_{n_k}),\\
& \lim_{j\to\infty} \int_{\Omega^b}
W(\nabla_\alpha \psi^b_{n_k,j}|{h_{\epsi_{n_k}}^{-1}}\nabla_3\psi^b_{n_k,j})\,\dx
=\calG_{\epsi_{n_k}}^b (\psi^b_{n_k}).
\end{aligned}
\end{equation*}
Using the E. De Giorgi's slicing method (for \(k\in\NN\)
fixed)  in the spirit of Lemma~\ref{lem:ontrace}
(also see , for instance, 
\cite[Lemma~2.2]{BFM03}), we can construct
a subsequence \(j_i\prec j\) and a sequence \((
\hat\psi^a_{k,i})_{i\in\Nb} \subset
W^{1,p}(\Omega^a;\R^3) \) such that
\(
\hat\psi^a_{k,i} =\psi^a_{n_k} \) on
\(\Gamma^a=\omega^a \times \{L \}\) and 
\(
\hat\psi^a_{k,i} =\psi^a_{n_k, j_i} \) on
  \(\omega^a \times \{0\}\) for all
\(i\in\NN\), 
  \( \hat\psi^a_{k,i}
\weakly_i \psi^a_{n_k}\)  weakly in
\(W^{1,p}(\Omega^a;\R^3)\), 
\( \hat{\bar\bcal}^a_{k,i} := \int_{\omega^a}
 \grad_\alpha 
\hat\psi^a_{k,i}\, \dx_\alpha
\weakly_i
\bar\bcal^a_{n_k}\)
weakly in \(L^{p}((0,L);\R^{3\times
 2})\), and \vspace{-1mm}
\begin{equation*}
\begin{aligned}
\limsup_{i\to\infty}  \int_{\Omega^a} W({r_{\epsi_{n_k}}^{-1}}\nabla_\alpha
\hat\psi^a_{k,i}|\nabla_3
\hat\psi^a_{k,i})\,\dx
\leq \lim_{j\to\infty} \int_{\Omega^a} W({r_{\epsi_{n_k}}^{-1}}\nabla_\alpha
\psi^a_{n_k,j}|\nabla_3\psi^a_{n_k,j})\,\dx.
\end{aligned}
\end{equation*}
Note that the trace equalities
\(
\hat\psi^a_{k,i} =\psi^a_{n_k} \) on
\(\Gamma^a\) and 
\(
\hat\psi^a_{k,i} =\psi^a_{n_k, j_i}
\) on
  \(\omega^a \times \{0\}\) for all
\(i\in\NN\), together with the inclusions
 \(((\bar\bcal^a_{n_k},\psi^a_{n_k}),(
\psi^b_{n_k}, \bar\bcal^b_{n_k})
)\in \widetilde {\mathcal{A}}_{\epsi_{n_k}}\)
and   \(((\bar\bcal^a_{n_k,j_i},\psi^a_{n_k,j_i}),(
\psi^b_{n_k,j_i}, \bar\bcal^b_{n_k,j_i})
)_{i\in\Nb}\subset \mathcal{A}_{\epsi_{n_k}}\),
 imply that    \(((\hat{\bar\bcal}^a_{k,i},
\hat\psi^a_{k,i}),(
\psi^b_{n_k,j_i}, \bar\bcal^b_{n_k,j_i})
)_{i\in\Nb}\subset \widetilde {\mathcal{A}}_{\epsi_{n_k}}\).
To conclude, we proceed  as in Step~3 of  the proof
of Lemma~\ref{lem:WbyQW} (from \eqref{eq:S3energies}
onwards).
\end{proof}

As in Section~\ref{Sect:l+},  to individualize
the  variables
\(\bar\bcal^a\) and \(\bar\bcal^b\)
in the elastic
part of the (scaled) total energy,  we introduce,
for \(0<\epsi\leq 1\), the functional
\(\widetilde F_\epsi:
\big(
 L^p((0,L);\R^{3 \times 2}) \times W^{1,p}(\Omega^a;\R^3)\big)\times
\big( W^{1,p}(\Omega^b;\R^3)
\times L^p(\omega^b;\R^{3}) \big) \to
(-\infty,\infty]
\) defined by
\vspace{-1mm}%
\begin{equation}
\label{tFepsi}
\begin{aligned}
\widetilde F_\epsi((\bar\bcal^a,\psi^a), (\psi^b,
\bar\bcal^b)):=
\begin{cases}
\displaystyle
\tfrac{r_\varepsilon^2}{h_\varepsilon}
F_\epsi^a(\psi^a) +
F_\epsi^b(\psi^b)
& \hbox{if }  ((\bar\bcal^a,\psi^a),(
\psi^b, \bar\bcal^b)
)\in \widetilde{\mathcal{A}}_\epsi\\
\infty & \hbox{otherwise,}
\end{cases}
\end{aligned}
\end{equation}
where \(F^a_\epsi\) and \(F^b_\epsi\)
are the functionals defined in \eqref{Fabepsi}
and \(\widetilde\A_\epsi\) is given by \eqref{tAepsi}. 
Next, we prove that a  \(\Gamma\)-convergence result similar
to Theorem~\ref{Thm:FepsitoF} holds.

\begin{theorem}\label{thm:tFepsitotF}
Let \(W:\R^{3\times 3} \to \R\) be a
Borel function satisfying
\eqref{pgrowth}, \eqref{coercrigid2},
and \eqref{WI=0}, and assume
that \(p\leq2\),  
 \(\lim_{\varepsilon \to 0}{h_\varepsilon}/{r_\varepsilon^{p+2}}
= 0\), and \(\ffi_{0,\epsi}^a \equiv
(r_\epsi x_\alpha,x_3)\).
Then, the sequence of functionals \((\widetilde
F_\epsi)_{\epsi>0}\)
defined by \eqref{tFepsi}
\(\Gamma\)-converges,
with respect to the weak topology in
\(
\big(L^p((0,L);\R^{3 \times 2}) \times
W^{1,p}(\Omega^a;\R^3)\big)\times
\big( W^{1,p}(\Omega^b;\R^3)
\times L^p(\omega^b;\R^{3}) \big) \),
to the functional
\(\widetilde F:
\big(L^p((0,L);\R^{3 \times 2}) \times
W^{1,p}(\Omega^a;\R^3)\big)\times
\big( W^{1,p}(\Omega^b;\R^3)
\times L^p(\omega^b;\R^{3}) \big) \to
(-\infty,\infty]
\) defined by
\begin{equation*}
\begin{aligned}
\widetilde F((\bar\bcal^a,\psi^a), (\psi^b, \bar\bcal^b)):=
\begin{cases}
\displaystyle
 F^b(\psi^b,\bar\bcal^b) & \hbox{if }
 \psi^a \equiv (0_\alpha,x_3),\, \bar\bcal^a
 \equiv \bar a\, \II_\alpha,\, \text{and
} \,
\psi^b \text{ is independent of } x_3
\\
\infty & \hbox{otherwise,}
\end{cases}
\end{aligned}
\end{equation*}
where \vspace{-1mm}%
\begin{equation*}
\begin{aligned}
F^b(\psi^b,\bar\bcal^b)= \int_{\omega^b}
\qce(\nabla_\alpha \psi^b|\bar\bcal^b)\,\dx_\alpha.
\end{aligned}
\end{equation*}
\end{theorem}

\begin{proof}
The proof of Theorem~\ref{thm:tFepsitotF}
follows along that of Theorem~\ref{Thm:FepsitoF},
but several adaptations are required.

We want to show that   given any subsequence \(\epsi_n\prec\epsi\),
the \(\Gamma\)-limit inferior,
\(\widetilde F^-\), of \((\widetilde F_{\epsi_n})_{n\in\Nb}\),
given by
\eqref{F-}
with \(\A_{\epsi_n}\) replaced
by \(\widetilde \A_{\epsi_n}\),
 coincides with \(\widetilde F\)  for all \(\big((\bar\bcal^a,\psi^a),
(\psi^b, \bar\bcal^b)\big) \in
\big(L^p((0,L);\R^{3 \times 2}) \times
W^{1,p}(\Omega^a;\R^3)\big)\times
\big( W^{1,p}(\Omega^b;\R^3)
\times L^p(\omega^b;\R^{3}) \big)  \).
For that, we will proceed
in several steps and, to simplify the notation, we
set 
\(\mu_n:={r_{\varepsilon_n}^{-1}}\),
\(\lambda_n:={h_{\varepsilon_n}^{-1}}\),
and 
\(\ell_n:={h_{\varepsilon_n}}/{r_{\varepsilon_n}^2}\).

\textit{Step 1.} In this step, we prove
that we may assume
 that \(W
\) is a continuous, quasiconvex
function.

By \eqref{eq:cxtyineq},
we have  \(\mathcal{C}(\qe) = \ce\)
and \(\mathcal{QC}(\qe) = \qce\).
On the other hand, by Lemma~\ref{lem:WbyQW2},
the \(\Gamma\)-limit inferior in
\eqref{F-} with  \(\A_{\epsi_n}\) replaced
by \(\widetilde \A_{\epsi_n}\), \(\ell^a_n\equiv
{r_{\epsi_n}^2}/{h_{\epsi_n}}\), and \(\ell^b_n\equiv
1\)
remains
unchanged if we replace \(W\) by its
quasiconvex envelope,
\(\qe\). Thus, without loss of generality, we may assume
 that \(W
\) is  quasiconvex, which together with \eqref{pgrowth} implies
that \(W\) is \(p\)-Lipschitz continuous
(see \eqref{eq:WpLip}). 

\textit{Step 2.} In this step, we prove
that if \(\widetilde F^-((\bar\bcal^a,\psi^a),
(\psi^b, \bar\bcal^b)) < \infty\), then
\(\psi^a \equiv (0_\alpha,x_3)\), \( \bar\bcal^a
 \equiv \bar a\, \II_\alpha\), and
\(\psi^b\)  is independent of  \(x_3\).

By definition of the \(\Gamma\)-limit
inferior, for all
\(n\in\Nb\), there exists  \(((\bar\bcal^a_n,\psi^a_n),(
\psi^b_n, \bar\bcal^b_n)
)\in \widetilde{\mathcal{A}}_{\epsi_n} \) such that
 \(((\bar\bcal^a_n,\psi^a_n),(
\psi^b_n, \bar\bcal^b_n)
)_{n\in\Nb}\) weakly converges to \(((\bar\bcal^a,\psi^a),
(\psi^b, \bar\bcal^b) )\) in  \(
\big(L^p((0,L);\R^{3 \times 2}) \times
W^{1,p}(\Omega^a;\R^3)\big)\times
\big( W^{1,p}(\Omega^b;\R^3)
\times L^p(\omega^b;\R^{3}) \big) \)
and \(\widetilde  F^-((\bar\bcal^a,\psi^a),
(\psi^b, \bar\bcal^b)) = \liminf_{n\to\infty}\big(\ell_n^{-1}F_{\epsi_n}^a
(\psi^a_n) + F_{\epsi_n}^b
(\psi^b_n) \big) \).
Because \(\widetilde F^-((\bar\bcal^a,\psi^a),
(\psi^b, \bar\bcal^b)) < \infty\), extracting
a subsequence
if needed, we may assume that there
exists a positive
constant, \(\bar C\), such that for
all \(n\in\Nb\),
we have \(\ell_n^{-1}F_{\epsi_n}^a
(\psi^a_n) + F_{\epsi_n}^b
(\psi^b_n) \leq \bar C\). 

Define
\vspace{-1mm}%
\begin{equation*}
\begin{aligned}
\tilde\psi^a_n := \begin{cases}
\psi^a_{\epsi_n} & \hbox{in } \Omega^a\\
\ffi^a_{0,\epsi_n} & \hbox{in } \omega^a \times [L,2L)
\end{cases}
\end{aligned}
\end{equation*}
and \(\tilde \Omega^a := \omega^a \times (0,2L)\). By definition of \(\widetilde {\mathcal{A}}_{\epsi_n}\), we have \(\tilde \psi^a_n \in W^{1,p} (\tilde\Omega^a;\RR^3)\) and, in view of  \eqref{WI=0}, we have  \(\int_{\tilde\Omega^a} W(r_{\epsi_n}^{-1} \grad_\alpha \tilde \psi^a_n| \grad_3 \tilde \psi^a_n)\, \dx \equiv F_{\epsi_n}^a (\psi^a_n) \); thus,
invoking \eqref{coercrigid2} and setting
\(\mathcal{K}:= SO(3) \cup SO(3)A\), we have
\vspace{-.5mm}%
\begin{equation*}
\begin{aligned}
\sup_{n\in\NN} \bigg(\frac{r_{\varepsilon_n}^2}{h_{\varepsilon_n}} \Vert \dist ( (r_{\epsi_n}^{-1}\grad_\alpha
\tilde \psi^a_{n}
|\grad_3 \tilde \psi^a_{n}) , \mathcal{K})\Vert^p_{L^p(\tilde \Omega^a)} +
 \Vert \dist ( (\grad_\alpha
\psi^b_n
|h_{\epsi_n}^{-1}\grad_3 \psi^b_n) , \mathcal{K})\Vert^p_{L^p(\Omega^b)}\Big) \bigg) \leq \bar C.
\end{aligned}
\end{equation*}
In particular,
using the fact that \(\mathcal{K}\)
is a compact subset of \(\RR^{3\times
3}\) and using  Proposition~\ref{prop:thm6FJM},   we 
have also
\vspace{-1mm}%
\begin{equation*}
\begin{aligned}
\sup_{n\in\NN} \|h_{\epsi_n}^{-1}\grad_3 \psi^b_n
\|_{L^p(\Omega^b;\RR^3)} 
<\infty \enspace \hbox{ and } \enspace \Vert (r_{\epsi_n}^{-1}\grad_\alpha
\tilde \psi^a_{n}
|\grad_3 \tilde \psi^a_{n}) - M^a_n  \Vert^p_{L^p(\tilde \Omega^a;\RR^{3\times
3})} \leq \tilde C \frac{h_{\epsi_n}}{r_{\epsi_n}^{p+2}},
\end{aligned}
\end{equation*}
where  \((M^a_n)_{n\in\NN}\subset
\mathcal{K}\) is a sequence of constant matrices.
Consequently, \(\psi^b\) is independent of \(x_3\). 
Moreover, extracting a subsequence if necessary,
we have  \(M^a_n \to M^a\) in
\(\RR^{3\times 3}\) for some \(M^a\in
\mathcal{K}\). Then, because \(\lim{h_{\epsi_n}}/{r_{\epsi_n}^{p+2}}= 0\) by hypothesis, it follows that \vspace{-1mm}
\begin{equation*}
\begin{aligned}
  (r_{\epsi_n}^{-1}\grad_\alpha
\tilde \psi^a_{n}
|\grad_3 \tilde \psi^a_{n}) \to
M^a \text{ in } {L^p(\tilde \Omega^a;\RR^{3\times
3})}.
\end{aligned}
\end{equation*}
Recalling that \(\ffi^a_{0,\epsi_n} \equiv (r_{\epsi_n} x_\alpha, x_3)\) and \(\tilde \psi^a_n \equiv \ffi^a_{0,\epsi_n} \) in \( \omega^a \times [L,2L)\), we obtain \(M^a\equiv \II\). Hence, 
\(\psi^a \equiv (0_\alpha,x_3)\) and \( \bar\bcal^a
 \equiv \bar a\, \II_\alpha\).

\textit{Step 3 (lower bound).} In this
step, we prove
that for all \(((\bar\bcal^a,\psi^a),
(\psi^b, \bar\bcal^b) )\) in  \(
\big(L^p((0,L);\R^{3 \times 2}) \times
W^{1,p}(\Omega^a;\R^3)\big)\times
\big( W^{1,p}(\Omega^b;\R^3)
\times L^p(\omega^b;\R^{3}) \big) \) such that 
\(\psi^a \equiv (0_\alpha,x_3)\), \( \bar\bcal^a
 \equiv \bar a\, \II_\alpha\), and
\(\psi^b\)  is independent of  \(x_3\),
  we have
\vspace{-2mm}%
\begin{equation*}
\begin{aligned}
\widetilde F^-((\bar\bcal^a,\psi^a),
(\psi^b, \bar\bcal^b)) \geq  F^b(\psi^b,\bar\bcal^b) .
\end{aligned}
\end{equation*}

To prove this estimate, it suffices to invoke the inequality \(W\geq 0\) and    \cite[Theorem~1.2~(i)]{BFM03}.

\textit{Step 4 (upper bound in terms
of the original
density \(W\) and for regular target
functions).} In this
step, we prove
that if 
\(\psi^a \equiv (0_\alpha,x_3)\), \( \bar\bcal^a
 \equiv \bar a\, \II_\alpha\), 
\(\psi^b \in W^{1,\infty} (\Omega^b;\RR^3)\)  is independent of  \(x_3\),
and \(\bcal^b\in C^1(\overline \omega^b;\RR^3)\), then \vspace{-1mm}
\begin{equation}
\label{eq:ubwrtWlz}
\begin{aligned}
\widetilde F^-((\bar\bcal^a,\psi^a),
(\psi^b, \bar\bcal^b)) \leq \int_{\omega^b}W(\nabla_\alpha
\psi^b| \bar\bcal^b)\,\dx_\alpha .
\end{aligned}
\end{equation}

We  recall that 
  $W$ can be assumed
to be a  continuous function by Step~1. The arguments we will use next
are inspired by
those in \cite[Propositions~3.1 and
4.1]{GGLM02}.

Let \(\gamma:= 2 \,\hbox{diam}(\omega^a)\). For all \(n\in\NN\) sufficiently large, we have
\begin{equation*}
\begin{aligned}
r_{\epsi_n} \omega^a \subset B(0_\alpha, \gamma r_{\epsi_n})\subset \subset B(0_\alpha, \gamma \sqrt{r_{\epsi_n}}) \subset \subset \omega^b.
\end{aligned}
\end{equation*}
For all such \(n\in\NN\), let \(\phi_{p,n} \in C^1_0(B(0_\alpha,
\gamma \sqrt{r_{\epsi_n}});[0,1])\) with \(\phi_{p,n}=1\) in \( \overline{B(0_\alpha,
\gamma {r_{\epsi_n}})}\)  be the solution to the \(p\)-capacity problem of \(\overline{B(0_\alpha, \gamma r_{\epsi_n})}\) with respect to \(B(0_\alpha,
\gamma \sqrt{r_{\epsi_n}})\); that is, the solution to
\begin{equation*}
\begin{aligned}
\min \bigg\{ \int_{B(0_\alpha,
\gamma \sqrt{r_{\epsi_n}})} |\grad_\alpha \phi(x_\alpha)|^p \,\dx_\alpha\!:\, \phi\in C^1_0(B(0_\alpha,
\gamma \sqrt{r_{\epsi_n}});[0,1]), \, \phi=1 \hbox{ in } \overline{B(0_\alpha,
\gamma {r_{\epsi_n}})} \bigg\}.
\end{aligned}
\end{equation*}
Then (see \cite[Example~2.12]{HJM93}), for \(A_n:= B(0_\alpha,
\gamma \sqrt{r_{\epsi_n}}) \backslash \overline{B(0_\alpha,
\gamma {r_{\epsi_n}})}\), we have
\begin{equation*}
\begin{aligned}
\int_{A_n} |\grad_\alpha \phi_{p,n}(x_\alpha)|^p \,\dx_\alpha = 
\begin{cases} \displaystyle
2\pi \Big( \frac{|2-p|}{p-1} \Big)^{p-1}  \bigg|r_{\epsi_n}^{\tfrac{p-2}{2(p-1)}} - r_{\epsi_n}^{\tfrac{p-2}{p-1}}\bigg|^{1-p} & \hbox{if } p\not=2 \\ 
2\pi \big(-\log({ \sqrt{r_{\epsi_n}}})\big)^{-1} & \hbox{if } p=2.
\end{cases}
\end{aligned}
\end{equation*}
Note that if $1< p \leq 2$, then
\begin{equation}
\label{bycapacity}
\begin{aligned}
\lim_{n\to\infty} \int_{\omega^b} |\grad_\alpha \phi_{p,n}(x_\alpha)|^p \,\dx_\alpha = \lim_{n\to\infty} \int_{A_n} |\grad_\alpha \phi_{p,n}(x_\alpha)|^p \,\dx_\alpha = 0.
\end{aligned}
\end{equation}

Next, for \(x=(x_\alpha,x_3) \in \Omega^a\), we define
\vspace{-1mm}%
\begin{equation*}
\begin{aligned}
&\psi^a_n(x) :=
(r_{\epsi_n} x_\alpha, x_3),  \quad \bar\bcal^a_n(x_3):= {r_{\epsi_n}^{-1}} \int_{\omega^a} \grad_\alpha \psi^a_n\,\dx_\alpha
\enspace 
\end{aligned}
\end{equation*}
and, for \(x=(x_\alpha,x_3) \in \Omega^b\), we define
\vspace{-1mm}%
\begin{equation*}
\begin{aligned}
&\psi^b_n(x):= (x_\alpha, h_{\epsi_n} x_3) \phi_{p,n}(x_\alpha) + \big(   h_{\varepsilon_n} x_3\bar\bcal^b(x_\alpha)
+ \psi^b(x_\alpha)\big) (1- \phi_{p,n}(x_\alpha)),\quad \bar\bcal^b_n(x_\alpha) := {h_{\epsi_n}^{-1}}
\int_{-1}^0 \grad_3 \psi^b_n\,\dx_3.
\end{aligned}
\end{equation*}
Because \(\ffi^a_{0,\epsi_n} \equiv (r_{\epsi_n} x_\alpha, x_3)\) and \(\phi_{p,n} \equiv 1\) in \(r_{\epsi_n} \omega^a \), we have \(  ((\bar\bcal^a_n,\psi^a_n),(
\psi^b_n, \bar\bcal^b_n)
)\in \widetilde{\mathcal{A}}_{\epsi_n}\). Moreover,
using  \eqref{WI=0}, \vspace{-2.5mm}
\begin{equation*}
\begin{aligned}
\psi^a_n \to \psi^a \hbox{ in } W^{1,p}(\Omega^a;\mathbb
R^3), \enspace \bar\bcal^a_n \to \bar\bcal^a \hbox{ in } L^p((0,L);\mathbb
R^{3\times 2}), \enspace
  \int_{\Omega^a} W({r_{\epsi_{n}}^{-1}}\nabla_\alpha
\psi^a_n|\nabla_3
\psi^a_n)\,\dx =0.
\end{aligned}
\end{equation*}
Assume now that \(1<p\leq2\). The convergences  \(\phi_{p,n}
\to 0\) \aev\ in \(\omega^b\) and  \eqref{bycapacity}, together with \eqref{pgrowth} and Vitali--Lebesgue's lemma, yield
\vspace{-1mm}%
\begin{equation*}
\begin{aligned}
\psi^b_n \to \psi^b \hbox{ in } W^{1,p}(\Omega^b;\mathbb
R^3), \enspace \bar\bcal^b_n \to \bar\bcal^b \hbox{ in } L^p(\omega^b;\mathbb
R^3), \enspace \lim_{n\to\infty} \int_{\Omega^b}W(\nabla_\alpha
\psi^{b}_n| \bar\bcal^b_n)\,\dx= \int_{\omega^b}W(\nabla_\alpha
\psi^b| \bar\bcal^b)\,\dx_\alpha.
\end{aligned}
\end{equation*}
Hence, \eqref{eq:ubwrtWlz} holds.

\smallskip
\textit{Step 5 (Upper bound).} In this
step, we prove
that for all \(((\bar\bcal^a,\psi^a),
(\psi^b, \bar\bcal^b) )\) in  \(
\big(L^p((0,L);\R^{3 \times 2}) \times
W^{1,p}(\Omega^a;\R^3)\big)\times
\big( W^{1,p}(\Omega^b;\R^3)
\times L^p(\omega^b;\R^{3}) \big) \) such that 
\(\psi^a \equiv (0_\alpha,x_3)\), \( \bar\bcal^a
 \equiv \bar a\, \II_\alpha\), and
\(\psi^b\)  is independent of  \(x_3\),
  we have
\vspace{-1mm}%
\begin{equation*}
\begin{aligned}
\widetilde F^-((\bar\bcal^a,\psi^a),
(\psi^b, \bar\bcal^b)) \leq  F^b(\psi^b,\bar\bcal^b) .
\end{aligned}
\end{equation*}

To prove this estimate, it suffices to argue as in Steps~5 and 6  of the proof of Theorem~\ref{Thm:FepsitoF} but invoking \eqref{eq:two} in place of
 Lemma~\ref{lem:relaxationresult}. 
\end{proof}

Arguing  as in Lemma~\ref{lem:ontrace}
(disregarding the terms related to \(\Omega^a\)), it can be easily checked that the
following result holds. This result enable us to address the boundary condition on \(\Gamma^b\).

\begin{lemma}\label{lem:ontraceb}
Let \(W:\R^{3\times 3} \to \R\) be a
Borel function satisfying
\eqref{pgrowth} and let \(\kappa\in\R\).
Fix \((\psi^b, \bar\bcal^b)  \in 
W^{1,p}(\Omega^b;\RR^3)\times L^p(\omega^b;\mathbb
R^3)$ with \(\psi^b\) independent of
\(x_3\).  For each
\(n\in\Nb\), let  \(\lambda_n \in \RR^+\)
and
\(
\psi^b_n\in W^{1,p}(\Omega^b;\R^3)
 \) be such
that \(\lambda_n\to\infty\),  \( \psi^b_n \to \psi^b\)  in
\(L^{p}(\Omega^b;\R^3)\),
 \(   \lambda_n \int_{-1}^0 \grad_3
 \psi_n^b\,\dx_3 \weakly \bar\bcal^b\)
weakly in \(L^{p}(\omega^b;\R^3)\),
 and
\vspace{-1.5mm} %
\begin{equation*}
\begin{aligned}
\lim_{n\to\infty}
  \int_{\Omega^b}W (\nabla_\alpha
\psi^b_n| \lambda_n\nabla_3 \psi^b_n)\,\dx
= \kappa.
\end{aligned}
\end{equation*}
Let 
 \((\ffi^b_n)_{n\in\Nb}
\subset W^{1,p}(\Omega^b;\R^3)\) and
 \(\ffi^b \in W^{1,p}(\Omega^b;\R^3)\)
satisfy \eqref{bcb}.
Assume further that  \(\psi^b = \ffi^b\)
on \(\Gamma^b
\). Then, there exist subsequences 
\(\lambda_{n_k}
\prec \lambda_n\) and \(\ffi^b_{{n_k}}
\prec \ffi^b_n\) and  a sequence 
\((\tilde
\psi^b_k
)_{k\in\Nb} \subset W^{1,p}(\Omega^b;\R^3)\) satisfying
\(\tilde\psi^b_k
= \ffi^b_{n_k}\)
on \(\Gamma^b\)  and \(\tilde\psi^b_k
= \psi^b_{n_k}\) in \(\omega^a \times
(-1,0)\) for all \(k\in\Nb\),
 \(\tilde \psi^b_k \weakly \psi^b\)
weakly in
 \(W^{1,p}(\Omega^b;\R^3)\),    
 \( \lambda_{n_k} \int_{-1}^0 \grad_3
 \tilde \psi_k^b\,\dx_3 
\weakly \bar\bcal^b\) weakly in \(L^{p}(\omega^b;\R^3)\),
 and \vspace{-1.5mm}
 \begin{equation*}
\begin{aligned}
\limsup_{k\to\infty}
 \int_{\Omega^b}W (\nabla_\alpha
\tilde \psi^b_k| \lambda_{n_k}\nabla_3 \tilde \psi^b_k)\,\dx
\leq \kappa.
\end{aligned}
\end{equation*}
\end{lemma}

\subsection{Proof of Theorem~\ref{thm:ellz}}\label{Subs:lzproof}
In this subsection, we prove Theorem~\ref{thm:ellz}. The proof is analogous to  those of Theorems~\ref{thm:ellr+} and \ref{thm:elli}, for which reason we only indicate here the main ideas.

\begin{proof}[Proof of Theorem~\ref{thm:ellz}]
Let \((\psi^a_\epsi,\psi^b_\epsi)_{\epsi>0}\)
be as in the statement of Theorem~\ref{thm:ellz};
that is, a sequence
in \(W^{1,p}(\Omega^a;{\mathbb
R}^3) \times W^{1,p}(\Omega^b;{\mathbb
R}^3)\) satisfying \( \psi^a_\epsi =
\ffi^a_{0,\epsi}\) on \(\Gamma^a\),
\(
 \psi^b_\epsi =   \ffi^b_{0,\epsi}\)
on  \(\Gamma^b\), \(\psi^a_\epsi(x_\alpha,0)
=
\psi^b_\epsi(r_\varepsilon x_\alpha,0)\)
for \aev\
\(x_\alpha\in\omega^a\), and
\vspace{-1.5mm}%
\begin{equation*}
\begin{aligned}
\frac{r_\epsi^2}{h_\epsi}E^a_\epsi(\psi^a_\epsi) + \frac{r_\epsi^2}{h_\epsi}E^b_\epsi(\psi^b_\epsi)< \inf_{(\psi^a,\psi^b) \in
 \Phi_\epsi}
 \Big(\frac{r_\epsi^2}{h_\epsi}E^a_\varepsilon(\psi^a) + \frac{r_\epsi^2}{h_\epsi}
 E^b_\varepsilon(\psi^b)\Big)
+ \rho(\epsi),
\end{aligned}
\end{equation*}
where \(\rho\) is a non-negative
function satisfying \(\rho(\epsi) \to
0\) as \(\epsi\to0^+\) and  \(\Phi_\epsi\) is given by \eqref{Phiepsi}. Recall that here \(\ffi^a_{0,\epsi} \equiv (r_\epsi x_\alpha, 0_3)\).

Arguing as at the beginning of the proof of Theorem~\ref{thm:elli} and as in Step~2 of Theorem~\ref{thm:tFepsitotF}, it follows that \vspace{-1mm} 
\begin{equation*}
(r_\epsi^{-1}\psi^a_\epsi, \nabla_3
\psi^a_\epsi) \to ( \bcal^a,\psi^a)
\hbox{ in } 
L^p(\Omega^a;\R^{3\times 2})\times W^{1,p}(\Omega^a;\R^3), 
\end{equation*}
 where  \(\psi^a\equiv
(0_\alpha,x_3)\)
and \(\bcal^a \equiv \II_\alpha\). Moreover,  
the sequence 
\((\psi^b_\epsi, \bar \bcal^b_\epsi)_{\epsi>0}\),
where \(\bar\bcal^b_\epsi:={h_{\varepsilon}^{-1}}
\int_{-1}^0\nabla_3
\psi^b_\epsi\, \dx_3\), is sequentially,
weakly compact in
\(   
W^{1,p}(\Omega^b;\R^3)\times 
L^p(\omega^b;\R^{3 }) \). If \((
\psi^b, \bar\bcal^b)\) is a corresponding accumulation
point, then \(\psi^b\in
\Phi_{\ell_0}^p\) (see \eqref{Phiplz}).

As in the proof of Theorems~\ref{thm:ellr+} and \ref{thm:elli}, if we prove that the sequence  \((\widetilde{\mathcal{E}}_\epsi)_{\epsi>0}\)
 of the functionals \(\widetilde{\mathcal{E}}_\epsi:
\big(
 L^p((0,L);\R^{3 \times 2}) \times W^{1,p}(\Omega^a;\R^3)\big)\times
\big( W^{1,p}(\Omega^b;\R^3) \times
L^p(\omega^b;\R^{3})
\big)  \to
(-\infty,\infty]
\) 
defined by \vspace{-1mm}
\begin{equation*}
\begin{aligned}
\widetilde{\mathcal{E}}_\epsi((\bar\bcal^a,\psi^a),
(\psi^b,
\bar\bcal^b)):=
\begin{cases}
\displaystyle
\frac{r_\epsi^2}{h_\epsi}E_\epsi^a(\psi^a) +
\frac{r_\epsi^2}{h_\epsi}E_\epsi^b(\psi^b) & \hbox{if }  ((\bar\bcal^a,\psi^a),(
\psi^b, \bar\bcal^b)
)\in \widetilde{\mathcal{A}}_\epsi \hbox{ and }
(\psi^a,\psi^b) \in
\Phi_\epsi\\
\infty & \hbox{otherwise}
\end{cases}
\end{aligned}
\end{equation*}
\(\Gamma\)-converges, with
respect to the weak
topology in
\(
\big(
 L^p((0,L);\R^{3 \times 2}) \times W^{1,p}(\Omega^a;\R^3)\big)\times
\big( W^{1,p}(\Omega^b;\R^3) \times
L^p(\omega^b;\R^{3})
\big)\), to the functional 
 \(\mathcal{E}_{\ell_0 }:
\big(
 L^p((0,L);\R^{3 \times 2}) \times W^{1,p}(\Omega^a;\R^3)\big)\times
\big( W^{1,p}(\Omega^b;\R^3) \times
L^p(\omega^b;\R^{3})
\big) \to
(-\infty,\infty]
\) defined 
by
\vspace{-1mm}%
\begin{equation*}
\begin{aligned}
\mathcal{E}_{\ell_0 }((\bar\bcal^a,\psi^a),
(\psi^b, \bar\bcal^b)):=
\begin{cases}
\displaystyle
E_{\ell_0 }
(\psi^b, \bar\bcal^b) & \hbox{if }     \psi^a\equiv (0_\alpha,x_3), 
\bar\bcal^a \equiv \bar a \II_\alpha, \hbox{ and } \psi^b\in
\Phi_{\ell_0}^p\\
\infty & \hbox{otherwise,}
\end{cases}
\end{aligned}
\end{equation*}
then Theorem~\ref{thm:ellz} follows. 

We omit  the proof of this \(\Gamma\)-convergence property because it can be proved exactly as its counterpart in Theorem~\ref{thm:ellr+} but invoking Theorem~\ref{thm:tFepsitotF} in place of Theorem~\ref{Thm:FepsitoF} and Lemma~\ref{lem:ontraceb} in place of Lemma~\ref{lem:ontrace}. 
\end{proof}

\begin{remark}
A remark similar  to Remark \ref{doublewells} holds in \(\ell=0\) case. 
\end{remark}

\section{On the system of applied forces}\label{Sect:forces}

In this section, we further extend our analysis by exploring variants of the system of applied forces. Precisely, in Section~\ref{Subs:nobending}, we consider the case in which one or both  the terms,  $G^a$ and $G^b$, inducing bending moments in the limit models  is not present. Next, in Section~\ref{Subs:divform}, we consider the case in which the applied forces are in divergence form.

\subsection{Models without (or partially without) bending moments}\label{Subs:nobending}
Here, we consider the case in which the terms $G^a$ or $G^b$ in the surface applied forces with a non-standard order of scaling magnitude in \eqref{forcesl+}--\eqref{forcesl0} is not present. Roughly speaking, in this case, the work done by the forces does not depend on $\bar \bcal^a$ or $\bar \bcal^b$, which allow us to pass the  corresponding minimum under the elastic energy integral sign; we then  recover the elastic energy densities in \cite{ABP91} and \cite{LDR00}.

Precisely, let \(W_0\) be the function defined in \eqref{eq:W0} and 
 \(W_1:\RR^{3\times 2} \to \RR\) be the function defined by 
 \(W_1(M_\alpha):=\inf_{b^b\in\R^3} 
W(M_\alpha|b^b)\) for \(M_\alpha\in \RR^{3\times 2}\).
Denote by \(\ce_{\!0}\)  and \(\qe_{\!1}\)  the convex and quasiconvex envelops of \(W_0\) and \(W_1\), respectively.
Following the same
arguments as those in  \cite[Proposition~1.1-(iv)]{BFM03}, we have that
\begin{equation*}
\begin{aligned}
\inf_{b^a \in \RR^{3\times 2}} \ce(b^a|\cdot) = \ce_{\!0}(\cdot) \enspace \hbox{ and } \enspace \inf_{b^b \in \RR^{3}} \qce(\cdot|b^b) = \qe_{\!1}(\cdot).
\end{aligned}
\end{equation*}
Then,  if the function $G^a$ in \eqref{forcesl+}
and \eqref{forcesli} is null, which means that \(\Bcal^a\equiv0\), we  may perform explicitly
the infimum in \eqref{Pell+} and \eqref{Pelli} (see Theorems~\ref{thm:ellr+} and \ref{thm:elli}) with
respect to $\bar\bcal^a$  and pass it under the elastic energy integral sign.
In this case, the elastic energy term in \((0,L)\) becomes
\vspace{-.5mm}%
\begin{equation*}
\begin{aligned}
\bar
a\int_0^L \mathcal{C}W_0(\nabla_3\psi^a)\,\dx_3.
\end{aligned}
\end{equation*}
Similarly, if the function  $G^b$  in \eqref{forcesl+} and \eqref{forcesl0} is null, we  may perform explicitly
the infimum in \eqref{Pell+} and \eqref{Pellz} (see Theorems~\ref{thm:ellr+} and \ref{thm:ellz})
 with
respect to $\bar
\bcal^b$ and pass it under the elastic energy integral sign. In this case, the elastic energy term in \(\omega^b\) becomes
\vspace{-.5mm}%
\begin{equation*}
\begin{aligned}
 \int_{\omega^b}
{\mathcal Q}W_1(\nabla_\alpha \psi^b)\,\dx_\alpha,
\end{aligned}
\end{equation*}
multiplied by \(\ell\) in the \(\ell^+\) case.

For instance, if both functions $G^a$ and $G^b$ in \eqref{forcesl+}  are  null, then the limit problem  \eqref{Pell+} reduces to
\vspace{-1mm}\begin{equation*}
\begin{aligned}
\min_{(\psi^a,\psi^b) \in\Phi_{\ell_+}^p  }
\bigg\{  &\bar
a\int_0^L {\mathcal C }W_0(\nabla_3\psi^a)\,\dx_3+ \ell\int_{\omega^b}
{\mathcal Q}W_1(\nabla_\alpha \psi^b)\,\dx_\alpha- \int_0^L
\big(
\bar f^a\cdot\psi^a+ \bar g^a\cdot\psi^a \big)\,\dx_3\\&  -\ell \int_{\omega^b}
\big(\bar f^b\cdot\psi^b +
(g^{b,+}-g^{b,-})\cdot\psi^{b}\big)
\,\dx_\alpha +\bar a\, \hat G^b(0_\alpha)\cdot
\psi^a(0_3) \bigg\};
\end{aligned}
\end{equation*}
setting further \(\hat G^b=0\), we recover  \cite[Theorem
5.1]{GaZa07} as a particular case. 

\subsection{Forces in divergence form}\label{Subs:divform}

Here, following a suggestion by Fran\c{c}ois Murat after part of this work was completed, we discuss the case where the system of applied forces to the multi-structure
is in a divergence form  in the spirit of \cite{GMMMS02, GMMMS07, MuSi99, Musi00} (see also \cite[Theorem~6.2]{FMP12}), allowing
for less regular volume and surface density terms.
Namely, in place of the \textit{classical} total energy in 
\eqref{energystandardforces}, we could  consider instead the total energy \vspace{-1mm}
\begin{equation}
 \label{energydivergenceforces}
 \begin{aligned}
 \int_{\Omega_\varepsilon}
 W(\nabla\tilde\psi)\,\d\tilde x - \int_{\Omega_\varepsilon}
 \tilde{H}_\epsi:\nabla \tilde\psi\,\d\tilde x,
 \end{aligned}
 \end{equation}
 where $\tilde{H}_\epsi \in L^q(\Omega_\epsi;\mathbb R^{3\times 3})$.
Note that given  \(\tilde f_\epsi\in L^q(\Omega_\varepsilon;{\mathbb
R}^3)\) and \(\tilde g_\epsi\in L^q(S_\varepsilon;{\mathbb
R}^3) \) such that
\begin{equation}
\label{eq:compatibility}
\begin{aligned}
\int_{\Omega_\varepsilon}
\tilde
f_\epsi\,\d\tilde x + \int_{S_\varepsilon}
\tilde g_\epsi\,\d{\HH}^2(\tilde x) =0,
\end{aligned}
\end{equation}
 we can find $\tilde{H}_\epsi \in L^q(\Omega_\epsi;\mathbb R^{3\times 3})$ satisfying
\begin{equation}\label{div=classical}
\begin{aligned}
 \int_{\Omega_\varepsilon}
 \tilde{H}_\epsi:\nabla \tilde\theta\,\d\tilde x = \int_{\Omega_\varepsilon}
\tilde
f_\epsi\cdot\tilde\theta\,\d\tilde x + \int_{S_\varepsilon}
\tilde g_\epsi\cdot\tilde\theta\,\d{\HH}^2(\tilde x)
\end{aligned}
\end{equation}
for all \(\tilde \theta \in W^{1,p}(\Omega_\varepsilon;{\mathbb
R}^3)\) (see \cite{GiTr01}). Thus, under the compatibility condition 
\eqref{eq:compatibility}, the classical formulation  
\eqref{energystandardforces} can be seen as a particular case of 
\eqref{energydivergenceforces}.
Conversely, if $\tilde{H}_\epsi \in L^q(\Omega_\epsi;\mathbb R^{3\times 3})$ is somewhat more regular with \(\tilde H_\epsi = 0\) on \(\Gamma^\epsi\) in the sense of traces, then formula \eqref{div=classical} holds true for %
\begin{equation}\label{eq:fromHreg}
\begin{aligned}
\tilde f_\epsi :=-(\div \tilde H_\epsi^1, \div \tilde H_\epsi^2, \div \tilde H_\epsi^3) \enspace \hbox{and} \enspace \tilde g_\epsi := \tilde H_\epsi \nu_\epsi,
\end{aligned}
\end{equation}
where \( \tilde H_\epsi^j\) stands for the \(jth\) row of \(\tilde H_\epsi\) and \(\nu_\epsi\) for the unit outer normal to \(S_\epsi\).
If  \(\tilde f_\epsi\) and \(\tilde g_\epsi\) given by \eqref{eq:fromHreg} belong to \(L^q(\Omega_\varepsilon;{\mathbb
R}^3)\) and \(L^q(S_\varepsilon;{\mathbb
R}^3)\), respectively, then   
\eqref{energydivergenceforces} can be seen as a particular case of 
\eqref{energystandardforces}. Note that if  \(\tilde H_\epsi \not=0\)
 on \(\Gamma^\epsi\), then we obtain the additional term 
 \(\int_{\Gamma^\epsi} (\tilde H_\epsi
\nu_\epsi)\cdot \tilde\theta \, \d{\HH}^2(\tilde x)\)  
in \eqref{div=classical}; however, using the 
 deformation condition \(\tilde \ffi_{0,\epsi}\in W^{1,p}(\Omega_\epsi;\RR^3)\)
imposed on \(\Gamma^\epsi\) (see \eqref{deftildePhiepsi}), this additional term can be easily handled.

Next, we elaborate on how to  reproduce our analysis in the previous sections but starting from    
\eqref{energydivergenceforces}. Under  mild hypotheses on \(\tilde H_\epsi\), we obtain in the limit a model incorporating bending-torsion moments; specifying further  \(\tilde H_\epsi\),  we obtain a limit model that, in addition to bending-torsion moments,  incorporates certain  body and surface forces. This limit model also contains
a term of the type \(c\cdot\psi^a(0_3)\)
for a certain constant \(c\); however,
this constant \(c\)   has, in general, a distinct physical interpretation from the constant   \(\bar a \hat G^b(0_\alpha)
\) in \eqref{Eell+}.   

Assume that $\tilde{H}_\epsi \in L^q(\Omega_\epsi;\mathbb R^{3\times 3})$. As in \eqref{forcesepsi}, we start by defining
\begin{equation*}
\begin{aligned}
&H^a_\epsi(x):=\tilde H_\epsi (r_\varepsilon
x_\alpha,
x_3)  \text{ for $x=(x_\alpha,x_3)\in\Omega^a$}
\, \hbox{
and }\,
 H^b_\epsi(x):=\tilde H_\epsi (
x_\alpha, h_\varepsilon
x_3) \text{ for $x=(x_\alpha,x_3)\in\Omega^b$}. \end{aligned}
\end{equation*}
Then, proceeding as in the Introduction, we are led to the re-scaled energy
\begin{equation}\label{eq:hatEdiv}
\begin{aligned}
F_\epsi^a (\psi^a) - \int_{\Omega^a} H_\epsi^a :({r_\varepsilon^{-1}}\nabla_\alpha
\psi^a|\nabla_3\psi^a)\, \dx 
+ \frac{h_\varepsilon}{r_\varepsilon^2} F_\epsi^b(\psi^b) - \frac{h_\varepsilon}{r_\varepsilon^2}\int_{\Omega^b}
H_\epsi^b :(
\nabla_\alpha \psi^b|{h_\varepsilon^{-1}}\nabla_3\psi^b)\,\dx, \end{aligned}
\end{equation}
where \(F^a_\epsi\) and \(F^b_\epsi\) are given by \eqref{Fabepsi}.

Let $(\psi^a_\epsi)_{\epsi>0}\subset W^{1,p}(\Omega^a;\RR^3)$
and $(\psi^b_\epsi)_{\epsi>0}\subset W^{1,p}(\Omega^b;\RR^3)$  be as in Lemma~\ref{lem:l21GGLM} and assume  that
\begin{equation}
\label{eq:convHepsi}
\begin{cases}
H^a_\varepsilon \to H^a \hbox{ in }L^q(\Omega^a;\mathbb R^{3\times
3})\enspace \hbox{and} \enspace \frac{h_\varepsilon}{r_\varepsilon^2}H^b_\varepsilon \to H^b \hbox{ in }L^q(\Omega^b;\mathbb R^{3\times
3}) & \hbox{if } \ell\in \RR^+ \hbox{ or } \ell=\infty,\\
\frac{r_\varepsilon^2}{h_\varepsilon} H^a_\varepsilon \to H^a \hbox{ in }L^q(\Omega^a;\mathbb R^{3\times
3})\enspace \hbox{and} \enspace H^b_\varepsilon
\to H^b \hbox{ in }L^q(\Omega^b;\mathbb R^{3\times
3}) & \hbox{if } \ell=0
\end{cases}
\end{equation}
for some \(H^a\in L^q(\Omega^a;\mathbb R^{3\times
3}) \) and \(H^b\in L^q(\Omega^b;\mathbb R^{3\times
3})\). Then, for  \(\ell\in \RR^+\) or \(\ell=\infty\),
we have%
\begin{equation}\label{eq:limwithH+}
\begin{aligned}
&\lim_{\epsi\to0^+} \bigg(\int_{\Omega^a} H_\epsi^a :({r_\varepsilon^{-1}}\nabla_\alpha
\psi^a_\varepsilon|\nabla_3\psi^a_\varepsilon)\, \dx + 
\frac{h_\varepsilon}{r_\varepsilon^2}\int_{\Omega^b}
H_\epsi^b :(
\nabla_\alpha \psi^b_\varepsilon|{h_\varepsilon^{-1}}\nabla_3\psi^b_\varepsilon)\,\dx \bigg) \\
&\quad = \int_{\Omega^a} H^a(x) : (\bcal^a(x)|\grad_3\psi^a(x_3))\,\dx + \int_{\Omega^b}
H^b(x) :(
\nabla_\alpha \psi^b(x_\alpha)|\bcal^b(x))\,\dx;
\end{aligned}
\end{equation}
similarly, for \(\ell=0\),
we have%
\begin{equation}\label{eq:limwithH0}
\begin{aligned}
&\lim_{\epsi\to0^+} \bigg(\frac{r_\varepsilon^2}{h_\varepsilon}\int_{\Omega^a} H_\epsi^a :({r_\varepsilon^{-1}}\nabla_\alpha
\psi^a_\varepsilon|\nabla_3\psi^a_\varepsilon)\, \dx + 
\int_{\Omega^b}
H_\epsi^b :(
\nabla_\alpha \psi^b_\varepsilon|{h_\varepsilon^{-1}}\nabla_3\psi^b_\varepsilon)\,\dx
\bigg) \\
&\quad = \int_{\Omega^a} H^a(x) :
(\bcal^a(x)|\grad_3\psi^a(x_3))\,\dx
+ \int_{\Omega^b}
H^b(x) :(
\nabla_\alpha \psi^b(x_\alpha)|\bcal^b(x))\,\dx.
\end{aligned}
\end{equation}
Moreover, if
\vspace{-1.5mm}%
\begin{equation}
\label{eq:extrahypH}
\begin{aligned}
H^a_\alpha\equiv H^a_\alpha(x_3) \quad \hbox{and}
\quad H^b_3\equiv
H^b_3(x_\alpha),
\end{aligned}
\end{equation}
where, as before, \(H^a=(H^a_\alpha|H^a_3)\)
and \(H^b=(H^b_\alpha|H^b_3)\), then the sum on the right-hand side of
\eqref{eq:limwithH+} and \eqref{eq:limwithH0}
equals
\vspace{-3mm}%
\begin{equation}\label{eq:hatL}
\begin{aligned}
\hat L((\bar\bcal^a,\psi^a),
(\psi^b,\bar
\bcal^b))&:= \int_0^L H^a_\alpha(x_3) : \bar\bcal^a(x_3)\,\dx_3 +
 \int_{\Omega^a} H^a_3(x)\cdot\grad_3\psi^a(x_3)\,\dx\\ &\qquad+ \int_{\Omega^b}
H^b_\alpha(x) :
\nabla_\alpha \psi^b(x_\alpha)\,\dx
+ \int_{\omega^b}H^b_3(x_\alpha) \cdot\bar\bcal^b(x_\alpha)\,\dx_\alpha. \end{aligned}
\end{equation}

Under the assumption in \eqref{eq:extrahypH}, results analogous  to Theorems~\ref{thm:ellr+}, \ref{thm:elli}, and \ref{thm:ellz} hold  replacing the terms involving the applied forces by \(\hat L\) given by   \eqref{eq:hatL}.
For instance, if \(\ell\in\RR^+\), we have the following result.

\begin{theorem}[\(\ell\in\RR^+\)]\label{thm:ellr+div}
Let \(W:\R^{3\times 3} \to \R\) be a
Borel function
satisfying
\eqref{pgrowth} and let \((\psi^a_\epsi,\psi^b_\epsi)_{\epsi>0}\)
be a diagonal
infimizing sequence of the sequence
of problems \eqref{Pepsi} with \(E^a_\epsi(\psi^a)
+ E^b_\epsi(\psi^b)\) replaced by the
functional  in \eqref{eq:hatEdiv},
where  \(\ell\)  given by \eqref{ell}
is such that  \(\ell\in\RR^+\),    
\((\ffi^a_{0,\epsi},\ffi^b_{0,\epsi})_{\epsi>0}\)
satisfies
\eqref{bca}--\eqref{bcb} and \eqref{junction},
and \eqref{eq:convHepsi} holds with
 \eqref{eq:extrahypH}.
 Then, the sequences
\((\bar \bcal^a_\epsi,\psi^a_\epsi)_{\epsi>0}\)
and \((\psi^b_\epsi,\bar \bcal^b_\epsi)_{\epsi>0}\),
where \(\bar\bcal^a_\epsi:={r_{\varepsilon}^{-1}}
\int_{\omega^a}\nabla_\alpha
\psi^a_\epsi\, \dx_\alpha\) and \(\bar\bcal^b_\epsi:=
{h_\epsi^{-1}}\int_{-1}^{0}\nabla_3
\psi^b\,\dx_3 \), are sequentially,
weakly compact in
\(  L^p((0,L);\R^{3 \times 2}) \times
W^{1,p}(\Omega^a;\R^3) \) and \(W^{1,p}(\Omega^b;\R^3)\times
L^p(\omega^b;\R^{3})\),
respectively. If \((\bar\bcal^a, \psi^a)\)
and \(( \psi^b,\bar \bcal^b)\) are corresponding
accumulation points, then \((\psi^a,\psi^b)\in
\Phi_{\ell_+}^p\) and they  solve the
minimization problem
\vspace{-1mm}%
\begin{equation*}
\begin{aligned}
\min
\Big\{\hat E_{\ell_+}((\bar\bcal^a, \psi^a),
( \psi^b,\bar \bcal^b))\!: \, &\,(\psi^a,\psi^b) \in\Phi_{\ell_+}^p,\, (\bar\bcal^a,\bar
\bcal^b) \in L^p((0,L);\R^{3 \times
2}) \times L^p(\omega^b;\R^{3}) \Big\},
\end{aligned}
\end{equation*}
where, recalling \(\hat L\) in \eqref{eq:hatL},
\vspace{-1mm}\begin{equation*}
\begin{aligned}
\hat E_{\ell_+}((\bar\bcal^a,\psi^a), (\psi^b,\bar
\bcal^b)):= &\, \bar
a\int_0^L \ce({\bar a^{-1}}\bar\bcal^a|\nabla_3\psi^a)\,\dx_3+
\ell\int_{\omega^b}
 \qce(\nabla_\alpha \psi^b|\bar\bcal^b)\,\dx_\alpha
 + \hat L((\bar\bcal^a,\psi^a),
(\psi^b,\bar
\bcal^b)).
\end{aligned}
\end{equation*}
\end{theorem}

\begin{remark}\label{rmk:ondivcase}
(i) If the assumption in \eqref{eq:extrahypH} does not hold, then the previous analysis
must be performed considering the full
3D dependence of \({r_{\varepsilon}^{-1}}
\nabla_\alpha
\psi^a_\epsi\)  and \({h_{\varepsilon}^{-1}}
\nabla_3
\psi^b_\epsi\),
as in \cite{BFM09}. This may lead us
into nonlocal limit models (see \cite[Remark~2.1]{BFM09})
and will be the object of a future work.
 (ii) Next, we analyse the relationship
between \(\hat L\) and  the terms involving the
applied forces in the limit functional \eqref{Eell+}. We start by observing that \(\int_0^L
\Bcal^a(x_3) : (\bar\bcal^a(x_3)|0)\,\dx_3
= \int_0^L
\Bcal^a_\alpha(x_3):\bar\bcal^a(x_3)\,\dx_3\).
Hence, \(-H^a_\alpha\) plays the role
of \(\Bcal^a_\alpha\), which are the relevant
components of \(\Bcal^a\) in \eqref{Eell+}.
If \(H^a_3(x_\alpha,\cdot)\in W^{1,q}(0,L;\RR^3)\)
for \aev~\(x_\alpha\in\omega^a\), then
\vspace{-1mm}
\begin{equation*}
\begin{aligned}
 \int_{\Omega^a} H^a_3(x)\cdot\grad_3\psi^a(x_3)\,\dx
&=\int_{\omega^a} H^a_3(x_\alpha,L)\,\dx_\alpha\cdot\ffi^a_0(L)
- \int_{\omega^a} H^a_3(x_\alpha,0)\,\dx_\alpha\cdot\psi^a(0)\\
&\qquad - \int_0^L \bigg(\int_{\omega^a} \grad_3H^a_3(x_\alpha,x_3)\,\dx_\alpha \bigg)\cdot
\psi^a(x_3)\,\dx_3.
  \end{aligned}
\end{equation*}
The first term on the right-hand side
of the previous identity is one of the additional
terms mentioned above when
commenting the case in which   \(\tilde H_\epsi \not=0\)
 on \(\Gamma^\epsi\). Regarding the
second term, we observe that from the mathematical
 point of view, \(- \int_{\omega^a} H^a_3(x_\alpha,0)\,\dx_\alpha\) plays
the role of \(\bar a \hat G^b(0_\alpha)\)
in \eqref{Eell+} but, in general, these two constants are of distinct
physical natures. Finally, concerning
the third term in the   identity above,
we observe that \(\int_{\omega^a}
\grad_3H^a_3(x_\alpha,\cdot)\,\dx_\alpha\) plays
the role of \(\bar f^a + \bar g^a\)
in \eqref{Eell+}.

Next, we focus on the last two integral
terms in \eqref{eq:hatL}. We first
note that \(- H^b_3\) plays
the role of \(\ell  G^b\) in \eqref{Eell+}. Moreover, if  
\(H^b_\alpha(\cdot, x_3)\in W^{1,q}(\omega^b;\RR^3)\)
for \aev~\(x_3\in(-1,0)\), then \vspace{-.mm}
\begin{equation*}
\begin{aligned}
 \int_{\Omega^b}
H^b_\alpha(x) :
\nabla_\alpha \psi^b(x_\alpha)\,\dx
&=\int_{\partial \omega^b}
\bigg(\int_{-1}^0H^b_\alpha(x_\alpha,x_3)\,\dx_3\bigg):\big(\ffi^b_0(x_\alpha)\otimes
\nu_\alpha(x_\alpha)\big)\,\d\mathcal{H}^1(x_\alpha)  \\ &\qquad - \int_{\omega^b}
\bigg(\int_{-1}^0\big(
 \grad_1 H^b_1(x_\alpha,x_3) + \grad_2 H^b_2(x_\alpha,x_3)\big)\,
 \dx_3\bigg)\cdot \psi^b(x_\alpha)\,\dx_\alpha.
\end{aligned}
\end{equation*}
As before, the first term on the right-hand side
of the previous identity is one of the
additional
terms mentioned above
when
commenting the case in which   \(\tilde
H_\epsi \not=0\)
 on \(\Gamma^\epsi\), while \(\int_{-1}^0\big(
 \grad_1 H^b_1(\cdot,x_3) + \grad_2
H^b_2(\cdot,x_3)\big)\,
 \dx_3\) plays the role of \(\ell\,(\bar f^b
 + g^{b,+} - g^{b,-})\) in \eqref{Eell+}.
 \end{remark}

We finish by observing that in light
of the previous analysis, and as in
\cite{GMMMS07}, a more complete model is obtained
starting from an energy that simultaneously
contains the classical applied forces
as in 
\eqref{energystandardforces} and forces in divergence form as in \eqref{energydivergenceforces}. The corresponding limit models can be easily deduced from our previous arguments.

\section*{Acknowledgements}
The authors thank  Professors F. Murat for having suggested to target the model proposed in Section~6.2 and  R. Alicandro and M.G. Mora for insights on the double-well case. 
E. Zappale is a member of GNAMPA- INdAM whose support is gratefully acknowledged.

\bibliographystyle{plain}

\bibliography{Junctions}

\begin{thebibliography}{10}

\bibitem{ABP91}
E.~Acerbi, G.~Buttazzo, and D.~Percivale.
\newblock A variational definition of the strain energy for an elastic string.
\newblock {\em J. Elasticity}, 25(2):137--148, 1991.

\bibitem{AcFu84}
E.~Acerbi and N.~Fusco.
\newblock Semicontinuity problems in the calculus of variations.
\newblock {\em Arch. Rational Mech. Anal.}, 86(2):125--145, 1984.

\bibitem{BZZ08}
J.-F. Babadjian, E.~Zappale, and H.~Zorgati.
\newblock Dimensional reduction for energies with linear growth involving the
  bending moment.
\newblock {\em J. Math. Pures Appl. (9)}, 90(6):520--549, 2008.

\bibitem{BlGr13}
D.~Blanchard and G.~Griso.
\newblock Junction between a plate and a rod of comparable thickness in
  nonlinear elasticity.
\newblock {\em J. Elasticity}, 112(2):79--109, 2013.

\bibitem{BoFo04}
M.~Bocea and I.~Fonseca.
\newblock A {Y}oung measure approach to a nonlinear membrane model involving
  the bending moment.
\newblock {\em Proc. Roy. Soc. Edinburgh Sect. A}, 134(5):845--883, 2004.

\bibitem{BFM03}
G.~Bouchitt{\'e}, I.~Fonseca, and M.L. Mascarenhas.
\newblock Bending moment in membrane theory.
\newblock {\em J. Elasticity}, 73(1-3):75--99 (2004), 2003.

\bibitem{BFM09}
G.~Bouchitt{\'e}, I.~Fonseca, and M.L. Mascarenhas.
\newblock The {C}osserat vector in membrane theory: a variational approach.
\newblock {\em J. Convex Anal.}, 16(2):351--365, 2009.

\bibitem{BCN17}
R.~Bunoiu, G.~Cardone, and S.A. Nazarov.
\newblock Scalar problems in junctions of rods and a plate. ii. self-adjoint
  extensions and simulation models.
\newblock 10 2017.

\bibitem{CMMO17}
G.~Carita, J.~Matias, M.~Morandotti, and D.R. Owen.
\newblock Dimension reduction in the context of structured deformations.
\newblock 2017.

\bibitem{CM}
N.~Chaudhuri and S.~M\"uller.
\newblock Rigidity estimate for two incompatible wells.
\newblock {\em Calc. Var. Partial Differential Equations}, 19(4):379--390,
  2004.

\bibitem{Cia88}
P.G. Ciarlet.
\newblock {\em Mathematical Elasticity: Three-dimensional elasticity},
  volume~I.
\newblock North-Holland, Amsterdam, 1988.

\bibitem{Cia90}
P.G. Ciarlet.
\newblock {\em Plates and junctions in elastic multi-structures}, volume~14 of
  {\em Recherches en Math{\'e}matiques Appliqu{\'e}es [Research in Applied
  Mathematics]}.
\newblock Masson, Paris; Springer-Verlag, Berlin, 1990.
\newblock An asymptotic analysis.

\bibitem{Cia97}
P.G. Ciarlet.
\newblock {\em Theory of Plates. Mathematical Elasticity.}, volume~II.
\newblock North-Holland, Amsterdam, 1997.

\bibitem{CS06}
S.~Conti and B.~Schweizer.
\newblock Rigidity and gamma convergence for solid-solid phase transitions with
  {SO}(2) invariance.
\newblock {\em Comm. Pure Appl. Math.}, 59(6):830--868, 2006.

\bibitem{LeD91}
H.~Le Dret.
\newblock {\em Probl{\`e}mes variationnels dans les multi-domaines}, volume~19
  of {\em Recherches en Math{\'e}matiques Appliqu{\'e}es [Research in Applied
  Mathematics]}.
\newblock Masson, Paris, 1991.
\newblock Mod{\'e}lisation des jonctions et applications. [Modeling of
  junctions and applications].

\bibitem{LDR95}
H.~Le Dret and A.~Raoult.
\newblock The nonlinear membrane model as variational limit of nonlinear
  three-dimensional elasticity.
\newblock {\em J. Math. Pures Appl. (9)}, 74(6):549--578, 1995.

\bibitem{LDR00}
H.~Le Dret and A.~Raoult.
\newblock Variational convergence for nonlinear shell models with directors and
  related semicontinuity and relaxation results.
\newblock {\em Arch. Ration. Mech. Anal.}, 154(2):101--134, 2000.

\bibitem{RiMSc}
R.~Ferreira.
\newblock {R}edu\c{c}\~ao {D}imensional em {E}lasticidade {N}\~ao {L}inear
  {A}trav\'es da {$\Gamma$}-{C}onverg\^encia ({D}imensional {R}eduction in
  {N}on-linear {E}lasticity via {$\Gamma$}-{C}onvergence).
\newblock Master's thesis, Faculty of Sciences of the University of Lisbon
  (FCUL), 2006.

\bibitem{FKP94}
I.~Fonseca, D.~Kinderlehrer, and P.~Pedregal.
\newblock Energy functionals depending on elastic strain and chemical
  composition.
\newblock {\em Calc. Var. Partial Differential Equations}, 2(3):283--313, 1994.

\bibitem{FL07}
I.~Fonseca and G.~Leoni.
\newblock {\em Modern methods in the calculus of variations: {$L\sp p$}
  spaces}.
\newblock Springer Monographs in Mathematics. Springer, New York, 2007.

\bibitem{FMP98}
I.~Fonseca, S.~M{{\"u}}ller, and P.~Pedregal.
\newblock Analysis of concentration and oscillation effects generated by
  gradients.
\newblock {\em SIAM J. Math. Anal.}, 29(3):736--756 (electronic), 1998.

\bibitem{FRS93}
D.D. Fox, A.~Raoult, and J.C. Simo.
\newblock A justification of nonlinear properly invariant plate theories.
\newblock {\em Arch. Rational Mech. Anal.}, 124(2):157--199, 1993.

\bibitem{FMP12}
L.~Freddi, M.G. Mora, and R.~Paroni.
\newblock Nonlinear thin-walled beams with a rectangular cross-section---{P}art
  {I}.
\newblock {\em Math. Models Methods Appl. Sci.}, 22(3):1150016, 34, 2012.

\bibitem{FMP13}
L.~Freddi, M.G. Mora, and R.~Paroni.
\newblock Nonlinear thin-walled beams with a rectangular cross-section---{P}art
  {II}.
\newblock {\em Math. Models Methods Appl. Sci.}, 23(4):743--775, 2013.

\bibitem{FJM02}
G.~Friesecke, R.D. James, and S.~M{\"u}ller.
\newblock A theorem on geometric rigidity and the derivation of nonlinear plate
  theory from three-dimensional elasticity.
\newblock {\em Comm. Pure Appl. Math.}, 55(11):1461--1506, 2002.

\bibitem{FJM06}
G.~Friesecke, R.D. James, and S.~M{\"u}ller.
\newblock A hierarchy of plate models derived from nonlinear elasticity by
  gamma-convergence.
\newblock {\em Arch. Ration. Mech. Anal.}, 180(2):183--236, 2006.

\bibitem{GaZa07}
G.~Gargiulo and E.~Zappale.
\newblock A remark on the junction in a thin multi-domain: the non convex case.
\newblock {\em NoDEA Nonlinear Differential Equations Appl.}, 14(5-6):699--728,
  2007.

\bibitem{GGLM02b}
A.~Gaudiello, B.~Gustafsson, C.~Lefter, and J.~Mossino.
\newblock Asymptotic analysis for monotone quasilinear problems in thin
  multidomains.
\newblock {\em Differential Integral Equations}, 15(5):623--640, 2002.

\bibitem{GGLM02}
A.~Gaudiello, B.~Gustafsson, C.~Lefter, and J.~Mossino.
\newblock Asymptotic analysis of a class of minimization problems in a thin
  multidomain.
\newblock {\em Calc. Var. Partial Differential Equations}, 15(2):181--201,
  2002.

\bibitem{GMMMS02}
A.~Gaudiello, R.~Monneau, J.~Mossino, F.~Murat, and A.~Sili.
\newblock On the junction of elastic plates and beams.
\newblock {\em C. R. Math. Acad. Sci. Paris}, 335(8):717--722, 2002.

\bibitem{GMMMS07}
A.~Gaudiello, R.~Monneau, J.~Mossino, F.~Murat, and A.~Sili.
\newblock Junction of elastic plates and beams.
\newblock {\em ESAIM Control Optim. Calc. Var.}, 13(3):419--457, 2007.

\bibitem{GaSi11}
A.~Gaudiello and A.~Sili.
\newblock Asymptotic analysis of the eigenvalues of an elliptic problem in an
  anisotropic thin multidomain.
\newblock {\em Proc. Roy. Soc. Edinburgh Sect. A}, 141(4):739--754, 2011.

\bibitem{GaZa06}
A.~Gaudiello and E.~Zappale.
\newblock Junction in a thin multidomain for a fourth order problem.
\newblock {\em Math. Models Methods Appl. Sci.}, 16(12):1887--1918, 2006.

\bibitem{GiTr01}
D.~Gilbarg and N.S. Trudinger.
\newblock {\em Elliptic partial differential equations of second order}.
\newblock Classics in Mathematics. Springer-Verlag, Berlin, 2001.
\newblock Reprint of the 1998 edition.

\bibitem{Gru93}
I.~Gruais.
\newblock Modeling of the junction between a plate and a rod in nonlinear
  elasticity.
\newblock {\em Asymptotic Anal.}, 7(3):179--194, 1993.

\bibitem{HJM93}
J.~Heinonen, T.~Kilpel\"ainen, and O.~Martio.
\newblock {\em Nonlinear potential theory of degenerate elliptic equations}.
\newblock Oxford Mathematical Monographs. The Clarendon Press, Oxford
  University Press, New York, 1993.
\newblock Oxford Science Publications.

\bibitem{LaNg14}
W.~Laskowski and H.T. Nguyen.
\newblock Effective energy integral functionals for thin films with bending
  moment in the {O}rlicz-{S}obolev space setting.
\newblock In {\em Function spaces {X}}, volume 102 of {\em Banach Center
  Publ.}, pages 143--167. Polish Acad. Sci. Inst. Math., Warsaw, 2014.

\bibitem{LaNg16}
W.~Laskowski and H.T. Nguyen.
\newblock Effective energy integral functionals for thin films with three
  dimensional bending moment in the {O}rlicz-{S}obolev space setting.
\newblock {\em Discuss. Math. Differ. Incl. Control Optim.}, 36(1):7--31, 2016.

\bibitem{DLS}
C.~De Lellis and L.Jr. Sz\'ekelyhidi.
\newblock Simple proof of two-well rigidity.
\newblock {\em C. R. Math. Acad. Sci. Paris}, 343(5):367--370, 2006.

\bibitem{DM93}
G.~Dal Maso.
\newblock {\em An introduction to {$\Gamma$}-convergence}.
\newblock Progress in Nonlinear Differential Equations and their Applications,
  8. Birkh{\"a}user Boston, Inc., Boston, MA, 1993.

\bibitem{Mat92}
J.P. Matos.
\newblock Young measures and the absence of fine microstructures in a class of
  phase transitions.
\newblock {\em European J. Appl. Math.}, 3(1):31--54, 1992.

\bibitem{MoMu03}
M.G. Mora and S.~M{\"u}ller.
\newblock Derivation of the nonlinear bending-torsion theory for inextensible
  rods by {$\Gamma$}-convergence.
\newblock {\em Calc. Var. Partial Differential Equations}, 18(3):287--305,
  2003.

\bibitem{MoMu07}
M.G. Mora and S.~M\"uller.
\newblock Derivation of a rod theory for multiphase materials.
\newblock {\em Calc. Var. Partial Differential Equations}, 28(2):161--178,
  2007.

\bibitem{MuSi99}
F.~Murat and A.~Sili.
\newblock Comportement asymptotique des solutions du syst{\`e}me de
  l'{\'e}lasticit{\'e} lin{\'e}aris{\'e}e anisotrope h{\'e}t{\'e}rog{\`e}ne
  dans des cylindres minces.
\newblock {\em C. R. Acad. Sci. Paris S{\'e}r. I Math.}, 328(2):179--184, 1999.

\bibitem{Musi00}
F.~Murat and A.~Sili.
\newblock Effets non locaux dans le passage 3d--1d en {\'e}lasticit{\'e}
  lin{\'e}aris{\'e}e anisotrope h{\'e}t{\'e}rog{\`e}ne.
\newblock {\em C. R. Acad. Sci. Paris S{\'e}r. I Math.}, 330(8):745--750, 2000.

\bibitem{Sca09}
L.~Scardia.
\newblock Asymptotic models for curved rods derived from nonlinear elasticity
  by {$\Gamma$}-convergence.
\newblock {\em Proc. Roy. Soc. Edinburgh Sect. A}, 139(5):1037--1070, 2009.

\bibitem{TrVi96}
L.~Trabucho and J.M. Viano.
\newblock Mathematical modelling of rods.
\newblock In {\em Handbook of numerical analysis, {V}ol.\ {IV}}, Handb. Numer.
  Anal., IV, pages 487--974. North-Holland, Amsterdam, 1996.

\bibitem{Sve93}
V.~\v{S}ver\'ak.
\newblock On the problem of two wells.
\newblock In {\em Microstructure and phase transition}, volume~54 of {\em IMA
  Vol. Math. Appl.}, pages 183--189. Springer, New York, 1993.

\end{thebibliography}

\end{document}